\documentclass[reqno]{amsart}
\usepackage[margin=1in]{geometry}  
\usepackage{graphicx}              
\usepackage{amsmath, amssymb, amsthm, epsfig}               
\usepackage{enumerate}
\usepackage{amsfonts}              
\usepackage{amsthm}                
\usepackage{amsthm}
\usepackage{enumerate}
\theoremstyle{plain}
\usepackage{color}
\usepackage{hyperref}
\hypersetup{colorlinks,citecolor=blue}
\usepackage{enumitem}
\usepackage{cite}
\newtheorem{theorem}{Theorem}
\newtheorem{remark}[theorem]{Remark}
\newtheorem{lemma}[theorem]{Lemma}

\newtheorem{corollary}[theorem]{Corollary}
\newtheorem{definition}[theorem]{Definition}

\newcommand{\weak}{\rightharpoonup}

\renewcommand{\l}{\left}
\renewcommand{\r}{\right}

\numberwithin{theorem}{section}
\numberwithin{equation}{section}

scaled\magstep2

\makeatletter
\def\section{\@startsection {section}{1}{\z@}{-3.5ex plus -1ex minus
        -.2ex}{2.3ex plus .2ex}{\large\bf}}
\makeatother \numberwithin{equation}{section}
\DeclareMathOperator{\dist}{dist} 
\begin{document}
\title[Global Existence and Finite-Time Blow-Up of Solutions]{Global Existence and Finite-Time Blow-Up of Solutions for Parabolic Equations Involving the Fractional Musielak $g_{x,y}$-Laplacian}
\author[R. Arora]{Rakesh Arora}
\address[R. Arora]{ Department of Mathematical Sciences, Indian Institute of Technology Varanasi (IIT-BHU), Uttar Pradesh 221005, India}
\email{rakesh.mat@iitbhu.ac.in, arora.npde@gmail.com}

\author[A. Bahrouni]{Anouar Bahrouni}
\address[A. Bahrouni]{
Mathematics Department, Faculty of Sciences, University of Monastir,
5019 Monastir, Tunisia}
\email{Anouar.Bahrouni@fsm.rnu.tn; bahrounianouar@yahoo.fr}

\author[N. K. Maurya]{Nitin Kumar Maurya}
\address[N. K. Maurya]{Department of Mathematical Sciences, Indian Institute of Technology Varanasi (IIT-BHU), Uttar Pradesh 221005, India}
\email{nitinkumarmaurya.rs.mat23@itbhu.ac.in }

\maketitle
\begin{abstract}
In this work, we study the parabolic fractional Musielak $g_{x,y}$-Laplacian equation:

\begin{equation*}
\left\{
\begin{aligned}
    u_{t} + (-\Delta)_{{g}_{x,y}}^{s} u &= f(x,u), && \text{in } \Omega \times (0, \infty), \\
    u &= 0, && \text{on } \mathbb{R}^N \setminus \Omega \times (0, \infty), \\
    u(x,0) &= u_0(x), && \text{in } \Omega,
\end{aligned}
\right.
\end{equation*}
where $(-\Delta)_{{g}_{x,y}}^{s}$ denotes the fractional Musielak $g_{x,y}$-Laplacian, and $f$ is a Carathéodory function satisfying subcritical growth conditions. Using the modified potential well method and Galerkin's method, we establish results on the local and global existence of weak and strong solutions, as well as finite-time blow-up, depending on the initial energy level (low, critical, or high). Moreover, we explore a class of nonlocal operators to highlight the broad applicability of our approach.  

This study contributes to the developing theory of fractional Musielak-Sobolev spaces, a field that has received limited attention in the literature. To our knowledge, this is the first work addressing the parabolic fractional $g_{x,y}$-Laplacian equation.
\end{abstract}
\tableofcontents
\section{Introduction}
In this article, we study the following parabolic equation involving the new fractional Musielak $g_{x,y}$--Laplacian
\begin{equation} \label{main:problem}
        \left\{
    \begin{aligned}
      u_{t}+(-\Delta)_{{g}_{x,y}}^{s} u &=  f(x,u), &&
 \text{in } \Omega \times (0, \infty),\\
 u &= 0, && \text{in}~  \mathbb{R}^N \setminus \Omega \times (0, \infty),\\
 u(x,0)&= u_0(x), && \text{in}~ \Omega,
\end{aligned}
    \right.
 \end{equation}
 where $\Omega\subseteq\mathbb{R}^N$, $N\geq sg^{-}$ is a bounded domain with smooth boundary and $f: \Omega \times \mathbb{R}\longrightarrow \mathbb{R}$ is a Carath\'eodory function satisfying subcritical growth conditions. For $s \in (0,1)$ and $g: \Omega \times \Omega \times [0,+\infty) \longrightarrow [0,+\infty)$, the operator $(-\Delta)^{s}_{g_{x,y}}$ is the new fractional Musielak ${g_{x,y}}$--Laplace operator defined by
 \begin{equation}\label{main:operator}
     (-\Delta)_{g_{x, y}}^{s} u(x):= \text{P.V.} \int_{\mathbb{R}^{N}} g_{x, y}\left(\frac{|u(x)-u(y)|}{|x-y|^{s}}\right) \frac{d y}{|x-y|^{N+s}}, \quad \text { for all } x \in \mathbb{R}^{N},
 \end{equation}
 where  $g_{x,y}(t) = g(x,y,t)$ and P.V. is a commonly used abbreviation for in the ``Principal Value" sense.\\

In recent years, the study of nonlinear equations involving the nonlocal operators e.g. the fractional Laplacian $(-\Delta)^s$, the $p$-fractional Laplacian $(-\Delta)^s_p$, $1<p< \infty$, has gained significant attention due to its rich analytic structure and wide-ranging applications in fields such as optimization, finance, anomalous diffusion, phase transitions, flame propagation, and minimal surfaces. The appropriate functional framework for such equations is provided by the fractional Sobolev spaces. For a comprehensive overview of the basic properties of these spaces and the operator $(-\Delta)^s$, as well as their applications to partial differential equations, we refer interested readers to \cite{Valdinoci-2012, Molica-2016} and the references therein.

Motivated by the above real-world applications, a few extensions of the fractional Sobolev space and the nonlocal operators have been introduced. Some of the extensions are the fractional Sobolev space with variable exponent in \cite{Kauf} and fractional Orlicz-Sobolev spaces in \cite{Bonder-Salort-2019} providing a bridge between fractional order theory, Orlicz-Sobolev theory. Since then, several foundational results have been proved, including embedding theorems, density results, and topological properties that allow for the applications of the variational approaches; see \cite{Alberico-Cianchi-Pick-Slavikova-2022, Alberico-Cianchi-Pick-Slavikova-2021-1, Alberico-Cianchi-Pick-Slavikova-2021-2,  Bahrouni-Radulescu-2018, Bonder-Salort-2019, Ho-Kim-2019, Azroul-Benkirane-Shimi-2019, Missaoui-Ounaies-2023, Silva-Carvalho-Albuquerque-Bahrouni-2021}.  

Recently, a more general functional space, the fractional Musielak-Orlicz-Sobolev space $W^{s, G_{x,y}}(\Omega)$ (see Section \ref{subsection:functions:spaces} below), was introduced in \cite{Azroul-Benkirane-Shimi-Srati-2020}. This space naturally generalizes both the fractional Sobolev space with variable exponent and the fractional Orlicz-Sobolev space. Moreover, the authors also defined the fractional Musielak $g_{x,y}$-Laplace operator $(-\Delta)^s_{g_{x,y}}$ for all $s \in (0,1)$ (see \eqref{main:operator}), which encompasses several particular cases, including the fractional operator with variable exponent, the fractional double-phase operator with variable exponent, and the anisotropic fractional $p$-Laplacian.


The study of nonlinear equations driven by the fractional Musielak $g_{x,y}$--Laplacian $(-\Delta)^s_{g_{x,y}}$ has found crucial applications in image processing, particularly in tasks such as denoising, inpainting, and edge detection. The fractional nature of this operator accounts for long-range interactions, enabling efficient smoothing while preserving fine details and sharp edges. Moreover, the variable growth conditions introduced by the Musielak function provide the flexibility to adapt to different image regions, allowing the model to handle diverse textures, noise levels, and transitions effectively. This adaptability makes the fractional Musielak $g_{x,y}$-Laplacian a powerful tool for image restoration and enhancement; for more details, see \cite{app}.

Very recently, the authors in \cite{M1, Azroul-Shimi-Srati-2023, Bahrouni-Missaoui-Ounaies-2024} established several abstract results in the framework of fractional Musielak-Sobolev spaces, including uniform convexity, the Radon-Riesz property with respect to the modular function, the $(S_+)$-property, a Brezis-Lieb type lemma for the modular function, various monotonicity results, as well as continuous and compact embedding theorems.
 As an application of above properties, they studied the existence of weak solutions to the following elliptic problem:
$$
\left\lbrace
\begin{array}{ll}
(-\Delta)^s_{g_{x,y}}u = f(x,u), & \text{in}\ \Omega, \\
u = 0, & \text{on}\ \mathbb{R}^N \setminus \Omega,
\end{array}
\right.
$$
where $N \geq 2$, $\Omega \subset \mathbb{R}^N$ is a bounded domain with a Lipschitz boundary, and $f : \Omega \times \mathbb{R} \longrightarrow \mathbb{R}$ is a Carathéodory function not necessarily satisfying the Ambrosetti–Rabinowitz condition. The study of elliptic problems involving the fractional Musielak $g_{x,y}$-Laplacian is very limited, as it is still an evolving area of research. Furthermore, the existing studies primarily focused on stationary problems.

To the best of our knowledge, there has been no investigation into parabolic equations related to the new fractional Musielak $g_{x,y}$-Laplacian. This work aims to bridge that gap by examining the global existence and finite-time blow-up of solutions for the parabolic equations featuring the new fractional $g_{x,y}$-Laplacian.

Concerning the parabolic equations, in the local case, \textit{i.e.}, when $s = 1$ and $g_{x,y}(t) = |t|^{p(x)-2}t$, equation \eqref{main:problem} goes back to the following problem:
\[
u_{t} - \operatorname{div}(|\nabla u|^{p(x)-2} \nabla u) = f(x,u).
\]
Although we do not aim to provide a comprehensive review of the extensive literature on solutions to these equations, it is worth emphasizing that some of the most intriguing aspects involve studying the existence and regularity of solutions, as well as their asymptotic behavior. Using the sub-differential approach and potential well-method, the authors in \cite{Akagi-Matsuura-2011, Nhan-Chuong-Truong-2020} established the local/global well-posedness of solutions for this equation. Moreover, they also investigated the large-time behavior of solutions. For additional results on global existence, uniqueness, regularity properties, and blow-up of solutions for the parabolic equation driven by the $p(x)$-Laplacian, we refer the reader to \cite{Antontsev-Shmarev-2009, Antontsev-Shmarev-2015, Arora-Shmarev-2021, Diening-Nagele-Ruzicka, Ant}. 
For a generalized function $g_{x,y}(\cdot)$ covering multi-phase problem and anisotropic problems, the authors in \cite{Chlebicka-Gwiazda-Goldstein-2019, Chlebicka-Gwiazda-Goldstein-2018, Arora-Shmarev-2023, Arora-Shmarev-2023-2, Arora-Shmarev-2023-3} investigate the existence, uniqueness, and qualitative properties of solutions.

Returning to the nonlocal case, \textit{i.e.}, when $s \in (0,1)$ and $g_{x,y}(t) = |t|^{p-2}t$, the author in \cite{Vazquez-2016} studied the existence, uniqueness, and various quantitative properties of strong nonnegative solutions for a Dirichlet problem involving the fractional $p$-Laplacian evolution equation:
\begin{equation}\label{vazz}
u_t(x,t) + (-\Delta)^s_p u(x,t) = 0.
\end{equation}
When equation \eqref{vazz} is coupled with a Neumann boundary condition and a Cauchy initial condition, the existence, uniqueness, and asymptotic behavior of strong solutions were established using semigroup methods in \cite{rossi}. In \cite{abd}, the authors examined equation \eqref{vazz} with a general nonlinearity depending only on $x$ and $t$, proving the existence and properties of entropy solutions. Specifically, they investigated aspects such as finite-time extinction and finite speed of propagation. See also \cite{boud, jak, Vazquez-2014, Xiang-Radulescu-Zhang-2018} for further results in this setting.

The concept of the potential well method was first introduced by Sattinger in \cite{Sattinger} to investigate nonlinear hyperbolic initial-boundary value problems. Since then, many researchers have applied potential well theory to study the existence of solutions for evolution equations; see \cite{Liu, Liu2, Payne}.\\  
In \cite{Pan}, the authors examined the following parabolic equation involving the fractional $p$-Laplacian:
\begin{equation} \label{intro:problem}
        \left\{
    \begin{aligned}
      u_{t}(x,t)+ (-\Delta)_{p}^{s} u(x,t) &= |u|^{q-2}u, &&
 \text{in}~ \Omega \times (0, T),\\
 u &= 0, && \text{on}~  \mathbb{R}^N \setminus \Omega \times (0, T),\\
 u(x,0)&= u_0(x), && \text{in}~  \Omega.
\end{aligned}
    \right.
 \end{equation}
When $1 < p < \frac{N}{s}$ and $p < q < p^\ast_s$, the authors established the existence of a global solution to problem \eqref{intro:problem} using the Galerkin method and potential well theory in the cases of low and critical initial energy, \textit{i.e.}, when $0<E(u_{0})<d$ or $E(u_{0})=d$, where $E(u_{0})$ denotes the initial energy. However, the case of high initial energy ($E(u_{0})>d$) and the possibility of finite-time blow-up for arbitrary initial energy were explored in \cite{Liao-Liu-Ye-2020}.

In contrast to the above cases, there has been limited research on evolution equations involving nonlocal operators with variable exponents. In \cite{boudd}, the author studied a nonlocal diffusion equation involving the fractional $p(x,y)$-Laplacian, \textit{i.e.}, when $g_{x,y}(t) = |t|^{p(x,y)-2}t$ and $s \in (0,1)$ in equation \eqref{main:problem}. Using the sub-differential approach, the well-posedness of the problem was established. Moreover, by combining potential well theory with the Nehari manifold, the existence of global solutions and finite-time blow-up of solutions was demonstrated. However, the study focused exclusively on the case of low initial energy and the case of critical and high energy remains an open problem.

As far as we are aware, no existing results address the well-posedness, finite-time extinction, and asymptotic behavior of solutions for fractional parabolic equations \eqref{main:problem} driven by the new fractional Musielak $g_{x,y}$--Laplacian $(-\Delta)^s_{g_{x,y}}$. This work aims to fill this gap by thoroughly investigating this important aspect. We establish results on the local and global existence of weak and strong solutions, as well as finite-time blow-up, depending on the initial energy level (low, critical, or high).  

 We follow the approach outlined in \cite{Liao-Gao-2017, Liao-Liu-Ye-2020} for the fractional $p$-Laplacian, which corresponds to the homogeneous case. However, the nonhomogeneity in our setting introduces challenges that prevent the direct application of the analysis presented in the aforementioned paper. To overcome these limitations, we develop specific tools tailored to this framework. 
 
 Comparing with the previous literature, the main contributions of this paper are as follows:
\begin{enumerate}
    \item[\textnormal{(a)}] In Theorem \ref{loc}, we prove the local existence of a strong solution to problem \eqref{main:problem} by adapting the sub-differential approach to our setting and converting \eqref{main:problem} into a first-order abstract evolution equation (see \eqref{subdifferenial main prob}) in $L^2(\Omega)$. The existence of a strong solution plays an important role in studying the asymptotic behavior of the solution at infinity and the finite-time blow-up; see Theorems \ref{thm:global-existence-strong}, \ref{Blow-up thm}, and \ref{Main thm3}. In particular, it is essential for showing the continuity properties of the energy functional $E(u(\cdot, t))$ and the Nehari functional $I(u(\cdot, t))$ with respect to time, and the invariance of the solution; see Lemmas \ref{continuity of E(u) with t} and \ref{u in W delta}. Such continuity properties are not addressed in \cite{Liao-Liu-Ye-2020} for the fractional $p$-Laplacian.

    \item[\textnormal{(b)}] For the low and critical initial energy cases, in Theorems \ref{thm:global-existence-weak} and \ref{main-exist-critical-weak}, we establish the global existence of a weak solution and prove its uniqueness if it is bounded. By constructing a suitable Galerkin scheme of approximations in the stable set $W$ and deriving uniform energy estimates, we show the local existence of a weak solution. Moreover, if the initial data $u_0$ belongs to the stable set $W$, the solution $u(\cdot, t)$ remains in $W$ for all $t \geq 0$, which further implies the global existence of a solution. In Theorems \ref{thm:global-existence-strong} and \ref{main-exist-critical-strong}, we establish the global existence of a strong solution and analyze its asymptotic behavior, including finite-time vanishing or decay to zero at infinity, depending on the energy norm of the solution and the bounds of $g_{x,y}$.

    In Theorem \ref{Blow-up thm}, we show that a strong solution blows up in finite time if the initial data belongs to the unstable set $V$. In \cite{Liao-Liu-Ye-2020, boudd, Aberqi}, finite-time blow-up is proved for the fractional $p(x,y)$-Laplacian restricted to the case $p^- > 2$ and for the fractional double-phase operator restricted to the subcritical case (see, e.g., \cite[(68)]{Aberqi}). The nonhomogeneity present in both the operator and the nonlinearities introduces various challenges, preventing the direct application of the analyses in the aforementioned papers. To address these issues, we introduce a new auxiliary function, distinct from those in \cite{Liao-Liu-Ye-2020, boudd, Aberqi}, which plays a crucial role in the application of Levine’s concavity method. We also develop new ideas tailored to this framework. Moreover, we derive estimates for the minimal blow-up time of the strong solution.

    \item[\textnormal{(c)}] For the high-energy case, in Theorem \ref{high initial energy main thm}, we provide sufficient conditions on the initial data for the global existence of a strong solution or blow-up in finite time.

    \item[\textnormal{(d)}] We extend the global existence and finite-time blow-up results for parabolic equations involving the fractional $p$-Laplacian from \cite{Liao-Liu-Ye-2020}, and the fractional $p(x,y)$-Laplacian from \cite{boudd}, to a more general class of nonlinear parabolic equations involving the fractional Musielak $g_{x,y}$--Laplacian $(-\Delta)^s_{g_{x,y}}$, including the nonlocal counterparts of local double-phase type operators studied in \cite{Arora-Shmarev-2023, Arora-Shmarev-2023-2}.
\end{enumerate}
 \vspace{0,3 cm}
 
 Throughout the paper, we assume that \(g\) satisfies the following conditions:
\begin{enumerate}[leftmargin=1.5cm,label=\textnormal{($g_0$)},ref=\textnormal{$g_0$}]
    \item \label{cond:g0} \(\lim_{t\to 0^{+}}g_{x,y}(t) = 0\) and \(\lim_{t\to \infty} g_{x,y}(t) = \infty\) for all \((x,y) \in (\Omega \times \Omega)\).
\end{enumerate}

\begin{enumerate}[leftmargin=1.5cm,label=\textnormal{($g_1$)},ref=\textnormal{$g_1$}]
    \item \label{cond:g1} The map \(t \mapsto g_{x,y}(t)\) is of class \(C^1\) on \((0,+\infty)\) for all \((x,y) \in \Omega \times \Omega\).
\end{enumerate}

\begin{enumerate}[leftmargin=1.5cm,label=\textnormal{($g_2$)},ref=\textnormal{$g_2$}]
    \item \label{assump:g2} The map \(t \mapsto g_{x,y}(t)\) is increasing on \((0,+\infty)\) for all \((x,y) \in \Omega \times \Omega\).
\end{enumerate}

\begin{enumerate}[leftmargin=1.5cm,label=\textnormal{($g_3$)},ref=\textnormal{$g_3$}]
    \item \label{Assump:g_3} There exist constants \(g^-, g^+ \in (1,+\infty)\) such that
    \[
    1 < g^- \leq \frac{g_{x,y}(t)t}{G_{x,y}(t)} \leq g^+ < g_{*,s}^-, \qquad 
     g^--1 \leq  \frac{g_{x,y}'(t)t}{g_{x,y}(t)} \leq g^+-1,
    \]
for all \((x,y) \in \Omega \times \Omega\) and \(t > 0\) where  \( G_{x, y}  : \Omega \times \Omega \times \mathbb{R} \to \mathbb{R}\) is defined as
    \[
     G_{x, y}(t):=\int_{0}^{|t|} g_{x, y}(\tau) d \tau, \quad \quad  g_{*,s}^- := 
    \begin{cases}
      \frac{Ng^-}{N - sg^-} & \text{if} \ N > s g^-,\\
      +\infty  & \text{if} \ N \leq s g^-.
    \end{cases}
    \]
\end{enumerate}

\begin{enumerate}[leftmargin=1.5cm,label=\textnormal{($g_4$)},ref=\textnormal{$g_4$}]
    \item \label{assump:g_4} The following integrability conditions hold for all \(x \in \Omega\):
    \[
    \int_{0}^{1} \frac{\widehat{G}^{-1}_x(\tau)}{\tau^{\frac{N+s}{N}}} \, d\tau < +\infty, \quad
    \int_{1}^{\infty} \frac{\widehat{G}^{-1}_x(\tau)}{\tau^{\frac{N+s}{N}}} \, d\tau = +\infty,
    \]
    where \(\widehat{G}_x : \mathbb{R} \to \mathbb{R}\) is defined as
    \begin{equation}\label{eqc}
    \widehat{G}_x(t) = \int_{0}^{|t|} \widehat{g}_x(\tau) \, d\tau, \quad \widehat{g}_x(\tau) = g(x,x,\tau).
    \end{equation}
\end{enumerate}

We also assume the function \(f : \Omega \times \mathbb{R} \to \mathbb{R}\) satisfies the following conditions:

\begin{enumerate}[leftmargin=1.5cm,label=\textnormal{($f_0$)},ref=\textnormal{$f_0$}]
    \item \label{assump:f_0}For all \( x \in \Omega \), we have \( f(x, \cdot) \in C^1(\mathbb{R}) \), and it satisfies \( f(x,0) = f'(x,0) = 0 \). Moreover, there exists a positive constant \( B \) such that  
\[
\min \{ F(x,1), F(x,-1) \} \geq B > 0, \quad \forall x \in \Omega.
\]
\end{enumerate}

\begin{enumerate}[leftmargin=1.5cm,label=\textnormal{($f_1$)},ref=\textnormal{$f_1$}]
    \item \label{Assump:f_1} For all \(x \in \Omega\), \(f(x,\cdot)\) is convex for \(t > 0\), and concave for \(t < 0\).
\end{enumerate}

\begin{enumerate}[leftmargin=1.5cm,label=\textnormal{($f_2$)},ref=\textnormal{$f_2$}]
    \item \label{Assump:f_2} There exist measurable, positive, and bounded functions \(h_1, h_2 : \Omega \to (1, +\infty)\) such that
    \[
    h_1(x) F(x,t) \leq t f(x,t), \quad t f(x,t) \leq h_2(x) F(x,t) \quad \forall (x,t) \in \Omega \times \mathbb{R},
    \]
    where \(F(x,t) := \int_{0}^{t} f(x,\tau) \, d\tau\). Additionally, \(\Phi \prec\prec \widehat{G}_x^*\) (see Subsection 2.1), where $\Phi:\Omega\times(0,+\infty)\rightarrow \mathbb{R}$ and \(\Phi(x,t) := t^{h_2(x)}\).
\end{enumerate}
\begin{enumerate}[leftmargin=1.5cm,label=\textnormal{($f_3$)},ref=\textnormal{$f_3$}]
    \item \label{Assump:f_3} For every \((x,t) \in \Omega \times \mathbb{R}\),
    \[
    t \left( f'(x,t)t - (g^+ -1) f(x,t) \right) > 0.
    \]
\end{enumerate}

Finally, we impose the following technical conditions on the functions \(g\) and \(h_2\):

\begin{enumerate}[leftmargin=1.5cm,label=\textnormal{($g_5$)},ref=\textnormal{$g_5$}]
    \item \label{assump:g_5} \(\max\{2, g^+\} < \min\{h_1^-, h_2^-\}\), where
    \[
    h_1^- := \min_{x \in \Omega} h_1(x) \quad \mbox{and} \quad h_2^- := \min_{x \in \Omega} h_2(x).
    \]
\end{enumerate}

\begin{enumerate}[leftmargin=1.5cm,label=\textnormal{($g_6$)},ref=\textnormal{$g_6$}]
    \item \label{Cond:g_6} \(h_2^+ < g_{*,s}^{-}\), where \(h_2^+ := \max_{x \in \Omega} h_2(x)\).
\end{enumerate}

\begin{remark}\label{rrem1}
Let  
\[
g_{x,y}(t) = |t|^{p(x,y)-2}t, \quad \forall x,y \in \Omega, \ \forall t \in \mathbb{R},
\]
and  
\[
f(x,t) = a(x)|t|^{q_1(x)-2}t + b(x)|t|^{q_2(x)-2}t, \quad \forall x \in \Omega, \ \forall t \in \mathbb{R},
\]
where \( a(x) \) and \( b(x) \) are two bounded functions satisfying  
\[
a(x)+b(x)\ \geq B > 0, \quad \forall x \in \Omega,
\]
for some positive constant \( B \).  

Moreover, we assume that \( p: \Omega \times \Omega \to (1,+\infty) \) is a continuous function such that  
\[
1 < p^- = \min_{(x,y) \in \overline{\Omega} \times \overline{\Omega}} p(x,y) \leq p(x,y) < p^+ = \max_{(x,y) \in \overline{\Omega} \times \overline{\Omega}} p(x,y) < +\infty.
\]
Similarly, the functions \( q_1, q_2: \Omega \to (1,+\infty) \) are continuous and satisfy  
\[
p^+ < q_1(x) \leq q_2(x) < p^{\ast}_{s} = \frac{N p(x,x)}{N - s p(x,x)}, \quad \forall \ x \in \Omega,
\]
as well as  
\[
\max\{2, p^+\} < q_1^- \quad \text{and} \quad q_2^+ < \frac{N p^-}{N - s p^-}.
\]
Thus, the functions \( f \) and \( g \) satisfy the conditions \eqref{assump:f_0}-\eqref{Assump:f_3} and \eqref{cond:g0}-\eqref{Cond:g_6}, respectively.
\end{remark}

\vspace{0,5 cm}
This paper is organized as follows. In Section $2$, we outline the definitions and properties of the Fractional Musielak-Orlicz spaces, introduce relevant notations, and prove a series of results related to potential well theory. In Section $3$, we demonstrate the well-posedness of the problem and show the existence of local solutions utilizing the sub-differential approach. Section $4$ discusses the existence of global solutions and examines the possibility for finite time blow-up and asymptotic behavior concerning problem \eqref{main:problem}. Finally, in Section 5, we provide several examples and discuss open problems.
\section{ Preliminaries}
This section is organized into two subsections. The first subsection revisits key definitions and known results concerning Musielak-Orlicz Sobolev spaces and fractional Musielak-Orlicz Sobolev spaces. In the second subsection, we introduce additional notations and definitions related to specific functionals and sets, which are essential for presenting our main results.
\subsection{Fractional Musielak-Orlicz Sobolev spaces}\label{subsection:functions:spaces}
In this subsection, we delve into the definitions and properties of Musielak-Orlicz spaces and fractional Musielak-Orlicz Sobolev spaces. For further details, we refer to the works in \cite{Alberico-Cianchi-Pick-Slavikova-2021-2,Bahrouni-Missaoui-Ounaies-2024,M1,M32}. Furthermore, we present several important properties of the function
$f$.
\begin{definition}
Let $\Omega$ be an open subset of $\mathbb{R}^{N}$. A function $G:\Omega \times \Omega \times \mathbb{R} \to \mathbb{R}$ is called a generalized N-function if it satisfies the following conditions:
\begin{enumerate}
\item[\text{(i)}] $G_{x,y}(t):=G(x,y,t)$ is even, continuous, increasing and convex in t, and for each $t\in \mathbb{R}$, $G(x,y,t)$ is measurable in $(x,y)$;
\item[\text{(ii)}] $\lim_{t\to 0} \frac{G_{x,y}(t)}{t} =0$ for a.e. $(x,y)\in \Omega\times\Omega;$
\item[\text{(iii)}] $\lim_{t\to \infty} \frac{G_{x,y}(t)}{t} =\infty$ for a.e. $(x,y)\in \Omega\times\Omega;$
\item[\text{(iv)}] $G_{x,y}(t) >0$  for all $t>0$ and a.e. $(x,y)\in\Omega\times\Omega$.
 \end{enumerate}
 \end{definition}
 \begin{definition}
We say that a generalized N-function $G_{x,y}$ satisfies the $\Delta_2$-condition if there exists $K > 0$ such that
$$G_{x,y}(2t)\leq KG_{x,y}(t),\ \ \text{for all}\ (x,y)\in \Omega\times\Omega\ \text{and all}\ t>0.$$
\end{definition}
\begin{definition}
For any generalized N-function $G_{x,y}$, the function
$\widetilde{G}_{x,y}:\Omega\times\Omega\times\mathbb{R}\longrightarrow
\mathbb{R}_+$ defined by
\begin{equation}\label{2eq50}
\widetilde{G}_{x,y}(t)=\widetilde{G}(x,y,t):=\sup_{\tau\geq 0}\left( t\tau-G_{x,y}(\tau)\right),\ \ \text{for all}\ (x,y)\in \Omega\times\Omega\ \text{and all}\ t>0
\end{equation}
 is called the complementary function of $G_{x,y}$.
\end{definition}
The assumptions $(g_0)-(g_3)$ ensure that $G_{x,y}$ and its
complementary function $\widetilde{G}_{x,y}$ are  generalized
N-functions (see \cite{M32}).\\

\begin{remark}(see \cite{M30})\label{rem1}
Assumption \eqref{Assump:g_3} gives that
 $$g^-\leq \frac{\widehat{g}(x,t)t}{\widehat{G}(x,t)}\leq g^+\ \text{and}\ \widetilde{g}^-\leq \frac{\widetilde{\widehat{g}}(x,t)t}{\widetilde{\widehat{G}}(x,t)}\leq \widetilde{g}^+,\ \text{for all}\ x\in \Omega\ \text{and all}\ t>0 $$
 where $\displaystyle{\widetilde{g}^-=\frac{g^+}{g^+-1}}$ and $\displaystyle{\widetilde{g}^+=\frac{g^-}{g^--1}}$.
 Moreover, $G_{x,y}$, $\widehat{G}_x$  and  $\widetilde{\widehat{G}}_x$ satisfy the $\Delta_2$-condition.
 \end{remark}
In view of the definition of the complementary function
$\widetilde{G}_{x,y}$, we have the following Young's type
inequality:
\begin{equation}\label{2eq95}
    \tau\sigma\leq G_{x,y}(\tau)+\widetilde{G}_{x,y}(\sigma),\ \text{for all}\ (x,y)\in \Omega\times\Omega\ \text{and all}\ \tau,\sigma\geq0.
\end{equation}

 We are now prepared to introduce the Musielak-Orlicz space.
 \begin{definition}
 Let $G_{x,y}$ be a generalized N-function. In correspondence to $\widehat{G}_{x}=G_{x,x}$ and an open subset $\Omega\subseteq \mathbb{R}^N$, the Musielak-Orlicz space is defined as follows
 $$L^{\widehat{G}_{x}}(\Omega):=\left\{u:\Omega\longrightarrow\mathbb{R} \text{ measurable }:J_{\widehat{G}_{x}}(\lambda u)<\infty \text{ for some } \lambda>0\right\},$$
where
$$J_{\widehat{G}_{x}}(u):=\int_{\Omega}\widehat{G}_x(|u|)\,dx.$$
 The space $L^{\widehat{G}_{x}}(\Omega)$ is endowed with the Luxemburg norm
 $$\|u\|_{L^{\widehat{G}_{x}}(\Omega)}:=\inf \{\lambda >0:J_{\widehat{G}_{x}}\left(\frac{u}{\lambda}\right)\leq1\}.$$
 \end{definition}
 We would like to mention that our assumptions $(g_0)-(g_3)$ ensure
that $\displaystyle{\left(L^{\widehat{G}_x}(\Omega),\Vert
\cdot\Vert_{L^{\widehat{G}_x}(\Omega)}\right)}$ is a separable and
reflexive Banach space.
 \begin{lemma}
 \label{RE:Gt}
(see \cite[Lemma 2.5]{Bahrouni-Missaoui-Ounaies-2024})     Assume that the assumptions \eqref{cond:g0}-\eqref{Assump:g_3} hold. Then, the function $\widehat{G}_x$ and $G_{x,y}$ satisfy the following properties:
     \begin{enumerate}
         \item[\text{(i)}] $\min\{\tau^{g^-} ,\tau^{g^+}\}G_{x,y}(t)\leq G_{x,y}(\tau t )\leq \max\{\tau^{g^-} ,\tau^{g^+}\}G_{x,y}(t)$ for all $x$, $y$ in $\Omega$ and for all $\tau,t>0.$
         \item[\text{(ii)}] $\min\{ \|u\|_{L^{\widehat{G}_{x}}(\Omega)}^{g^-},\|u\|_{L^{\widehat{G}_{x}}(\Omega)}^{g^+}\} \leq J_{\widehat{G}_{x}}(u)\leq\max\{ \|u\|_{L^{\widehat{G}_{x}}(\Omega)}^{g^{-}},\|u\|_{L^{\widehat{G}_{x}}(\Omega)}^{g^{+}}\}$ for all $u \in L^{\widehat{G}_{x}}(\Omega).$
     \end{enumerate}
 \end{lemma}
 As a consequence of \eqref{2eq95}, we have the following lemma:
\begin{lemma}
\label{lemhol}
(H\"older's type
inequality)
  Let $\Omega$ be an open subset of $\mathbb{R}^N$. Let $\widehat{G}_x$ be a generalized N-function and $\widetilde{\widehat{G}}_x$ its complementary function, then
  \begin{equation}\label{2eq96}
     \left\vert \int_{\Omega} uv\,dx \right\vert \leq 2 \Vert u\Vert_{L^{\widehat{G}_x}(\Omega)} \Vert v\Vert_{L^{\widetilde{\widehat{G}}_x}(\Omega)},\ \text{for all}\ u\in L^{\widehat{G}_x}(\Omega)\ \text{and all}\ v\in L^{\widetilde{\widehat{G}}_x}(\Omega).
  \end{equation}
\end{lemma}
\vspace{1\baselineskip}
Next, we define the fractional Musielak-Sobolev space.
\begin{definition}
Let $G_{x,y}$ be a generalized N-function, $s\in(0,1)$ and $\Omega$ an open subset of $\mathbb{R}^N$. The fractional Musielak-Sobolev space is defined as follows
 $$W^{s,G_{x,y}}(\Omega):=\{u\in L^{\widehat{G}_{x}}(\Omega):J_{s,G_{x,y}}(\lambda u) < +\infty\text{, for some }\lambda >0\},$$
 where
 $$J_{s,G_{x,y}}(u):=\int_{\Omega}\int_{\Omega}G_{x,y}\left(\frac{u(x)-u(y)}{|x-y|^s}\right)\frac{dx\,dy}{|x-y|^N}.$$
 The space $W^{s,G_{x,y}}(\Omega)$ is endowed with the norm
$$\|u\|_{W^{s,G_{x,y}}(\Omega)}:=\|u\|_{L^{\widehat{G}_{x}}(\Omega)}+[u]_{s,G_{x,y}}\text{, for all } u\in W^{s,G_{x,y}}(\Omega),$$
with $[u]_{s,G_{x,y}}$ is the $(s,G_{x,y})-$Gagliardo seminorm defined by
$$[u]_{s,G_{x,y}}:=\inf \{ \lambda >0:J_{s,G_{x,y}}\left(\frac{u}{\lambda}\right)\leq 1\}. $$
\end{definition}
\begin{remark}\label{rem2}
    The assumption \eqref{Assump:g_3} ensures that the functions $\widehat{G}_x$ and $\widetilde{\widehat{G}}_x$ satisfy the $\Delta_2$-condition. Consequently, the space $W^{s,G_{x,y}}(\Omega)$ is reflexive and separable as a Banach space.
\end{remark}
Let $\widehat{G}_x$ be defined as in \eqref{eqc}, the assumptions \eqref{cond:g0}-\eqref{assump:g2}  confirm that for each $x\in \Omega$, $\widehat{G}_x:\mathbb{R}_+\longrightarrow \mathbb{R}_+$  is an increasing homeomorphism. Hence, the inverse function $\widehat{G}_x^{-1}$ of $\widehat{G}_x$  exists. Then, we can define the inverse of an important function which is the Musielak-Sobolev conjugate function of $\widehat{G}_x$, denoted by $\widehat{G}_x^*$, as follows:
\begin{equation}\label{2eq10}
\left(\widehat{G}_x^*\right)^{-1}(t):= \int_{0}^{t} \frac{\widehat{G}^{-1}_x(\tau)}{\tau^{\frac{N+s}{N}}}\,d\tau,\ \ \text{for all}\ x\in \Omega\ \text{and all}\ t\geq 0.
\end{equation}
\begin{lemma}\label{RELATION:MODULAR & NORM}
     (see \cite[Lemma 2.10]{Bahrouni-Missaoui-Ounaies-2024})
  Assume that  assumptions \eqref{cond:g0}-\eqref{assump:g_4} hold with $g^-,g^+\in (1,\frac{N}{s})$ and  $s\in (0,1)$. Then, we have the following properties:
  \begin{enumerate}
   \item[$(i)$] $\displaystyle{\min\left\lbrace \left[u\right]_{s,G_{x,y}}^{g^-},\left[u\right]_{s,G_{x,y}}^{g^+}\right\rbrace \leq J_{s,G_{x,y}}(u)\leq \max\left\lbrace \left[u\right]_{s,G_{x,y}}^{g^-},\left[u\right]_{s,G_{x,y}}^{g^+}\right\rbrace}$, for all $u \in W^{s,G_{x,y}}(\Omega)$;
      \item[$(ii)$] $\displaystyle{\min\left\lbrace \Vert u\Vert_{L^{\widehat{G}_x^*}(\Omega)}^{g^{-}_{*,s}},\Vert u\Vert_{L^{\widehat{G}^*_x}(\Omega)}^{g^{+}_{*,s}}
      \right\rbrace\leq J_{\widehat{G}^*_x}(u)\leq \max\left\lbrace \Vert u\Vert_{L^{\widehat{G}^*_x}(\Omega)}^{g^{-}_{*,s}},
      \Vert u\Vert_{L^{\widehat{G}^*_x}(\Omega)}^{g^{+}_{*,s}}\right\rbrace }$, for all $u\in L^{\widehat{G}^*_x}(\Omega)$.
  \end{enumerate}
  where $\displaystyle{g^{-}_{*,s}=\frac{Ng^-}{N-sg^-}}$, $\displaystyle{g^{+}_{*,s}:=\frac{Ng^+}{N-sg^+}}$, and $\displaystyle{J_{\widehat{G}^*_x}(u):= \int_{\Omega} \widehat{G}^*_x(\vert u\vert)\,dx}$.
\end{lemma}
\begin{definition}
 We say that a generalized N-function $G_{x,y}$ satisfies the fractional boundedness condition, written $G_{x,y}\in \mathcal{B}_{f}$, if there exist $C_{1},C_{2}> 0$ such that
     \begin{equation}
         \label{fractional boundedness}
         \begin{aligned}
             C_{1}&\leq G_{x,y}(1)\leq C_{2}\qquad \text{for all  }(x,y)\in \Omega\times\Omega.
         \end{aligned}
     \end{equation}
 \end{definition}
 \begin{definition}
  Let $\widehat{A}_x$ and $\widehat{B}_x$ be two generalized  N-functions. We say that $\widehat{A}_x$ essentially grows more slowly than $\widehat{B}_x$ near
infinity, and we write  $\widehat{A}_x\prec\prec \widehat{B}_x$, if for all $k > 0$, we have
$$\lim\limits_{t\rightarrow +\infty}\frac{\widehat{A}_x(kt)}{\widehat{B}_x(t)}=0,\ \ \text{uniformly in }\ x\in \Omega.$$
\end{definition}
\begin{theorem}(see \cite{Azroul-Shimi-Srati-2023})\label{thm2}
  Let $s\in (0,1)$, $G_{x,y}$ a generalized N-function satisfying \eqref{cond:g0}-\eqref{Assump:g_3}, and $\Omega$
 a bounded domain in $\mathbb{R}^N$ with $C^{0,1}$-boundary regularity.
 \begin{enumerate}
     \item[$(i)$] If \eqref{fractional boundedness} and \eqref{assump:g_4} hold, then the embedding $W^{s,G_{x,y}}(\Omega)\hookrightarrow L^{\widehat{G}_x^*}(\Omega)$ is continuous.
     \item[$(ii)$]Moreover, for any generalized N-function $\widehat{A}_x$ such that $\widehat{A}_x\prec\prec \widehat{G}_x^*$, the embedding $W^{s,G_{x,y}}(\Omega)\hookrightarrow L^{\widehat{A}_x}(\Omega)$ is compact.
 \end{enumerate}
\end{theorem}
Next, we introduce a closed linear subspace of $W^{s,G_{x,y}}(\Omega)$. To this end,  assume that $\Omega$ is an open set in $\mathbb{R}^N$ and $$Q= (\mathbb{R}^N \times\mathbb{R}^N)\setminus(\mathcal{C}\Omega\times \mathcal{C}\Omega), \quad \mathcal{C}\Omega=\mathbb{R}^N \setminus \Omega.$$  Denote
 $$W_{0}^{s,G_{x,y}}(\Omega)=\{u : u\in L^{\widehat{G}_x}(\mathbb{R}^{N}), \ u=0 \text{ in }\mathbb{R}^N \setminus\Omega,\frac{u(x)-u(y)}{|x-y|^s} \in L^{\widehat{G}_x}(Q, \frac{dxdy}{|x-y|^N})\}.$$
 The space $W_{0}^{s,G_{x,y}}(\Omega)$ is a normed linear subspace of $W^{s,G_{x,y}}(\mathbb{R}^N)$ equipped with the norm (\mbox{see \cite{Azroul-Shimi-Srati-2023}})
 $$\|u\|_{W_{0}^{s,G_{x,y}}(\Omega)}=[u]_{s,G_{x,y}}.$$
 Due to \eqref{cond:g0}-\eqref{Assump:g_3} and Theorem \ref{thm2}, $W_{0}^{s, G_{x,y}}(\Omega)$ is compactly embedded in $L^{\Phi}(\Omega)$ {\it i.e.} there exists a positive constant $C_{1,G}$ such that
 \begin{equation}
 \label{Compacrt embedding constant}
 \begin{aligned}
    \|u\|_{L^{\Phi}(\Omega)}&\leq C_{1,G} [u]_{s,G_{x,y}} \quad \text{for all} \ u \in W_{0}^{s, G_{x,y}}(\Omega).
 \end{aligned}
 \end{equation}
       
 \begin{remark}
The results derived above remain valid if we replace $W^{s,G_{x,y}}(\Omega)$ with $W_{0}^{s,G_{x,y}}(\Omega)$.
\end{remark}
Next, we present several important properties of the function $f$ satisfying the assumptions \eqref{assump:f_0}-\eqref{Assump:f_2}.
\begin{lemma}
    \label{tf(x,t)>0}
    Let $f$ satisfy the assumptions \eqref{assump:f_0}-\eqref{Assump:f_1}, then:
    \begin{enumerate}
        \item [$(i)$]
    $tf(x,t)\geq 0$ and $F(x,t)\geq 0$ for all $t\in \mathbb{R}$ and a.e. $x\in \Omega.$
    \item [$(ii)$]For each $x\in \Omega$, $f(x,.)$ is non-decreasing on $\mathbb{R}$.
    \end{enumerate}
    \end{lemma}
    \begin{proof}
       $(i)$ From assumption \eqref{Assump:f_1} we have that $f(x,\cdot)$ is convex for $t>0$ and concave for $t<0.$ So,
        we divide the proof into two cases: \vspace{0.1cm}\\
\textbf{Case 1:} $t\geq 0$. In light of \eqref{Assump:f_1}, for any $t,t_{0}\geq 0$, we deduce that
        \begin{equation}
        \label{convexity for t}
        \begin{aligned}
            f(x,t)&\geq f(x,t_{0})+f'(x,t_{0})(t-t_{0})
        \end{aligned}
        \end{equation}
        and
        \begin{equation}
        \label{convexity for t_{0}}
        \begin{aligned}
            f(x,t_{0})&\geq f(x,t)+f'(x,t)(t_{0}-t).
        \end{aligned}
        \end{equation}
        Thus, adding \eqref{convexity for t} and \eqref{convexity for t_{0}}, we get

            $$\left( f'(x,t)-f'(x,t_{0})\right)(t_{0}-t)\leq 0,$$
          and
            $$\left( f'(x,t)-f'(x,t_{0})\right)(t-t_{0})\geq 0.
        $$
        Therefore, choosing $t_{0}=0$ and by \eqref{assump:f_0} , we obtain
         \begin{equation}\label{pos}f'(x,t)  \geq 0, \ \ \forall \ t\geq 0.\end{equation}
        Thus, again in view of \eqref{assump:f_0}, we conclude that $$f(x,t)t\geq 0, \ \forall \ t\geq 0.$$
\textbf{Case 2:} $t\leq 0$. By repeating the above argument, substituting the convexity of $f$ with its concavity, we infer that
        \begin{equation}\label{neg}f'(x,t)  \geq 0, \ \ f(x,t)\geq 0  \ \ \mbox{and} \ \ f(x,t)t \geq 0, \ \ \forall \ t\leq 0.\end{equation}
Having in mind that $F'(x,t)=f(x,t)$ and $F(x,0)=0$ for a.e. $x \in \Omega$, and exploiting \eqref{pos} and \eqref{neg}, we deduce that
         $$F(x,t)\geq 0, \ \ \forall \ t\in \mathbb{R}.$$
         $(ii)$ The proof of assertion $(ii)$ follows from \eqref{pos} and \eqref{neg}. This finishes the proof.
    \end{proof}
\begin{lemma}
\label{F:bounds}
    If $f$ satisfies the conditions \eqref{assump:f_0}-\eqref{Assump:f_2}, then:
    \begin{enumerate}
    \item[\text{(i)}] There exists a positive constant $A$ such that $$|F(x,t)| \leq A |t|^{h_2(x)}, \ \mbox{for all} \ \ t\in \mathbb{R}.$$
        \item[\text{(ii)}] There exists a positive constant $B$ such that  $$|F(x,t)| \geq B|t|^{h_1(x)}, \ \ \mbox{for all} \ \ |t|\geq 1.$$
    \end{enumerate}
\end{lemma}
\begin{proof}
$(i)$ Note that, from Lemma \ref{tf(x,t)>0}, we have $F(x,t)> 0$, for all $x\in \Omega$ and all $t\neq 0.$ On the other hand, from \eqref{Assump:f_2}, \eqref{pos} and \eqref{neg}, one has
    $$\frac{F'(x,t)}{F(x,t)}\leq h_{2}(x)\frac{1}{|t|}, \ \ \forall t\neq 0.$$
    It follows,  integrating the above inequality with respect to $t$ ($t\neq 0)$, that
$$\ln{\left(F(x,t)\right)}\leq h_{2}(x)\ln{(|t|)}+\ln{(A)} \ \ \mbox{and} \ \ \ln{\left(F(x,t)\right)}\leq \ln{\left(A|t|^{h_{2}(x)}\right)},$$
for some positive constant $A$, which implies that
$$|F(x,t)|=F(x,t)\leq A|t|^{h_{2}(x)}, \ \ \forall t\in \mathbb{R}.$$
$(ii)$ Recall that, from Lemma \ref{tf(x,t)>0}, $tf(x,t)$ and $F(x,t)$ are always non-negative. We consider two cases: \vspace{0.1cm}\\
\textbf{Case 1:} $t < 0$. Then, by \eqref{Assump:f_2}, one has
      $$
    \frac{h_{1}(x)}{t}\geq \frac{F'(x,t)}{F(x,t)}.
     $$
       Integrating the above inequality with respect to $t$ from $t$ to $-1$, we get
       \[F(x,t)\geq F(x,-1)(-t)^{h_{1}(x)}.\]
\textbf{Case 2:} $t >0$. Following the same argument used above, we obtain that
       \[F(x,t)\geq F(x,1)t^{h_{1}(x)}.\]
    From Cases $1$ and $2$, we conclude that
    \[F(x,t)\geq B|t|^{h_{1}(x)} \text{  for all } |t|\geq 1,\]
    where
    \[B=\min\{ F(x,-1),F(x,1)\}.\] 
    This ends the proof.
\end{proof}
\begin{corollary}
\label{bounds on f:f}
    Under the conditions of Lemma \ref{F:bounds}, we have:\\
    \begin{enumerate}
        \item[\text{(i)}] $|tf(x,t)|\leq h_2(x)A|t|^{h_2(x)}$, $|f(x,t)|\leq h_2(x)A|t|^{h_2(x)-1}$, for all $t\in \mathbb{R}$ and all $x\in \Omega$.
        \item[\text{(ii)}] $tf(x,t)\geq Bh_1(x)|t|^{h_1(x)}$, for  all $|t|\geq 1$ and all $x\in \Omega$.
    \end{enumerate}
    \end{corollary}
    \begin{proof}
        Both assertions follow from Lemmas \ref{F:bounds} and \eqref{tf(x,t)>0}, combined with condition \eqref{Assump:f_2}.
        \end{proof}

\subsection{ Potential well}
Throughout this paper, we will use the following notations. For $x, y \in \mathbb{R}^N$ and $u \in W_{0}^{s,G_{x,y}}(\Omega)$, we denote:
\[
    d\mu := \frac{dx\,dy}{|x - y|^N} \quad \text{and} \quad
    D^{s}u := \frac{u(x) - u(y)}{|x - y|^s}.
\]
The energy functional $E : W_{0}^{s,G_{x,y}}(\Omega) \to \mathbb{R}$, corresponding to the stationary counterpart of the problem \eqref{main:problem}, is defined as:
\[
    E(u) := \iint_{Q} G_{x,y} \left( D^{s} u \right) d\mu - \int_{\Omega} F(x,u) \, dx.
\]

From Lemma \ref{F:bounds}, condition \eqref{Assump:f_2}, and the continuous embedding of $W_{0}^{s,G_{x,y}}(\Omega)$ into $L^{\Phi}(\Omega)$ (Theorem \ref{thm2}), it follows that the energy functional $E$ is well-defined. Moreover, $E \in C^{1}(W_{0}^{s,G_{x,y}}(\Omega), \mathbb{R})$ and
\[
   <E'(u),u> = I(u) := \iint_{Q} g_{x,y} \left( D^{s} u \right) \left( D^{s} u \right) d\mu - \int_{\Omega} f(x,u) u \, dx,
\]
 where $<\cdot,\cdot>$ is  the duality brackets for the pair $\left((W_{0}^{s,G_{x,y}}(\Omega))^{\ast}, W_{0}^{s,G_{x,y}}(\Omega)\right)$.\\
The functional $I : W_{0}^{s,G_{x,y}}(\Omega) \to \mathbb{R}$ is referred to as the Nehari functional. Next, we define the Nehari manifold as
\[
    \mathcal{N} := \left\{ u \in W_{0}^{s,G_{x,y}}(\Omega) \setminus \{ 0 \} \mid I(u) = 0 \right\}.
\]
The potential well and its corresponding sets are defined as
\[
    W := \left\{ u \in W_{0}^{s,G_{x,y}}(\Omega) \mid I(u) > 0, E(u) < d \right\} \cup \{0\}
\]
and
\[
    V := \left\{ u \in W_{0}^{s,G_{x,y}}(\Omega) \mid I(u) < 0, E(u) < d \right\},
\]
where the depth $d$ of the potential well $W$ is defined as
\[
    d := \inf_{u \in \mathcal{N}} E(u).
\]
       \begin{lemma}\label{positive:depth}
           Let $f$ satisfy the conditions \eqref{assump:f_0}-\eqref{Assump:f_2}, and let $g$ satisfy the conditions \eqref{cond:g0}-\eqref{assump:g_5}. Then, the depth $d$ of the potential well $W$ is positive. Moreover, there exists a positive constant $\delta$ such that $d > \delta$.

       \end{lemma}
       \begin{proof}
        Let $u\in\mathcal{N}$. Then, from  \eqref{Assump:g_3}, \eqref{Assump:f_2}, Lemma \ref{RELATION:MODULAR & NORM}, Corollary \ref{bounds on f:f} (i),  and Theorem \ref{thm2}, we obtain
        
        \begin{align}
         \label{lower bound of norm u}
   g^{-}\min\{[u]_{s,G_{x,y}}^{g^-},[u]_{s,G_{x,y}}^{g^+}\}&\leq g^{-} \iint_{Q}G_{x,y}\left(D^{s}u\right)d\mu \leq \iint_{Q}g_{x,y}\left(D^{s}u\right)\left(D^{s}u\right)d\mu \nonumber\\
   & = \int_{\Omega}f(x,u)u\ dx \leq \int_{\Omega}A h_{2}(x)|u|^{h_{2}(x)}dx \nonumber\\
   &\leq Ah_{2}^{+}\max\{\|u\|_{L^{\Phi}(\Omega)}^{h_{2}^{-}},\|u\|_{L^{\Phi}(\Omega)}^{h_{2}^{+}}\} \nonumber\\
  &\leq Ah_{2}^{+}\max\left\{C_{1,G}^{h_{2}^{-}}[u]_{s,G_{x,y}}^{h_{2}^{-}},  C_{1,G}^{h_{2}^{+}}[u]_{s,G_{x,y}}^{h_{2}^{+}}\right\} \nonumber\\
  &\leq Ah_{2}^{+}C^{*}_{G}\max\left\{[u]_{s,G_{x,y}}^{h_{2}^{-}},[u]_{s,G_{x,y}}^{h_{2}^{+}}\right\},
        \end{align}
        where
        $C^{*}_{G}:= \max\{C_{1,G}^{h_{2}^{-}},C_{1,G}^{h_{2}^{+}}\}.$
       Furthermore, we can divide \eqref{lower bound of norm u} into two cases: $[u]_{s,G_{x,y}} \geq 1$ and $[u]_{s,G_{x,y}} < 1$, which yields

        \begin{equation*}
            \label{lower bound of u in W0}
        \begin{aligned}
            [u]_{s,G_{x,y}}&\geq \min\left\{\left(\frac{g^{-}}{Ah_{2}^{+}C^{*}_{G}}\right)^{\frac{1}{h_{2}^{+}-g^{-}}}, \left(\frac{g^{-}}{Ah_{2}^{+}C^{*}_{G}}\right)^{\frac{1}{h_{2}^{-}-g^{+}}}\right\}:=\delta_{\min}.
        \end{aligned}
        \end{equation*}
        Therefore, by using \eqref{Assump:f_2}, \eqref{Assump:g_3}, \eqref{assump:g_5}, Lemmas \ref{RELATION:MODULAR & NORM} and \ref{tf(x,t)>0}, we infer that
       \begin{align*}
        E(u) & \geq {\iint_{Q}}G_{x,y}\left(D^{s}u\right)d\mu-\int_{\Omega}\frac{f(x,u)u}{h_{1}(x)}dx\\
        &\geq{\iint_{Q}}G_{x,y}\left(D^{s}u\right)d\mu-\frac{1}{h_{1}^{-}}{\iint_{Q}}g_{x,y}\left(D^{s}u\right)\left(D^{s}u\right)d\mu\\
        & \geq \iint_{Q} G_{x,y}\left(D^{s}u\right)d\mu-\frac{g^{+}}{h_{1}^{-}}{\iint_{Q}}G_{x,y}\left(D^{s}u\right)d\mu\\
        & \geq \left(1-\frac{g^{+}}{h_{1}^{-}}\right){\iint_{Q}}G_{x,y}\left(D^{s}u\right)d\mu\\
        &\geq \left(1-\frac{g^{+}}{h_{1}^{-}}\right) \min\{[u]_{s,G_{x,y}}^{g^-},[u]_{s,G_{x,y}}^{g^+}\}\\
        &\geq \left(1-\frac{g^{+}}{h_{1}^{-}}\right) \min\{\delta_{\min}^{g^{-}}, \delta_{\min}^{g^{+}}\}.
\end{align*}
       Thus, the desired result follows.
       \end{proof}
       \begin{lemma}
       \label{I(u)>0 implies E(u)>0}
         For $u \in W_{0}^{s,G_{x,y}}(\Omega)$, if $I(u) > 0$, $f$ satisfies the condition \eqref{Assump:f_2}, and $g$ satisfies the conditions \eqref{cond:g0}-\eqref{assump:g_5}, then $E(u) > 0$.
       \end{lemma}
       \begin{proof}
            We proceed by contradiction. Suppose that there exists $u \in W_{0}^{s,G_{x,y}}(\Omega)$  such that $E(u)\leq 0$ and $I(u) >0$.
            From the conditions \eqref{Assump:f_2}, \eqref{Assump:g_3} and Lemma \ref{tf(x,t)>0}, we have
            \begin{align*}
            {\iint_{Q}}G_{x,y}\left(D^{s}u\right)d\mu&\leq\int_{\Omega}\frac{f(x,u)u}{h_{1}(x)}dx \leq \int_{\Omega}\frac{f(x,u)u}{h_{1}^{-}}dx\\
           &\leq\frac{1}{h_{1}^{-}}{\iint_{Q}}g_{x,y}\left(D^{s}u\right)\left(D^{s}u\right)d\mu\\
           &\leq\frac{g^{+}}{h_{1}^{-}}{\iint_{Q}}G_{x,y}\left(D^{s}u\right)d\mu.
                \end{align*}
             This implies that $h_{1}^{-} \leq g^{+}$, which contradicts condition \eqref{assump:g_5}. This contradiction completes the proof.
       \end{proof}
\begin{lemma}
\label{Lemma:3.3}
    Let $f$ satisfy the conditions \eqref{assump:f_0}--\eqref{Assump:f_3}, and $g$ satisfy the conditions \eqref{cond:g0}--\eqref{Assump:g_3}. Then, for any $u \in W_{0}^{s,G_{x,y}}(\Omega)$ with $[u]_{s,G_{x,y}} \neq 0$, we have:
\begin{enumerate}
    \item[$(i)$] $\lim_{\lambda \to 0^{+}} E(\lambda u) = 0$, \quad $\lim_{\lambda \to +\infty} E(\lambda u) = -\infty$.
    \item[$(ii)$] There exists a unique $\lambda^{*} = \lambda^{*}(u) > 0$ such that $\frac{dE(\lambda u)}{d\lambda} \big|_{\lambda = \lambda^{*}} = 0$. Moreover, $E(\lambda u)$ is increasing on $0 < \lambda \leq \lambda^{*}$, decreasing on $\lambda^{*} \leq \lambda < \infty$, and attains its maximum at $\lambda = \lambda^{*}$.
    \item[$(iii)$] $I(\lambda u) \geq 0$ for $0 < \lambda \leq \lambda^{*}$, $I(\lambda u) < 0$ for $\lambda^{*} < \lambda < \infty$, and $I(\lambda^{*} u) = 0$.
\end{enumerate}
\end{lemma}
\begin{proof}
$(i)$ Let $\lambda>0$ and $u \in W_{0}^{s,G_{x,y}}(\Omega)$ with $[u]_{s,G_{x,y}} \neq 0$. Then, by invoking Lemmas \ref{RE:Gt} and \ref{F:bounds}, we obtain
\begin{align*}
    E(\lambda u) &\leq{\iint_{Q}}G_{x,y}\left(\lambda D^{s}u\right)d\mu-\int_{\Omega\cap\{|u|\geq \frac{1}{\lambda}\}}B\lambda^{h_{1}(x)}|u|^{h_1(x)}dx
\\ &\leq\max\{\lambda^{g^{+}},\lambda^{g^{-}}\}{\iint_{Q}}G_{x,y}\left(D^{s}u\right)d\mu\\
& \qquad - B\min\{\lambda^{h_{1}^{+}},\lambda^{h_{1}^{-}} \}\int_{\Omega\cap\{|u|\geq \frac{1}{\lambda}\}}|u|^{h_1(x)}dx \label{Eq:max bound}.
\end{align*}
It follows, due to \eqref{assump:g_5} and the continuity of $E$, that
\[\lim_{\lambda \to +\infty}E(\lambda u)=-\infty, \quad \text{and} \quad \lim_{\lambda \to 0^{+}}E(\lambda u)=E(0)= 0.\]
$(ii)$  Let $\lambda>0$ and $u \in W_{0}^{s,G_{x,y}}(\Omega)$ with $[u]_{s,G_{x,y}} \neq 0$. In light of Lemmas \ref{RE:Gt} and \ref{F:bounds}, we obtain
\begin{align*}
E(\lambda u)&\geq{\iint_{Q}}G_{x,y}\left(\lambda  D^{s}u\right)d\mu -\int_{\Omega}A\lambda^{h_{2}(x)}|u|^{h_2(x)}dx
\\ &\geq\min\{\lambda^{g^{+}},\lambda^{g^{-}}\}{\iint_{Q}}G_{x,y}\left(D^{s}u\right) d\mu\\
    &\qquad  - A\max\{\lambda^{h_{2}^{+}}, \lambda^{h_{2}^{-}}\}\int_{\Omega}|u|^{h_2(x)}dx,\label{Eq:min bound}
\end{align*}
which implies that
 $$E(\lambda u)>0 \ \ \mbox{ for all } \ \ \lambda\in\left(0,\min\left\{ \left(\frac{K_{1}(u)}{A K_{2}(u)}\right)^{\frac{1}{h_{2}^{-}-g^{+}}}, \left(\frac{K_{1}(u)}{A K_{2}(u)}\right)^{\frac{1}{h_{2}^{+}-g^{-}}} \right\}\right),$$
where $$K_1(u)={\iint_{Q}}G_{x,y}\left(D^{s}u\right)d\mu \ \ \text{  and  }\ \ K_2(u)=\int_{\Omega}|u|^{h_2(x)}dx.$$ \\
Therefore, having in mind that  $\lim_{\lambda \to +\infty}E(\lambda u)=-\infty$, there exists $\lambda=\lambda^{*}(u)>0$ such that $\frac{dE(\lambda u)}
   {d\lambda}|_{\lambda=\lambda^{*}}=0,$ namely,
   \begin{equation}
       \label{exixsts lamds}
       \begin{aligned}
           \iint_{Q}g_{x,y}\left(\lambda^{*} D^{s}u\right)\left(D^{s}u \right)d\mu&=\int_{\Omega}f(x,\lambda^* u)u\ dx.
       \end{aligned}
   \end{equation}
   On the other hand, by \eqref{Assump:f_3}, \eqref{Assump:g_3} and \eqref{exixsts lamds}, we obtain
   \begin{equation}
       \label{existence of lambda d lambda}
   \begin{aligned}
      \frac{d^{2}E(\lambda u)}{d\lambda^{2}}|_{\lambda=\lambda^*}&={\iint_{Q}}g_{x,y}'\left(\lambda^{*}D^{s}u\right)\left( D^{s}u\right)^{2}d\mu-\int_{\Omega}f'(x,\lambda^{*} u) u^{2}dx\\
      &= \frac{1}{(\lambda^{*})^{2}}{\iint_{Q}}g_{x,y}'\left(\lambda^{*}D^{s}u\right)\left( \lambda^{*} D^{s}u\right)^{2}d\mu-\int_{\Omega}f'(x,\lambda^{*} u) (\lambda^{*} u)^{2}dx\\ 
    &\leq \frac{1}{(\lambda^{*})^{2}} \left[{\iint_{Q}}(g^{+}-1) g_{x,y}\left(\lambda^{*}D^{s}u\right)\left(\lambda^{*}D^{s}u\right)d\mu-\int_{\Omega}f'(x,\lambda^{*} u)(\lambda^* u)^{2}dx \right]\\
    &\leq \frac{1}{(\lambda^{*})^{2}}\left[\int_{\Omega}\left[(g^{+}-1)f(x,\lambda^{*}u)-f'(x,\lambda^{*} u)(\lambda^* u)\right](\lambda^* u)dx\right]\\
    &<0.
     \end{aligned}
   \end{equation}
   This proves that $E(\lambda u)$ is increasing on $0 < \lambda \leq \lambda^{*}$ and decreasing on $\lambda^{*} \leq \lambda < \infty$.
Next, we prove that $\lambda = \lambda^{*}(u)$ is uniquely determined. Assume that there exist two distinct roots, say $\lambda_1$ and $\lambda_2$, such that

$$\frac{dE(\lambda u)}{d\lambda}|_{\lambda=\lambda_{1}}=0\text{, }\frac{d^{2}E(\lambda u)}{d\lambda^{2}}|_{\lambda=\lambda_{1}}<0,$$
and
$$\frac{dE(\lambda u)}{d\lambda}|_{\lambda=\lambda_{2}}=0\text{, }\frac{d^{2}E(\lambda u)}{d\lambda^{2}}|_{\lambda=\lambda_{2}}<0.$$
Thus, there exists $\lambda_3$ such that $\lambda_1 < \lambda_3 < \lambda_2$ and $E(\lambda_3 u)$ is the minimum of $E(\lambda u)$ on the interval $[\lambda_1, \lambda_2]$. Therefore, we have
$$
\frac{dE(\lambda u)}{d\lambda} \bigg|_{\lambda = \lambda_3} = 0, \quad \frac{d^2 E(\lambda u)}{d\lambda^2} \bigg|_{\lambda = \lambda_3} \geq 0,
$$
which leads to a contradiction with \eqref{existence of lambda d lambda}.\\
$(iii)$ Assertion $(iii)$ follows from $(ii)$ and the fact that
\begin{align}\label{Derivative of I} 
 I(\lambda u)&=\lambda\left(\frac{dE(\lambda u)}{d\lambda}\right).
\end{align}
\end{proof}
 For any $\delta>0$, we define the modified functional and Nehari manifold as follows:\\
$$I_{\delta}(u):=\delta {\iint_{Q}}g_{x,y}\left(D^{s}u\right)\left(D^{s}u\right)d\mu-\int_{\Omega}f(x,u)u\ dx,$$
and
$$\mathcal{N}_{\delta}:=\{u\in W_{0}^{s,G_{x,y}}(\Omega)\setminus\{0\}\ |\ I_{\delta}(u)=0\}.$$
The corresponding modified potential well and its corresponding set are defined as
$$W_{\delta}:=\{u\in W_{0}^{s,G_{x,y}}(\Omega)\ |\ I_{\delta}(u)>0,E(u)< d(\delta)\}\cup{\{0\}},$$
and
$$V_{\delta}:=\{u\in W_{0}^{s,G_{x,y}}(\Omega)\ |\ I_{\delta}(u)<0,E(u)< d(\delta) \},$$
 where the depth of the modified potential well $W_{\delta}$ is defined as
 $$d(\delta)=\inf_{u\in\mathcal{N}_{\delta}}E(u).$$
Set
\[ y(\delta)=\frac{\delta g^{-}}{h_{2}^{+}AC^{*}_{G}},  \  z(\delta)= \frac{\delta g^{-}}{h_{2}^{+}AC_{max}},\]

and 

\[
C^{*}_{G}:= \max\{C_{1,G}^{h_{2}^{-}},C_{1,G}^{h_{2}^{+}}\}, \quad  C_{max}:= \max \left\{ C^{g^{-}}_{1,G} \ ,\ C^{g^{+}}_{1,G} \right\}.
\]

\begin{lemma}
\label{Upper estimate of norm}
   Let the conditions \eqref{assump:f_0}-\eqref{Assump:f_2} and \eqref{cond:g0}-\eqref{assump:g_5} hold true. Then, for any $u\in W_{0}^{s,G_{x,y}}(\Omega)$,  we have:
   \begin{enumerate}
  \item[$(i)$] If $ I_{\delta}(u)<0$, then $[u]_{s,G_{x,y}} > \min\left\{\left(y(\delta)\right)^{\frac{1}{h_{2}^{-}-g^{+}}}, \left(y(\delta)\right)^{\frac{1}{h_{2}^{+}-g^{-}}}\right\}$.
   \item[$(ii)$] If $I_{\delta}(u)=0$, then either $[u]_{s,G_{x,y}}=0$ or $[u]_{s,G_{x,y}}\geq \min\left\{\left(y(\delta)\right)^{\frac{1}{h_{2}^{-}-g^{+}}}, \left(y(\delta)\right)^{\frac{1}{h_{2}^{+}-g^{-}}}\right\}.$
     \item[$(iii)$] If $ I_{\delta}(u)<0$, then  $\|u\|_{L^{\Phi}(\Omega)} > \min \left\{\left(z(\delta)\right)^{\frac{1}{h_{2}^{-}-g^{+}}}, \left(z(\delta)\right)^{\frac{1}{h_{2}^{+}-g^{-}}}\right\}$.
    
   \end{enumerate}
\end{lemma}
\begin{proof}
  $(i)$ Let  $u\in W_{0}^{s,G_{x,y}}(\Omega)$
such that $I_{\delta}(u) < 0.$ From condition \eqref{Assump:g_3}, Lemma \ref{RELATION:MODULAR & NORM}, Corollary \ref{bounds on f:f} and Theorem \ref{thm2}, we have
    \begin{align}\label{I delta u < 0}
    I_{\delta}(u) & \geq \delta g^{-} \iint_{Q}G_{x,y}\left(D^{s}u\right)d\mu - h_{2}^{+}A\int_{\Omega}|u|^{h_{2}(x)}dx  \\
    \label{lower bound in Phi}
     &\geq \delta g^{-} \min\left\{[u]_{s,G_{x,y}}^{g^{-}},  [u]_{s,G_{x,y}}^{g^{+}}\right\} -h_{2}^{+}A\max\{ \|u\|_{L^{\Phi}(\Omega)}^{h_{2}^{-}},\|u\|_{L^{\Phi}(\Omega)}^{h_{2}^{+}}\}\nonumber\\
      &\geq \delta g^{-} \min\left\{[u]_{s,G_{x,y}}^{g^{-}},  [u]_{s,G_{x,y}}^{g^{+}}\right\}-h_{2}^{+}A C^{*}_{G}\max\left\{[u]_{s,G_{x,y}}^{h_{2}^{-}},  [u]_{s,G_{x,y}}^{h_{2}^{+}}\right\}.
      \nonumber
    \end{align}
Since  $I_{\delta}(u)< 0$, from \eqref{I delta u < 0}, we get 
   \begin{equation} \label{I delta norm lower bound}
      h_{2}^{+}A C^{*}_{G}\max\left\{[u]_{s,G_{x,y}}^{h_{2}^{-}},  [u]_{s,G_{x,y}}^{h_{2}^{+}}\right\} >\delta g^{-} \min\left\{[u]_{s,G_{x,y}}^{g^{-}},  [u]_{s,G_{x,y}}^{g^{+}}\right\}, \end{equation}
%
which implies  that
$$
            \frac{ \max\left\{[u]_{s,G_{x,y}}^{h_{2}^{-}},  [u]_{s,G_{x,y}}^{h_{2}^{+}}\right\}}{\min\left\{[u]_{s,G_{x,y}}^{g^{-}},  [u]_{s,G_{x,y}}^{g^{+}}\right\} } >  \frac{\delta g^{-}}{h_{2}^{+}A C^{*}_{G}}.
            $$
        Therefore, by dividing \eqref{I delta norm lower bound} into two cases: $[u]_{s,G_{x,y}} \geq 1$ and $[u]_{s,G_{x,y}} <1 $, we conclude that
   $$[u]_{s,G_{x,y}}> \min \left\{\left(y(\delta)\right)^{\frac{1}{h_{2}^{-}-g^{+}}}, \left(y(\delta)\right)^{\frac{1}{h_{2}^{+}-g^{-}}}\right\}.$$

$(ii)$  Let  $u\in W_{0}^{s,G_{x,y}}(\Omega)$ such that $I_{\delta}(u) = 0$. If $[u]_{s,G_{x,y}} = 0$, we are done. Otherwise, if $[u]_{s,G_{x,y}} \neq 0$, then the proof follows by adapting the same argument used previously.\\
$(iii)$ Let  $u\in W_{0}^{s,G_{x,y}}(\Omega)$
such that $I_{\delta}(u) < 0$. Then, from \eqref{I delta u < 0} and   \eqref{Compacrt embedding constant}, we get
   \begin{align*}
     I_{\delta}(u) &  \geq \delta g^{-} \min\left\{\frac{\|u\|^{g^{-}}_{L^{\Phi}(\Omega)}}{{C^{g^{-}}_{1,G}}}, \frac{\|u\|^{g^{+}}_{L^{\Phi}(\Omega)}}{{C^{g^{+}}_{1,G}}}  \right\} -h_{2}^{+}A\max\{ \|u\|_{L^{\Phi}(\Omega)}^{h_{2}^{-}},\|u\|_{L^{\Phi}(\Omega)}^{h_{2}^{+}}\}.
   \end{align*}
   It follows, since  $I_{\delta}(u)< 0$, that
   
      $$ h_{2}^{+}A\max\{ \|u\|_{L^{\Phi}(\Omega)}^{h_{2}^{-}},\|u\|_{L^{\Phi}(\Omega)}^{h_{2}^{+}}\}  > \delta g^{-} \min\left\{\frac{\|u\|^{g^{-}}_{L^{\Phi}(\Omega)}}{{C^{g^{-}}_{1,G}}}, \frac{\|u\|^{g^{+}}_{L^{\Phi}(\Omega)}}{{C^{g^{+}}_{1,G}}}  \right\},
      $$
      which proves that
       \begin{align}
   \label{L Phi norm lower bound}
        \frac{\max\{ \|u\|_{L^{\Phi}(\Omega)}^{h_{2}^{-}},\|u\|_{L^{\Phi}(\Omega)}^{h_{2}^{+}}\}}{\min\left\{ \|u\|^{g^{-}}_{L^{\Phi}(\Omega)}, \|u\|^{g^{+}}_{L^{\Phi}(\Omega)}  \right\} }  >  \frac{\delta g^{-}}{ h_{2}^{+}A C_{max}}.
   \end{align}
    Therefore, by dividing \eqref{L Phi norm lower bound} into two cases: $\|u\|_{L^{\Phi}(\Omega)} \geq 1$ and $\|u\|_{L^{\Phi}(\Omega)} <1 $, we conclude that
    \[ \|u\|_{L^{\Phi}(\Omega)} > \min \left\{\left(z(\delta)\right)^{\frac{1}{h_{2}^{-}-g^{+}}}, \left(z(\delta)\right)^{\frac{1}{h_{2}^{+}-g^{-}}}\right\}.\]  

\end{proof}

By repeating the arguments of Lemma \ref{Lemma:3.3}, we have the following result for the modified functional $I_\delta:$ 

\begin{corollary}\label{cor}
\label{I delta}
   Let $f$ satisfy the conditions \eqref{assump:f_0}--\eqref{Assump:f_3}, and $g$ satisfy the conditions \eqref{cond:g0}--\eqref{Assump:g_3}. Then, for any $u \in W_{0}^{s,G_{x,y}}(\Omega)$ with $[u]_{s,G_{x,y}} \neq 0$ and $\delta>0 $ there exists a unique $\lambda^{*}=\lambda^{*}(\delta, u)$ such that:
$I_{\delta}(\lambda u) \geq 0$ for $0 < \lambda \leq \lambda^{*}$, $I_{\delta}(\lambda u) < 0$ for $\lambda^{*} < \lambda < \infty$, and $I_{\delta}(\lambda^{*} u) = 0$.
\end{corollary}

\begin{lemma}
\label{prop:d(delta)}
  Let the conditions \eqref{assump:f_0}--\eqref{Assump:f_3} and \eqref{cond:g0}--\eqref{Assump:g_3} hold. Then, the function $d(\delta)$ satisfies the following properties:
\begin{enumerate}
    \item[$(i)$] $\lim_{\delta \to 0^{+}} d(\delta) = 0$  \ and \ $\lim_{\delta \to \infty} d(\delta) = -\infty$.
    \item[$(ii)$] $d(\delta)$ is increasing on $0 < \delta \leq 1$, decreasing on $\delta \geq 1$, and attains its maximum, $d = d(1)$, at $\delta = 1$.
\end{enumerate}
\end{lemma}
\begin{proof}
$(i)$ By invoking Corollary \ref{cor}, for any $u\in W_{0}^{s,G_{x,y}}(\Omega)$ with $[u]_{s,G_{x,y}}\neq 0$ and $\delta> 0$, there exists a unique $\lambda=\lambda(\delta,u)$ such that $I_{\delta}(\lambda u)=0$. Thus,
$$\delta {\iint_{Q}}g_{x,y}\left(\lambda     D^{s}u\right)\left(\lambda D^{s}u\right)d\mu =\int_{\Omega}f(x, \lambda u) \lambda u\ dx.$$
Therefore, for a fixed $u\in W_{0}^{s,G_{x,y}}(\Omega)$ with $[u]_{s,G_{x,y}}\neq 0$, we define
\begin{equation}
\label{definition of eta}
\begin{aligned}
\delta=\eta(\lambda):&=\frac{\int_{\Omega}f(x, \lambda u) \lambda u\ dx}{{\iint_{Q}}g_{x,y}\left(\lambda D^{s}u\right)\left(\lambda D^{s}u\right)d\mu}.
\end{aligned}
\end{equation}

\textbf{Claim 1:} $\eta(\lambda)$ is increasing on $(0,+\infty)$.
To this end,
 differentiating with respect to $\lambda$, using the condition \eqref{Assump:g_3} along with Lemma \ref{tf(x,t)>0}, we obtain
\begin{align*}
\eta'(\lambda)&\geq \frac{\left({\iint_{Q}}g_{x,y}\left(\lambda D^{s}u\right)\left(\lambda D^{s}u\right)d\mu \right)}{\left({\iint_{Q}}g_{x,y}\left(\lambda D^{s}u\right)\left(\lambda D^{s}u\right)d\mu\right)^{2}}\times \left( \int_{\Omega}f'(x, \lambda u) \lambda u^{2}dx\right)\\
& -\frac{ \left( (g^{+}-1){\iint_{Q}}g_{x,y}\left(\lambda D^{s}u\right)\left( D^{s}u\right)d\mu\right) }{\left({\iint_{Q}}g_{x,y}\left(\lambda D^{s}u\right)\left(\lambda D^{s}u\right)d\mu\right)^{2}}\times\left(\int_{\Omega}f(x,\lambda u)\lambda u\ dx \right)\\
&\geq \frac{\left({\iint_{Q}}g_{x,y}\left(\lambda D^{s}u\right)\left(\lambda D^{s}u\right)d\mu\right)}{\lambda \left({\iint_{Q}}g_{x,y}\left(\lambda D^{s}u\right)\left(\lambda D^{s}u\right)d\mu\right)^{2}}\left[ \int_{\Omega}\left(f'(x,\lambda u)(\lambda u)^{2}-  (g^{+}-1) f(x,\lambda u)\lambda u\right)dx\right]\\
&>0,
\end{align*}
since, from the condition \eqref{Assump:f_3}, we know that
$$\left(f'(x,\lambda u)(\lambda u)^{2}- (g^{+}-1)f(x,\lambda u)\lambda u\right)>0.$$
This proves Claim $1$.\\

\textbf{Claim 2:} $\lim_{\lambda \to 0^{+}}\eta(\lambda)=0$ and $\lim_{\lambda \to +\infty}\eta(\lambda)=+\infty.$ Indeed, combining \eqref{Assump:g_3} and \eqref{assump:g_5} with Corollary \ref{bounds on f:f}, as well as Lemmas \ref{RE:Gt} and \ref{RELATION:MODULAR & NORM}, we obtain

\begin{align}
\label{eta:lim 0}
\eta(\lambda)&\leq \frac{\int_{\Omega} A h_{2}(x)|\lambda u|^{h_{2}(x)}dx}{g^{-}{\iint_{Q}}G_{x,y}\left(\lambda D^{s}u\right)d\mu},\nonumber\\
&\leq\frac{Ah_{2}^{+}\max\left\{\lambda^{h_{2}^{+}}, \lambda^{h_{2}^{-}}\right\}\int_{\Omega}|u|^{h_{2}(x)}}{g^{-} \min\left\{\lambda^{g^{-}}, \lambda^{g^{+}}\right\}\min\left\{ [u]_{s,G_{x,y}}^{g^{-}},[u]_{s,G_{x,y}}^{g^{+}} \right\}}\to 0 \text{ as } \lambda\to 0^{+},
\end{align}
and
\begin{align*}
    \quad \eta(\lambda)&\geq \frac{\int_{\Omega\cap\{|u|\geq\frac{1}{\lambda}\}}Bh_{1}(x)|\lambda u|^{h_{1}(x)}dx}{g^{+}{\iint_{Q}}G_{x,y}\left(\lambda D^{s}u\right)d\mu}\\
   &\geq \frac{Bh_{1}^{-}\min\left\{\lambda^{h_{1}^{+}}, \lambda^{h_{1}^{-}}\right\}\int_{\Omega \cap\{|u|\geq\frac{1}{\lambda}\}}|u|^{h_{1}(x)}} { g^{+}\max\left\{\lambda^{g^{-}}, \lambda^{g^{+}}\right\}\max\left\{ [u]_{s,G_{x,y}}^{g^{-}},[u]_{s,G_{x,y}}^{g^{+}} \right\}} \to \infty \text{ as } \lambda\to +\infty.
    \end{align*}
This proves Claim $2$.

From Claims 1 and 2, together with \eqref{definition of eta}, it follows that $\lambda(\delta, u) = \eta^{-1}(\delta)$, and that the mapping $\delta \mapsto \lambda(\delta, u)$ is increasing on $(0, +\infty)$. Moreover,

$$\lim_{\delta \to 0^{+}}\lambda(\delta,u)=0 \text{,          }\lim_{\delta \to +\infty}\lambda(\delta,u)=+\infty.$$
Since  $\lambda(\delta,u) u \in \mathcal{N_{\delta}}$, then from the definition of $d(\delta)$, we know  that $$d(\delta)\leq E(\lambda u).$$
Hence, using Lemma \ref{Lemma:3.3}(i), we conclude that
$$0\leq \lim_{\delta \to 0^{+}}d(\delta)\leq \lim_{\delta \to 0^{+}}E\left(\lambda(\delta,u)u\right)=0,$$
and
$$\lim_{\delta \to +\infty} d(\delta)\leq \lim_{\delta \to +\infty} E\left(\lambda(\delta,u) u\right)=-\infty.$$
This completes the proof of assertion (i).\\
$(ii)$ To establish assertion (ii), it suffices to prove that for any \(0 < \delta' < \delta'' < 1\) or \(\delta' > \delta'' > 1\), and for any \(u \in \mathcal{N}_{\delta''}\), there exists a \(v \in \mathcal{N}_{\delta'}\) and a constant \(\epsilon(\delta', \delta'') > 0\) such that
\[
E(u) - E(v) \geq \epsilon(\delta', \delta'').
\]
Indeed, for \(u \in \mathcal{N}_{\delta''}\), we have \(I_{\delta''}(u) = 0\), which implies that  \(\lambda(\delta'') = 1\). Using Lemma \ref{Upper estimate of norm}(ii), it follows that
\[
[u]_{s,G_{x,y}} \geq \min \left\{ \left( y(\delta'') \right)^{\frac{1}{h_2^{-} - g^{+}}}, \, \left( y(\delta'') \right)^{\frac{1}{h_2^{+} - g^{-}}} \right\}.\] By Lemma \ref{Lemma:3.3} (ii), there exists a constant \(\lambda(\delta') > 0\) such that \(v = \lambda(\delta') u \in \mathcal{N}_{\delta'}\).
Let \(g(\lambda) = E(\lambda u)\). Then,

\begin{align}\label{an}
    \frac{dg(\lambda)}{d\lambda}&=\frac{d}{d\lambda}\left[{\iint_{Q}}G_{x,y}\left(\lambda D^{s}u\right)d\mu-\int_{\Omega}F(x,\lambda u)dx \right] \nonumber\\
    &={\iint_{Q}}g_{x,y}\left(\lambda D^{s}u\right)\left(D^{s}u\right)d\mu-\int_{\Omega}f(x,\lambda u)u\ dx\nonumber\\
    &=\frac{1}{\lambda}\left[ {\iint_{Q}}g_{x,y}\left(\lambda D^{s}u\right)\left(\lambda D^{s}u\right)d\mu-\int_{\Omega}f(x,\lambda u)\lambda u\ dx\right]\nonumber\\
    &=\frac{1}{\lambda}\left[(1-\delta''){\iint_{Q}}g_{x,y}\left(\lambda D^{s} u \right)\left(\lambda D^{s}u\right)d\mu +  I_{\delta''}(\lambda u) \right]
        \end{align}
We consider two case:\\
    \textbf{Case $1$}: $0<\delta'<\delta''<1$. Since $\lambda$ is increasing and $\lambda(\delta'')=1$, then
    \begin{align*}
        E(u)-E(v)=g(1)-g(\lambda(\delta'))&=\int_{\lambda(\delta')}^{1}\frac{dg(\lambda)}{d\lambda}d\lambda.
    \end{align*}
 By Corollary \ref{I delta}, $I_{\delta''}(\lambda u) \geq 0$ for all $\lambda(\delta') < \lambda < 1$. Therefore, from \eqref{Assump:g_3}, \eqref{an}, Lemmas \ref{RE:Gt} and \ref{RELATION:MODULAR & NORM}, we get
    \begin{align*}
        E(u)-E(v) &\geq \int_{\lambda(\delta')}^{1}\frac{1}{\lambda}(1-\delta'')g^{-}{\iint_{Q}}G_{x,y}\left(\lambda D^{s}u\right)d\mu d\lambda \\
        &\geq g^{-}(1-\delta'' )\min\left\{ [u]_{s,G_{x,y}}^{g^{-}},[u]_{s,G_{x,y}}^{g^{+}} \right\} \int_{\lambda(\delta')}^{1}\min\left\{\lambda^{g^{-}-1}, \lambda^{g^{+}-1}\right\}d\lambda\\
        &\geq g^{-}(1-\delta'' )\min\left\{ [u]_{s,G_{x,y}}^{g^{-}},[u]_{s,G_{x,y}}^{g^{+}} \right\} \int_{\lambda(\delta')}^{1}\lambda^{g^{+}-1}d\lambda\\
   &\geq\frac{g^{-}}{g^{+}}\min\left\{ [u]_{s,G_{x,y}}^{g^{-}},[u]_{s,G_{x,y}}^{g^{+}} \right\}(1-\delta'')\left(1-\lambda(\delta'\right)^{g^{+}}).
    \end{align*}
    It follows, in light of Lemma \ref{Upper estimate of norm}(ii), that
$$
    E(u)-E(v)>\epsilon(\delta'',\delta'),
$$
where
$$\epsilon(\delta'',\delta')=\frac{g^{-}}{g^{+}}\min\left\{\left(y(\delta'')\right)^{\frac{g^{+}}{h_{2}^{-}-g^{+}}}, \left(y(\delta'')\right)^{\frac{g^{-}}{h_{2}^{+}-g^{-}}},\left(y(\delta'')\right)^{\frac{g^{-}}{h_{2}^{-}-g^{+}}}, \left(y(\delta'')\right)^{\frac{g^{+}}{h_{2}^{+}-g^{-}}}\right\}\alpha(\delta',\delta''),$$
and $$\alpha(\delta',\delta'')=(1-\delta'')(1-\lambda(\delta')^{g^{+}}).$$

    \textbf{Case 2}: \(\delta' > \delta'' > 1\). As $\lambda$ is increasing, it follows that  \(\lambda(\delta') > \lambda(\delta) > 1 = \lambda(\delta'')\).
    Hence, by repeating the same arguments as in \textbf{Case 1}, we obtain
    we get
    $$
        E(u)-E(v)>\epsilon(\delta',\delta''),
    $$
    where
   $$ \epsilon(\delta',\delta'')=\min\left\{\left(y(\delta'')\right)^{\frac{g^{+}}{h_{2}^{-}-g^{+}}}, \left(y(\delta'')\right)^{\frac{g^{-}}{h_{2}^{+}-g^{-}}},\left(y(\delta'')\right)^{\frac{g^{-}}{h_{2}^{-}-g^{+}}}, \left(y(\delta'')\right)^{\frac{g^{+}}{h_{2}^{+}-g^{-}}} \right\} \beta(\delta',\delta''),$$
   and
   $$\beta(\delta',\delta'')=(\delta''-1)(\lambda(\delta')^{g^{+}}-1).$$

Consequently, since \( d(\delta) \) is continuous, increasing for \( 0 < \delta \leq 1 \), and decreasing for \( \delta \geq 1 \), it attains its maximum value at \( \delta = 1 \), where \( d(1) = d \).
\end{proof}
\begin{lemma}
\label{Lemma: 3.6}
Suppose the assumptions of Lemma \ref{prop:d(delta)} are satisfied. Let \(u \in W_{0}^{s,G_{x,y}}(\Omega)\) with \(0 < E(u) < d\), and assume that \(\delta_1 < 1 < \delta_2\), where \(\delta_1\) and \(\delta_2\) satisfy the equation \(d(\delta) = E(u)\). Then, the sign of \(I_{\delta}(u)\) remains unchanged for \(\delta_1 < \delta < \delta_2\).
\end{lemma}
\begin{proof}
   Clearly, since  $E(u) > 0$,  it follows that  $[u]_{s,G_{x,y}} \neq 0$.
 If the sign of \(I_{\delta}(u)\) changes for \(\delta_1 < \delta < \delta_2\), there exists \(\overline{\delta} \in (\delta_1, \delta_2)\) such that \(I_{\overline{\delta}}(u) = 0\). By the definition of \(d(\delta)\), this implies \(E(u) \geq d(\overline{\delta})\). However, this contradicts the fact that \(E(u) = d(\delta_1) = d(\delta_2) < d(\overline{\delta})\), as established in Lemma~\ref{prop:d(delta)}(ii).
\end{proof}
\begin{definition}
\label{Def:weak solution}
A function \( u \in L^{\infty}(0,T;W_{0}^{s,G_{x,y}}(\Omega)) \) with \( u_{t} \in L^{2}(0,T;L^{2}(\Omega)) \) is said to be a weak solution of the problem \eqref{main:problem} if 
\begin{enumerate}
    \item \( u(\cdot,0) = u_{0} \) a.e in $\Omega$.
    \item for all $\phi \in W_{0}^{s,G_{x,y}}(\Omega)$ and for a.e. $t \in  [0,T]$ the following equality holds:
\begin{equation}
\label{definition of weak sol}
\begin{aligned}
    \left(u_{t},\phi\right)_{L^2(\Omega)} + \left(u,\phi\right)_{W_{0}^{s,G_{x,y}}(\Omega)}
    &= \left(f(x,u),\phi \right)_{L^{2}(\Omega)},
\end{aligned}
\end{equation}
where
\[
\left(u,\phi \right)_{W_{0}^{s,G_{x,y}}(\Omega)} = \iint_{Q} g_{x,y}\left(D^{s}u\right) (D^{s}\phi) \ d\mu.
\]
\item  $u \in C(0, T; L^\Phi(\Omega))$ and the following energy  relation is satisfied:
\begin{equation}
\label{sol: u}
\begin{aligned}
    \int_{0}^{t} \|u_t(\cdot, \tau)\|^{2}_{L^{2}(\Omega)} \, d\tau + E(u(\cdot,t))
    &\leq E(u_{0}), \quad \text{a.e. } t \in [0,T).
\end{aligned}
\end{equation}

\end{enumerate}
\end{definition}

\begin{definition}
\label{def strong solution}
A  weak solution $u$ of the problem \eqref{main:problem} is said to be a strong solution if the following energy conservation law holds
\begin{equation}
    \label{strong solution}
\int_{0}^{t} \|u_t(\cdot, \tau)\|^{2}_{L^{2}(\Omega)} \, d\tau + E(u(\cdot,t))
    = E(u(\cdot,0)),
\end{equation}
for any time interval $[0,t]\subset [0,T).$
\end{definition}
\begin{definition}[Maximal existence time]
    Let \( u \) be a local or strong solution of the problem \eqref{main:problem}. We define the maximal existence time \( T_{\max} \) of \( u \) as follows:
    \begin{enumerate}
        \item If \( u \) exists for all \( 0 \leq t < +\infty \), then \( T_{\max} = +\infty \).
        \item If there exists \( t_0 \in (0, +\infty) \) such that \( u \) exists for all \( t \in (0, t_0) \) but does not exist at \( t = t_0 \) in the sense that 
        \[
        \|u(\cdot, t)\|_{W_{0}^{s,G_{x,y}}(\Omega)} \to +\infty \quad \text{as} \quad t \to t_0^-,
        \]
        then \( T_{\max} = t_0 \).
    \end{enumerate}
\end{definition}
\begin{lemma}
\label{continuity of E(u) with t}
    Let the conditions \eqref{assump:f_0}--\eqref{Assump:f_2}, \eqref{cond:g0}--\eqref{Assump:g_3}, and \eqref{Cond:g_6} hold, and let $u$ be a global strong solution of problem \eqref{main:problem} in the sense of Definition \eqref{def strong solution}. Then, we have $u\in C(0,\infty;W_{0}^{s,G_{x,y}}(\Omega))$, and the mappings
    \[
    t\mapsto E(u(\cdot,t))\qquad  \text{and} \qquad t\mapsto I(u(\cdot,t))
    \]
    are continuous on $[0,\infty)$.
\end{lemma}

\begin{proof}
Let $u$ be a global strong solution of problem \eqref{main:problem}. Using Lemma \ref{F:bounds}(i), condition \eqref{Cond:g_6}, and the Lebesgue dominated convergence theorem, it is easy to show that for any $t_1 \in [0, \infty)$, we have
\begin{align}
    \label{F(x,u(t)) is cont.}
    &\lim_{t\to t_1}\int_{\Omega}F(x,u(x,t))\,dx=\int_{\Omega}F(x,u(x,t_1))\,dx
\end{align}
and 
\begin{align}
    \label{f(x,u(t))u is cont.}
    &\lim_{t \to t_1} \int_{\Omega}f(x,u(x,t)) u(x,t)\,dx = \int_{\Omega} f(x,u(x,t_1)) u(x,t_1)\,dx.
\end{align}
Let $\epsilon>0$ and choose $\delta = \epsilon \left(\|u_t\|_{L^{2}(0,\infty;L^2{(\Omega)})}\right)^{-1}$ such that $|t - t_{1}| < \delta$. Then, from \eqref{strong solution}, we obtain
\begin{align}
   \left|E(u(\cdot,t))- E(u(\cdot,t_1))\right| &=\left|\int_{0}^{t} \|u_t(\cdot, \tau)\|^{2}_{L^{2}(\Omega)} \, d\tau - \int_{0}^{t_1} \|u_t(\cdot, \tau)\|^{2}_{L^{2}(\Omega)} \, d\tau\right| \nonumber\\
   &=\left|\int_{t_1}^{t}\|u_t(\cdot, \tau)\|^{2}_{L^{2}(\Omega)} \, d\tau\right|\nonumber\\
   &\leq \|u_t\|_{L^{2}(0,\infty;L^2{(\Omega)})} \left|t - t_{1}\right| < \epsilon. \nonumber
\end{align}
Therefore, the map $t \mapsto E(u(\cdot,t))$ is continuous for all $t\in[0,\infty)$. Now, using the continuity of $E(u(\cdot,t))$ and \eqref{F(x,u(t)) is cont.}, we obtain
\[
\lim_{t \to t_1}\iint_{Q}G_{x,y}\left(D^{s}u(t)\right)d\mu=\iint_{Q}G_{x,y}\left(D^{s}u(t_1)\right)d\mu.
\]
This implies that $u\in C(0,\infty;W_{0}^{s,G_{x,y}}(\Omega))$ and
\begin{align}
\label{u(tn) strong conv}
D^{s}(u(\cdot,t)) \to D^{s}(u(\cdot,t_1)) \quad\text{in }\,L^{\widehat{G}_x}(Q), \quad \text{where} \quad D^{s}(u(\cdot,t)):= \frac{u(x,t)-u(y,t)}{|x-y|^s}.
\end{align}
From \eqref{u(tn) strong conv} and \cite[Theorem 4.9]{Brezis}, there exists a function $h \in L^{\widehat{G}_x}(Q)$ such that
\[
D^{s}(u(\cdot,t))\to D^{s}(u(\cdot,t_1)) \quad \text{as} \quad t \to t_1, \quad \text{for a.e. in } Q,
\]
and 
\[
\left|D^{s}(u(\cdot,t))\right|\leq  h(\cdot), \quad \text{a.e. in } Q.
\]
Since the function $t \mapsto g_{x,y}(t)$ is $C^{1}(0,\infty)$ and increasing, we obtain
\[
g_{x,y}(D^{s}(u(\cdot,t)))D^{s}(u(\cdot,t))\to g_{x,y}(D^{s}(u(\cdot,t_1)))D^{s}(u(\cdot,t_1)) \quad \text{a.e. in } Q,
\]
and 
\[
\left|g_{x,y}(D^{s}(u(\cdot,t)))D^{s}(u(\cdot,t))\right| \leq |g_{x,y}(h(x))h(x)|, \quad \text{a.e. in } Q.
\]
Applying the Lebesgue dominated convergence theorem, we obtain 
\begin{align}
\label{g{x,y}}
\lim_{t \to t_1}\iint_{Q}g_{x,y}(D^{s}(u(\cdot,t)))D^{s}(u(\cdot,t)) d\mu = \iint_{Q}g_{x,y}(D^{s}(u(\cdot,t_1)))D^{s}(u(\cdot,t_1))\, d\mu.
\end{align}
Combining \eqref{f(x,u(t))u is cont.} and \eqref{g{x,y}}, we conclude that
\[
\lim_{t \to t_1}I(u(\cdot,t))= I(u(\cdot,t_1)).
\]
Hence, the proof follows.
\end{proof}

\begin{lemma}
\label{u in W delta}
Assume that the conditions \eqref{assump:f_0}--\eqref{Assump:f_3} and \eqref{cond:g0}--\eqref{Assump:g_3} hold. Let \(u\) be a global  strong solution of problem \eqref{main:problem}  in the sense of Definition \ref{def strong solution}, with initial data \(u_0\) satisfying \(E(u_0) < d\) and \(E(u_0) = d(\delta_1) = d(\delta_2)\) for some \(\delta_1 < 1 < \delta_2\).
\begin{enumerate}
    \item[$(i)$] If \(I(u_0) > 0\), then \( u(\cdot,t) \in W_{\delta}\) for \(\delta_1 < \delta < \delta_2\)  and for \(0 < t < \infty \).
    \item[$(ii)$] If \(I(u_0) < 0\), then \(u(\cdot,t) \in V_{\delta}\) for \(\delta_1 < \delta < \delta_2\)  and for \(0 < t < \infty\).
\end{enumerate}
\end{lemma}

\begin{proof}
$(i)$ For \( 0 < E(u_{0}) = d(\delta_{1}) = d(\delta_{2}) < d \) and \( I(u_{0}) > 0 \), it follows from Lemma~\ref{Lemma: 3.6} that \( u_{0} \in W_{\delta} \) for \( \delta_{1} < \delta < \delta_{2} \) and \( \delta_1 < 1 < \delta_2 \). Next, we prove that \( u(\cdot,t) \in W_{\delta} \) for \( \delta_{1} < \delta < \delta_{2} \) and \( 0 < t < \infty \). Assume the contrary: there exist \(t_{0} \in (0,\infty) \) and \( \delta_{0} \in (\delta_{1},\delta_{2}) \) such that
\( u(\cdot, t_{0}) \notin  W_{\delta_{0}} \) and \(u(\cdot, t_{0}) \not\equiv 0\). If \(t_0\) is not unique, without loss of generality, we assume \(t_0\) is the first time such that \( u(\cdot, t_{0}) \notin W_{\delta_{0}} \). Then we have \( u(\cdot, t) \not\equiv 0 \) for all \(t \in (0, t_0]\) and
\[
I_{\delta_{0}}(u(\cdot, t_{0})) = 0\text{ or } I_{\delta_{0}}(u(\cdot, t_{0})) < 0  \text{ or } E(u(\cdot, t_{0})) = d(\delta_{0}) \text{ or } E(u(\cdot, t_{0})) > d(\delta_{0})  .
\]
Clearly, from \eqref{sol: u} and Lemma \ref{prop:d(delta)}, we know \( E(u(\cdot, t_{0})) < E(u_0) = d(\delta_1)=d(\delta_2) \leq d(\delta_{0}) \). Thus, either \(I_{\delta_{0}}(u(\cdot, t_{0})) = 0\) or \(I_{\delta_{0}}(u(\cdot, t_{0})) < 0\). Suppose \(I_{\delta_{0}}(u(\cdot, t_{0})) < 0\). 
Since \(I_{\delta_{0}}(u_{0})>0 \) and \(I_{\delta_{0}}(u(\cdot, t_{0}))< 0 \), and using  Lemma \ref{continuity of E(u) with t}, there exists \(t_{1} \in (0, t_0)\) such that \(I_{\delta_{0}}(u(\cdot, t_{1}))=0\).
Thus, \( I_{\delta_{0}}(u(\cdot, t_1)) = 0 \) and \( [u(\cdot, t_1)]_{s,G_{x,y}} \neq 0 \). This implies, from the definition of \( d(\delta_{0}) \), that \( E(u(\cdot, t_{1})) \geq d(\delta_{0}) \), which contradicts \eqref{sol: u}. A similar contradiction arises if \(I_{\delta_{0}}(u(\cdot, t_{0})) = 0\).

$(ii)$ Following the above argument, we have \( u_{0} \in V_{\delta} \) for \( \delta_{1} < \delta < \delta_{2} \). Next, we prove that \( u(\cdot,t) \in V_{\delta} \) for \( \delta_{1} < \delta < \delta_{2} \) and \( 0 < t < \infty \).
Assume the contrary: there exist \( t_{0} \in (0,\infty) \) and \( \delta_{0} \in (\delta_{1},\delta_{2}) \) such that
\( u(\cdot, t_{0}) \notin V_{\delta_{0}} \). If \(t_0\) is not unique, without loss of generality, we assume \(t_0\) is the first time such that \( u(\cdot, t_{0}) \notin V_{\delta_{0}} \). Then we have four possible cases:
\[
I_{\delta_{0}}(u(\cdot, t_{0})) = 0\text{ or } I_{\delta_{0}}(u(\cdot, t_{0})) > 0  \text{ or } E(u(\cdot, t_{0})) = d(\delta_{0}) \text{ or } E(u(\cdot, t_{0})) > d(\delta_{0})  .
\]
Clearly, from \eqref{sol: u} and Lemma \ref{prop:d(delta)}, we know \( E(u(\cdot, t_{0})) < d(\delta_{0}) \). Thus, either \(I_{\delta_{0}}(u(\cdot, t_{0})) = 0\) or \(I_{\delta_{0}}(u(\cdot, t_{0})) > 0.\)
Suppose \( I_{\delta_{0}}(u(\cdot, t_{0})) = 0 \), then it follows that \( I_{\delta_{0}}(u(\cdot, t)) < 0 \) for all \( 0 \leq t < t_{0} \).  Therefore,
by Lemma~\ref{Upper estimate of norm}(iii), we obtain
\[
\|u(\cdot,t)\|_{L^{\Phi}(\Omega)} > \min\left\{\left(z(\delta_{0})\right)^{\frac{1}{h_{2}^{-}-g^{+}}}, \left(z(\delta_{0})\right)^{\frac{1}{h_{2}^{+}-g^{-}}}\right\} >0, \quad \text{for } 0 \leq t < t_{0}.
\]
Since \(u \in C(0,\infty; L^\Phi(\Omega))\), this implies
 \[
\|u(\cdot,t_{0})\|_{L^{\Phi}(\Omega)} \neq 0  \quad \textit{i.e.} \quad u(\cdot,t_{0})\not \equiv 0.
\]
Combining this with \( I_{\delta_{0}}(u(\cdot, t_{0})) = 0 \), we deduce that \( u(\cdot, t_{0}) \in \mathcal{N}_{\delta_{0}} \), which contradicts \eqref{sol: u}. A similar contradiction arises if \( I_{\delta_{0}}(u(\cdot, t_{0}))>0\). This ends the proof.
\end{proof}
\section{Local existence of strong solution}
In this section, we examine the well-posedness of problem \eqref{main:problem} and establish the existence of a local strong solution. Before presenting the main result, we introduce some operators and functionals.  

Let \( H \) be a Hilbert space with inner product \( (\cdot,\cdot) \) and norm \( \|\cdot\|_{H} \). For a functional \( \varphi:H\to (-\infty,+\infty] \), we define the sublevel set as  
\[
D(\varphi,r):=\left\{u\in H \,:\, \varphi(u)\leq r \right\}, \quad \text{for } r\in \mathbb{R},  
\]
and the domain of \( \varphi \) as  
\[
D(\varphi)= \bigcup_{r\in \mathbb{R}} D(\varphi,r).
\]

\begin{definition}
    Let \( \varphi:H\to (-\infty,+\infty] \) be a functional. The subdifferential \( \partial\varphi \) of \( \varphi \) is defined as  
    \[
    \partial\varphi(u)=\left\{ f\in H \,: \, \varphi(v)-\varphi(u)\geq (f,v-u), \ \forall v\in H \right\}
    \]
    for any \( u \in H \).
\end{definition}
It is well known that  the subdifferential $\partial\varphi$ is maximal monotone operator and its domain satisfies $D(\partial \varphi)\subset D(\varphi)$.
Next, we recall the chain rule for subdifferentials (see, \cite[Lemma 3.3, p.73]{Heim}) and a key existence result (see,  \cite[Theorem 3.4, p. 297]{Ishii}).
\begin{lemma}
\label{continuity of strong solution}
    Let \( \varphi:H \to (-\infty,+\infty] \) be a proper, convex, and lower semicontinuous functional. Suppose that \( u \in W^{1,1}(0,T;H) \) and that \( u(\cdot,t) \in D(\partial\varphi) \) for almost every \( t \in [0,T] \), where \( T>0 \).  

    If there exists a function \( g \in L^{2}(0,T;H) \) such that  
    \[
    g(\cdot,t) \in \partial\varphi(u(\cdot,t)) \quad \text{for a.e. } t \in [0,T],
    \]  
    then the function \( t \mapsto \varphi(u(\cdot,t)) \) is absolutely continuous on \( [0,T] \) and satisfies  
    \[
    \frac{d}{dt} \varphi(u(\cdot,t)) = \left( g(\cdot,t), \frac{du(\cdot,t)}{dt} \right) \quad \text{for a.e. } t \in [0,T].
    \]
\end{lemma}
\begin{theorem}
\label{wellposedness}
    Let \( \varphi:H \to (-\infty,+\infty] \) be a proper, convex, and lower semicontinuous functional, and let \( \psi: H \to \mathbb{R} \) be another functional. Assume that the following conditions hold:
    \begin{itemize}
        \item[$(i)$] For any \( r \in \mathbb{R} \), the sublevel set \( D(\varphi,r) \) is compact in \( H \).
        \item[$(ii)$] \( D(\varphi) \subset D(\psi) \).
        \item[$(iii)$] The set  
        \[
        \left\{(\partial\psi)^{0}(u) \mid u \in D(\varphi,r) \right\}
        \]
        is bounded in \( H \) for any \( r \in \mathbb{R} \), where \( (\partial\psi)^{0} \) denotes the element of minimal norm in the subdifferential \( \partial\psi(u) \).
    \end{itemize}
    
    Then, for each initial condition \( h \in D(\varphi) \), there exist \( T>0 \) and a strong solution \( u \) to the initial value problem:
    \begin{equation*}
    \begin{cases}
        \frac{du(\cdot, t)}{dt} + \partial\varphi(u(\cdot, t)) - \partial\psi(u(\cdot, t)) \ni 0, & \text{in } H, \quad t \in (0,T), \\[5pt]
        u(\cdot,0) = h(\cdot), & \text{in } \Omega.
    \end{cases}
    \end{equation*}
\end{theorem}

Now, we reformulate the system \eqref{main:problem} as a Cauchy problem for an abstract evolution equation in the Hilbert space \( H = L^{2}(\Omega) \). To achieve this, we define the functionals \( \varphi: H \to (-\infty,+\infty] \) and \( \psi: H \to (-\infty,+\infty] \) as follows:
\[
\varphi(u) =
\begin{cases}
    \displaystyle \iint_{Q} G_{x,y}\left(D^{s}u\right)\,d\mu, & \text{if } u\in W_{0}^{s,G_{x,y}}(\Omega),\\[5pt]
    +\infty, & \text{if } u\in H\setminus W_{0}^{s,G_{x,y}}(\Omega).
\end{cases}
\]
and
\[
\psi(u) =
\begin{cases}
    \displaystyle \int_{\Omega}F(x,u)\,dx, & \text{if } u\in L^{\Phi}(\Omega),\\[5pt]
     +\infty, & \text{if } u\in H\setminus L^{\Phi}(\Omega).
\end{cases}
\]
It is straightforward to verify that the functionals \( \varphi \) and \( \psi \) are proper, convex, and lower semicontinuous.

Next, we consider the operator \( \mathcal{L}: W_{0}^{s, G_{x,y}}(\Omega) \to \left(W_{0}^{s, G_{x,y}}(\Omega)\right)^{\ast} \) defined by
\[
(\mathcal{L}(u),v) := \iint_{Q} g_{x,y}\left(D^{s}u\right) (D^{s}v) \, d\mu, \qquad \forall \, u,v \in W_{0}^{s, G_{x,y}}(\Omega),
\]
where \( \left(W_{0}^{s, G_{x,y}}(\Omega)\right)^{\ast} \) denotes the dual space of \( W_{0}^{s, G_{x,y}}(\Omega) \).

According to \cite[Example 2.3.7, p.26]{Heim}, the operator \( \mathcal{L}_{H} \), which is the realization of \( \mathcal{L} \) in \( H = L^{2}(\Omega) \), is defined as
\begin{equation*}
 D(\mathcal{L}_{H}) = \left\{ u\in W_{0}^{s, G_{x,y}}(\Omega) \mid \mathcal{L}(u)\in H \right\}, 
\quad \text{with} \quad \mathcal{L}_{H}(u) = \mathcal{L}(u), \quad \forall u \in D(\mathcal{L}_{H}).
\end{equation*}
It follows that \( \mathcal{L}_{H} \) is a maximal monotone operator.
      \begin{lemma}
\label{operator and subdifferential}
Assume that the conditions \eqref{cond:g0}-\eqref{Assump:g_3} and \eqref{assump:f_0}-\eqref{Assump:f_2} hold. Then,
\[
\partial\varphi(u) = (-\Delta)^{s}_{g_{x,y}} u, \quad \text{and} \quad \partial\psi(u) = f(x,u).
\]
\end{lemma}

\begin{proof}
Both \( \mathcal{L}_{H} \) and \( \partial\varphi \) are maximal monotone operators in \( H \). Therefore, it suffices to show that \( \mathcal{L}_{H}(u) \subset \partial\varphi(u) \). 

Let \( u \in D(\mathcal{L}_{H}) \). Then, for all \( w \in W_{0}^{s, G_{x,y}}(\Omega) \), we have
\begin{align}
\label{convexity of G}
(\mathcal{L}_{H}(u), w - u) &= (\mathcal{L}(u), w - u) \nonumber \\
&= \iint_{Q} g_{x,y}\left(D^{s}u\right) D^{s}(w - u) \, d\mu \nonumber \\
&= \iint_{Q} g_{x,y}\left(D^{s}u\right) D^{s}w \, d\mu - \iint_{Q} g_{x,y}\left(D^{s}u\right) D^{s}u \, d\mu.
\end{align}
Using the convexity of \( G_{x,y}(\cdot) \), we obtain
\begin{equation*}
G_{x,y}(t_1) - G_{x,y}(t_2) \geq g_{x,y}(t_2)(t_1 - t_2), \quad \forall \ t_1, t_2 \in [0,\infty).
\end{equation*}
Setting \( t_1 = D^{s}w \) and \( t_2 = D^{s}u \), we get
\begin{align}
\label{g is convex}
G_{x,y}(D^{s}w) - G_{x,y}(D^{s}u) \geq g_{x,y}(D^{s}u) D^{s}w - g_{x,y}(D^{s}u) D^{s}u.
\end{align}
Combining \eqref{convexity of G} and \eqref{g is convex}, we obtain
\begin{align}
\label{v in subdifferential}
(\mathcal{L}_{H}(u), w - u) &\leq \iint_{Q}  G_{x,y}(D^{s}w) \, d\mu - \iint_{Q} G_{x,y}(D^{s}u) \, d\mu \nonumber \\
&= \varphi(w) - \varphi(u).
\end{align}
If \( w \in H \setminus W_{0}^{s,G_{x,y}}(\Omega) \), then \( \varphi(w) = +\infty \), and thus \eqref{v in subdifferential} still holds trivially. This proves that \( \mathcal{L}_{H}(u) \in \partial\varphi(u) \), which implies \( \mathcal{L}_{H}(u) = \partial\varphi(u) \).\\
Similarly, by exploiting the convexity of the mapping \( t \mapsto F(x,t) \), we can show that \( \partial\psi(u) = f(x,u) \).
\end{proof}

Now, in view of Lemma \ref{operator and subdifferential}, the system \eqref{main:problem} can be rewritten as the following abstract Cauchy problem:
\begin{equation}
\label{subdifferenial main prob}
\left\{
\begin{aligned}
    \frac{du}{dt} + \partial\varphi(u)-\partial\psi(u) &\ni 0,  && \text{in } H,\, 0<t<T, \\
    u(\cdot,0) &= u_{0}(\cdot) && \text{in } \Omega.
\end{aligned}
\right.
\end{equation}
Next, we assume the following condition:
\begin{enumerate}[leftmargin=1.5cm,label=\textnormal{($\tilde{g}_6$)},ref=\textnormal{$\tilde{g}_6$}]
    \item \label{assump:h_0}
    \( h_{2}^{+} \leq \frac{g^{-}_{\ast,s}}{2} + 1. \)
\end{enumerate}
\begin{remark}\label{rrem2}
We consider the same example as in Remark \ref{rrem1}, with  
\[
q_2^+ < \frac{N p^-}{2(N - s p^-)} + 1.
\]
Thus, the functions \( f \) and \( g \) satisfy the conditions \eqref{assump:f_0}-\eqref{Assump:f_3} and \eqref{cond:g0}-\eqref{assump:h_0}, respectively.
\end{remark}

\begin{lemma}
\label{D(phi,r) is compact}
    Let the conditions \eqref{cond:g0}-\eqref{assump:g_5}, \eqref{assump:h_0}, and \eqref{assump:f_0}-\eqref{Assump:f_2} hold. Then, the set \( D(\varphi,r) \) is compact in \( L^{2}(\Omega) \) for any \( r\in\mathbb{R} \). Moreover, \( D(\varphi) \subset D(\psi) \).
\end{lemma}

\begin{proof}
    From assumptions \eqref{assump:g_5}-\eqref{assump:h_0} and Theorem \ref{thm2}, it follows that \( W_{0}^{s,G_{x,y}}(\Omega) \) is compactly embedded in \( L^{2}(\Omega) \) and continuously embedded in \( L^{\Phi}(\Omega) \). Hence, \( D(\varphi) \subset D(\psi) \), and by Lemma \ref{RELATION:MODULAR & NORM}, the set \( D(\varphi,r) \) is compact in \( L^{2}(\Omega) \) for any \( r\in\mathbb{R} \).
\end{proof}

\begin{lemma}
\label{del phi is bounded}
   Let the conditions \eqref{cond:g0}-\eqref{assump:g_4}, \eqref{assump:f_0}-\eqref{Assump:f_2}, and \eqref{assump:h_0} hold. Then,
   \[
   \left\{(\partial\psi)^{0}(u) \mid u\in D(\varphi,r)\right\} = \{f(\cdot,u)\}
   \]
   for any \( r\in \mathbb{R} \), and \( f(\cdot, u) \in L^{2}(\Omega) \).
\end{lemma}

\begin{proof}
    It is enough to consider \( r>0 \). Since \( \psi \in C^{1}(L^{\Phi}(\Omega), \mathbb{R}) \), there exists a unique \( f_{u} \in L^{2}(\Omega) \) such that \( \partial\psi(u) = \{ f_{u} \} \), and so, \( (\partial\psi)^{0}(u) = \{ f_{u} \} \). Moreover,
    \[
    \int_{\Omega} f(x,u) v \, dx = (f_{u},v)_{L^{2}(\Omega)}, \quad \text{for all } v\in C_{0}^{\infty}(\Omega).
    \]
    The above equality implies \( f(x,u) = f_{u} \) a.e. in \( \Omega \). Now, using Lemma \ref{F:bounds}(ii) and applying H\"older's inequality, we obtain
    \begin{align}
    \label{bound of least norm sub}
         \left| (f_{u},v)_{L^{2}(\Omega)}\right| &\leq A h_{2}^{+} \int_{\Omega} |u|^{h_{2}(x)-1} |v(x)|\,dx \nonumber\\
        &\leq A h_{2}^{+} \||u|^{h_{2}(x)-1}\|_{L^{2}(\Omega)} \|v\|_{L^{2}(\Omega)}.
    \end{align}
    From \eqref{assump:h_0}, we have 
    \begin{align}
    \label{bound of least norm sub-1}
        \int_{\Omega} |u|^{2(h_{2}(x)-1)}\,dx &= \int_{\Omega \cap \{|u| \leq 1\}} |u|^{2(h_{2}(x)-1)}\,dx + \int_{\Omega \cap \{|u| > 1\}} |u|^{2(h_{2}(x)-1)}\,dx \nonumber \\
        &\leq |\Omega| + \int_{\Omega \cap \{|u| > 1\}} |u|^{g^{-}_{\ast,s}}\,dx \nonumber \\
        &\leq |\Omega| + C [u]_{s,G_{x,y}}^{g^{-}_{\ast,s}},
    \end{align}
    where \( C>0 \) is the embedding constant. Since \( u\in D(\varphi,r) \), there exists a constant \( C_{1} = C_{1}(r) > 0 \) such that \( [u]_{s,G_{x,y}}^{g^{-}_{\ast,s}} \leq C_1(r) \). Now, combining \eqref{bound of least norm sub} and \eqref{bound of least norm sub-1}, we obtain
    \[
    \|f_{u}\|_{L^{2}(\Omega)} = \sup_{\|v\|_{L^{2}(\Omega)}\leq 1} \left| (f_{u},v)_{L^{2}(\Omega)}\right| \leq C_{r},
    \]
    where \( C_{r} = A h_2^+ \left(|\Omega| + C C_1(r)\right)^{\frac{1}{2}} \).
\end{proof}

\begin{theorem}\label{loc}
    Let \( u_{0}\in W_{0}^{s,G_{x,y}}(\Omega) \), and let the conditions \eqref{cond:g0}-\eqref{assump:g_5}, \eqref{assump:f_0}-\eqref{Assump:f_2}, and \eqref{assump:h_0} hold. Then, there exists a \( T>0 \) such that the problem \eqref{main:problem} admits a strong solution \( u \) on \( \Omega\times[0,T] \) in the sense of Definition \ref{def strong solution}.
\end{theorem}

\begin{proof}
    The existence of a strong solution \( u \) follows directly from Theorem \ref{wellposedness}, combined with Lemmas \ref{D(phi,r) is compact} and \ref{del phi is bounded}. Since \( u \) satisfies \eqref{v in subdifferential}, it follows that  
\[ 
g = f(x,u) - u_t \in \partial \varphi(u).
\]  
Therefore, by setting \( g = f(\cdot,u) - u_t \) in Lemma \ref{continuity of strong solution}, we obtain  
\[
u \in C([0,T];W_{0}^{s,G_{x,y}}(\Omega)).
\]
Moreover, the energy identity holds:
\[
\int_{0}^{t} \|u_{t}(\cdot,\tau)\|_{L^{2}(\Omega)}^{2} + E(u(\cdot,t)) = E(u(\cdot,0)), \qquad \forall\, t\in [0,T].
\]
\end{proof}

\section{Global existence of solutions and finite time blow up}
In this section, we study the existence of global weak solutions to equation \eqref{main:problem} and the phenomenon of finite-time blow-up of strong solution. We consider three cases: low initial energy \(\big(E(u_0) < d\big)\), critical initial energy \(\big(E(u_0) = d\big)\), and high initial energy \(\big(E(u_0) > d\big)\).
\subsection{On the Case of Low Initial Energy:  $E(u_0)<d$}
In this subsection, we analyze the global existence and blow-up of weak solutions under the condition \(E(u_0) < d\). Specifically, we demonstrate that if \(I(u_0) > 0\), problem~\eqref{main:problem} admits a global weak solution. Moreover, if \(I(u_0) < 0\), we establish the finite-time blow-up of a strong solution of problem~\eqref{main:problem}, meaning that the strong solution ceases to exist in finite time.\\
Now, in order to construct Galerkin's approximation scheme, we assume the following condition:
\begin{enumerate}[leftmargin=1.5cm, label=\textnormal{($g_7$)}, ref=\textnormal{$g_7$}]
    \item \label{Cond:g_7} The function \( G_{x,y} \) belongs to \( \mathcal{B}_{f} \), and both \( G_{x,y} \) and \( \widehat{G}_{x} \) are locally integrable. That is, for any \( t > 0 \) and every compact set \( A \subset \Omega \), we have  
    \[
    \int_{A} \int_{A} G_{x,y}(t) \,dx\,dy < \infty, \qquad \text{and} \qquad \int_{A} \widehat{G}_{x}(t)\,dx < \infty.
    \]
\end{enumerate}

\begin{enumerate}[leftmargin=1.5cm, label=\textnormal{($g_8$)}, ref=\textnormal{$g_8$}]
    \item \label{Cond:g_8} The function \( G(x-z, y-z, t) \) satisfies the translation invariance property:  
    \[
    G(x-z, y-z, t) = G(x,y,t), \quad \forall \ (x,y), (z,z) \in \Omega \times \Omega, \ \forall \ t \geq 0.
    \]
\end{enumerate}

\begin{remark}
We consider the same example as in Remark \ref{rrem1}, with  
$$p(x,y)=Q(|x-y|),$$
where $Q:\mathbb{R}\to \mathbb{R}$ is a continuous function. Thus, the functions \( f \) and \( g \) satisfy the conditions \eqref{assump:f_0}-\eqref{Assump:f_3} and \eqref{cond:g0}-\eqref{Cond:g_8}, respectively.
\end{remark}

Set $\{e_i\}$ and $\{\lambda_i\}$ be the eigenfunctions and the corresponding eigenvalues
of the Dirichlet problem for the Laplacian
such that $\{e_i\}$ forms an orthogonal basis of $L^2(\Omega).$ Since $\Omega$ is a smooth bounded domain, following the arguments in \cite[Section 5]{Arora-Shmarev-2023-2} and using the fact that $C_0^\infty(\Omega)$ is dense in $W_{0}^{s,G_{x,y}}(\Omega)$ (see \cite[Theorem 1.7]{M18}), we have 
\[
\bigcup_{m=1}^\infty \mathcal{P}_m \ \text{is dense in} \ W_{0}^{s,G_{x,y}}(\Omega),
\ \text{and} \ 
\bigcup_{m=1}^\infty \mathcal{N}_m \ \text{is dense in} \ L^{1}(0,\infty;W_{0}^{s,G_{x,y}}(\Omega)) \cap L^2(0,\infty; L^2(\Omega))
\]
where 
\[
\mathcal{P}_m:= \{e_1, e_2, \dots, e_m\} \quad \text{and} \quad \mathcal{N}_m := \{w(x,t) : w(x,t) = \sum_{i=1}^m \theta_i(t) e_i(x), \theta_i \in C^{0,1}[0,\infty)\}.
\]

  \begin{theorem}\label{thm:global-existence-weak}
     Let the conditions \eqref{assump:f_0}--\eqref{Assump:f_3} and \eqref{cond:g0}--\eqref{Cond:g_8} hold, and assume that \( u_0 \in W_{0}^{s,G_{x,y}}(\Omega) \). If \( E(u_0) < d \) and \( I(u_0) > 0 \), then problem \eqref{main:problem} admits a global weak solution  
\[
u \in L^\infty(0, \infty; W_{0}^{s,G_{x,y}}(\Omega)), \quad u_t \in L^2(0, \infty; L^2(\Omega)),
\]
such that \( u(\cdot, t) \in W \) for all \( 0 \leq t < \infty \). Moreover, the weak solution is unique if it is bounded.
\end{theorem}

\begin{proof} Consider the Galerkin approximations:
\begin{equation*}
u^{(n)}(x, t) = \sum_{j=1}^n c_j^{(n)}(t) e_j(x),
\end{equation*}
where the functions $c_k^{(n)}(t): [0, T] \to \mathbb{R}$ satisfy the following system of ordinary differential equations for $j = 1, 2, \dots, n$:
\begin{equation}
\label{Galerkin:approximate}
\left\{
\begin{aligned}
    \left(\frac{du^{(n)}}{dt}, e_j\right)_{L^2(\Omega)} + \left(u^{(n)}, e_j\right)_{W_{0}^{s,G_{x,y}}(\Omega)} &= \left(f(x, u^{(n)}), e_j\right)_{L^2(\Omega)}, \\
    u^{(n)}(x, 0) &= \sum_{i=1}^n \left(u(x, 0), e_i\right)_{L^2(\Omega)} e_i(x),
\end{aligned}
\right.
\end{equation}
with
\begin{equation*}
\left(u^{(n)}, e_j\right)_{W_{0}^{s,G_{x,y}}(\Omega)} = \iint_Q g_{x,y}\left(\frac{u^{(n)}(x,t) - u^{(n)}(y,t)}{|x-y|^s}\right) \frac{e_j(x) - e_j(y)}{|x-y|^s} \, d\mu,
\end{equation*}
and the initial condition
\begin{equation*}
 u^{(n)}(x, 0) \to u_0 \quad \text{in } \ W_{0}^{s,G_{x,y}}(\Omega), \quad \text{as } n \to \infty.
\end{equation*}
It follows from \eqref{Galerkin:approximate} that
\begin{align}\label{aneq}
   \frac{dc^{(n)}_{j}}{dt}(t) &= -\iint_{Q} g_{x,y}\left(\frac{u^{(n)}(x,t) - u^{(n)}(y,t)}{|x-y|^s}\right) \frac{e_j(x) - e_j(y)}{|x-y|^s} \, d\mu \\
   &\qquad\qquad + \left(f(x, u^{(n)}), e_j\right)_{L^2(\Omega)}. \nonumber
\end{align}
Define
\begin{align*}
F_j^{(n)}(c^{(n)}(t)) &= -\iint_{Q} g_{x,y}\left(\frac{\sum_{i=1}^n c_i^{(n)}(t)e_i(x) - \sum_{i=1}^n c_i^{(n)}(t)e_i(y)}{|x-y|^s}\right) \frac{e_j(x) - e_j(y)}{|x-y|^s} \, d\mu \\
&\qquad\qquad + \int_{\Omega} f\left(x, \sum_{i=1}^n c_i^{(n)}(t)e_i(x)\right) e_j(x) \, dx,
\end{align*}
and
\[
c^{(n)}(\cdot) = \left(c_i^{(n)}(\cdot)\right)_{i=1}^n, \quad F^{(n)}(\cdot) = \left(F_i^{(n)}(\cdot)\right)_{i=1}^n, \quad c^{(n)}_0 = \left(\left(u_0, e_i\right)_{L^2(\Omega)}\right)_{i=1}^n.
\]
In view of \eqref{aneq}, equation \eqref{Galerkin:approximate} can be rewritten as
\begin{equation}
\label{Galerkin:system}
\left\{
\begin{aligned}
    \frac{dc^{(n)}}{dt}(t) &= F^{(n)}(c^{(n)}(t)), \\
    c^{(n)}(0) &= c^{(n)}_0.
\end{aligned}
\right.
\end{equation}

\textbf{Claim $1$}:  Equation \eqref{Galerkin:system} admits a solution in $[0, T_{\max})$. To this end, 
by multiplying the first equality of \eqref{Galerkin:system} by $c^{n}(t)$, we obtain
\begin{align*}
    \frac{1}{2}\frac{d|c^{(n)}(t)|^{2}}{dt} & = c^{(n)}(t)\frac{dc^{(n)}(t)}{dt} \\
    & = -\iint_{Q}g_{x,y}\left(\frac{\sum_{i=1}^{n}c_{i}^{(n)}(t)e_{i}(x)-\sum_{i=1}^{n}c_{i}^{(n)}(t)e_{i}(y)}{|x-y|^{s}}\right)\\
    &\qquad \qquad \qquad \times \frac{\left(\sum_{i=1}^{n}c_{i}^{(n)}(t)e_{i}(x)-\sum_{i=1}^{n}c_{i}^{(n)}(t)e_{i}(y)\right)}{|x-y|^{s}}d\mu\\
    & + \int_{\Omega}f(x,\sum_{i=1}^{n}c_{i}^{(n)}(t)e_{i}(x))\sum_{i=1}^{n}c_{i}^{(n)}(t)e_{i}(x)dx.
    \end{align*}
    Note that, from \eqref{Assump:g_3}, the first term on the right side of the above inequality is negative. Now by using  Lemma \ref{RELATION:MODULAR & NORM} and Corollary \ref{bounds on f:f}(i), we get
\begin{align}\label{aneqq}
\frac{1}{2}\frac{d|c^{(n)}(t)|^{2}}{dt}&\leq \int_{\Omega}Ah_{2}(x)\left|\sum_{i=1}^{n}c_{i}^{(n)}(t)e_{i}(x))\right|^{h_2(x)}dx \\
&\leq Ah_{2}^{+}\int_{\Omega}\sum_{i=1}^{n}|c_{i}^{(n)}(t)|^{h_{2}(x)}\sum_{i=1}^{n}|e_{i}(x)|^{h_{2}(x)}dx\nonumber\\
&\leq Ah_{2}^{+} \sum_{i=1}^{n}\max\left\{|c_{i}^{(n)}(t)|^{h_{2}^{+}}, |c_{i}^{(n)}(t)|^{h_{2}^{-}}\right\}\int_{\Omega}\sum_{i=1}^{n}|e_{i}(x)|^{h_{2}(x)}dx \nonumber \\
&\leq Ah_{2}^{+} \sum_{i=1}^{n}\max\left\{|c_{i}^{(n)}(t)|^{h_{2}^{+}}, |c_{i}^{(n)}(t)|^{h_{2}^{-}}\right\}\sum_{i=1}^{n}\int_{\Omega}|e_{i}(x)|^{h_{2}(x)}dx\nonumber\\
&\leq Ah_{2}^{+} \sum_{i=1}^{n}\max\left\{|c_{i}^{(n)}(t)|^{h_{2}^{+}}, |c_{i}^{(n)}(t)|^{h_{2}^{-}}\right\}\nonumber\\
&\qquad\qquad \qquad \sum_{i=1}^{n}\max\left\{\|e_{i}(x)\|^{h_{2}^{-}}_{L^{\Phi}(\Omega)},\|e_{i}(x)\|^{h_{2}^{+}}_{L^{\Phi}(\Omega)}\right\}\nonumber\\
&\leq Ah_{2}^{+} \sum_{i=1}^{n}\max\left\{|c_{i}^{(n)}(t)|^{h_{2}^{+}}, |c_{i}^{(n)}(t)|^{h_{2}^{-}}\right\}\nonumber\\
& \qquad \qquad \qquad \sum_{i=1}^{n}\max\left\{(C^{*}_{G})^{h_{2}^{-}}[e_{i}(x)]^{h_{2}^{-}}_{s,G_{x,y}},(C^{*}_{G})^{h_{2}^{+}}[e_{i}(x)]^{h_{2}^{+}}_{s,G_{x,y}}\right\}\nonumber\\
\label{C^{n}(t)<1}
&\leq A h_{2}^{+} \sum_{i=1}^{n}\max\left\{ |c_{i}^{(n)}(t)|^{h_{2}^{+}}, |c_{i}^{(n)}(t)|^{h_{2}^{-}}\right\}C(n),
\end{align}
where
\[C(n)= \sum_{i=1}^{n}\max\left\{(C^{*}_{G})^{h_{2}^{-}}[e_{i}(x)]^{h_{2}^{-}}_{s,G_{x,y}},(C^{*}_{G})^{h_{2}^{+}}[e_{i}(x)]^{h_{2}^{+}}_{s,G_{x,y}}\right\}.\]

Next, to derive the upper estimates of $|c^{(n)}(t)|$, we consider two cases $|c^{(n)}(t)|<1$ and $|c^{(n)}(t)|\geq 1.$ If $|c^{(n)}(t)|<1$, then we are done. Otherwise, from \eqref{C^{n}(t)<1} we infer that
\begin{align*}
\frac{1}{2}\frac{d|c^{(n)}(t)|^{2}}{dt}&\leq A h_{2}^{+} |c^{(n)}(t)|^{h_{2}^{+}}C(n).
\end{align*}
Solving the ordinary differential inequality and using \eqref{assump:g_5}, we find that
\begin{align}
|c^{(n)}(t)|&\leq\frac{1}{\left(|c_{0}^{(n)}|^{2-h_{2}^{+}}+(2-h_{2}^{+})A h_{2}^{+}(C(n))t \right)^{\frac{1}{h_{2}^{+}-2}}}, \quad \forall \   t\in [0,\tilde{T}[,
\end{align}
where $\tilde{T}=\frac{|c_{0}^{(n)}|^{2-h_{2}^{+}}}{(h_{2}^{+}-2)A h_{2}^{+}C(n)}.$
Then, there exists a sufficiently small $\epsilon > 0$ such that $|c^{(n)}(t)| \leq C$  for all $t \in [0, \tilde{T}-\epsilon]$, where $C=C(\tilde{T}-\epsilon)$ is a constant. Thus, the map
$$
t \mapsto F^{(n)}(c^{(n)}(t))
$$
is bounded for $t \in [0, \tilde{T} - \epsilon]$, with its supremum denoted by $M_0$. Denote

\begin{align*}
 E:=\left\{(t,c^{(n)}(t))\in [0, \tilde{T}-\epsilon] \times \mathbb{R^{N}} : |c^{(n)}(t)-c^{(n)}_{0}|\leq C \right\}.
\end{align*}
By Peano's theorem, there exists solutions of \eqref{Galerkin:system} on $[0,T_{1}],$
where $T_{1}=min\left\{\tilde{T}-\epsilon, C M_0^{-1} \right\}.$
Now, by taking $T_{1}$ as new initial point 
and repeating the above argument in view of Peano's theorem, there exists a solution \(c^{n}(t)\) to equation \eqref{Galerkin:system} in $[0, T_{\max})$. This proves Claim $1$.
\\

\textbf{Claim 2}: We prove that \(u^{(n)}(\cdot,t) \in W\) for all $t \in [0, T_{\max})$. For this purpose, we multiply \eqref{Galerkin:approximate} by \(\frac{dc_j^{(n)}}{dt}\), sum over \(j\) from \(1\) to \(n\), and integrate from \(0\) to \(t\).
As a result, we obtain:
\begin{equation}
\label{Energy:approximation}
\begin{aligned}
\int_{0}^{t}||u^{(n)}_t(\cdot,\tau)||^{2}_{L^{2}(\Omega)} ~d\tau +E(u^{(n)}(\cdot, t))&=E(u^{(n)}(\cdot, 0)).
\end{aligned}
\end{equation}
Hence, from the fact that $u^{(n)}(\cdot,0)\rightarrow u_{0} \text{ in }W_{0}^{s,G_{x,y}}(\Omega)$, we get 
\[
E(u^{(n)}(\cdot,0))\rightarrow E(u_{0}) < d \quad \text{and} \quad I(u^{(n)}(\cdot,0))\rightarrow I(u_{0})>0 \ \text{as} \ n \to \infty. 
\]
Therefore, for sufficiently large \( n \), we infer that  
\begin{equation}
\label{for large n}
    \int_{0}^{t} \|u^{(n)}_t(\cdot,\tau)\|^{2}_{L^{2}(\Omega)} \,d\tau + E(u^{(n)}(\cdot, t)) = E(u^{(n)}(\cdot, 0)) < d, \quad \text{and} \quad I(u^{(n)}(\cdot,0)) > 0,
\end{equation}  
which implies that \( u^{(n)}(\cdot,0) \in W \).  

Now, suppose that \( u^{(n)}(\cdot,t) \notin W \) for sufficiently large \( n \). Then, there exists \( t_0 > 0 \) such that \( u^{(n)}(\cdot,t_0) \notin W \). If \( t_0 \) is not unique, we may assume, without loss of generality, that \( t_0 \) is the first time for which \( u^{(n)}(\cdot, t_{0}) \notin W \). This implies that \( u^{(n)}(\cdot, t) \not\equiv 0 \) for all \( t \in (0, t_0] \), and  
\[
E(u^{(n)}(\cdot,t_0)) = d, \quad \text{or} \quad E(u^{(n)}(\cdot,t_0)) > d, \quad \text{or} \quad I(u^{(n)}(\cdot,t_0)) = 0, \quad \text{or} \quad I(u^{(n)}(\cdot,t_0)) < 0.
\]
Clearly, from \eqref{for large n}, we have \(E(u^{(n)}(\cdot,t_0)) < d\).
Now, suppose \(I(u^{(n)}(\cdot,t_0)) = 0\). This implies that \(u^{(n)}(\cdot,t_0) \in \mathcal{N}\), {\it i.e.},
$$
d \leq E(u^{(n)}(\cdot,t_0)).
$$
This contradicts \eqref{for large n}. Now, if \( I(u^{(n)}(\cdot,t_0))< 0 \), then, since \( I \) is a continuous functional and \( I(u^{(n)}(\cdot,0))>0 \), there exists \( t_{1}\in (0,t_{0}) \) such that  
\[
I(u^{(n)}(\cdot,t_{1}))=0 \quad \text{and} \quad u^{(n)}(\cdot, t_1) \not\equiv 0.
\]  
This further implies that \( u^{(n)}(\cdot,t_1) \in \mathcal{N} \), \textit{i.e.},  
\[
d \leq E(u^{(n)}(\cdot,t_1)),
\]  
which contradicts \eqref{for large n} once again, thereby proving \textbf{Claim 2}.
\\
\textbf{Claim 3}: Uniform boundedness of $u^{(n)}$ in $W_{0}^{s,G_{x,y}}(\Omega))$ and of $f(x,u^{(n)})$ in $L^{\Phi^\ast}(\Omega)$ for all $t \in [0, T_{\max})$. \\ From Claim $2$,   \eqref{Assump:f_2} and \eqref{Assump:g_3}, we obtain
\begin{align*}
d>E(u^{(n)})&=\iint_{Q}G_{x,y}\left(D^{s}u^{(n)}\right)d\mu-\int_{\Omega}F(x,u^{(n)})~dx\\
&\geq\iint_{Q}G_{x,y}\left(D^{s}u^{(n)}\right)d\mu-\int_{\Omega}\frac{1}{h_{1}(x)}u^{(n)}f(x,u^{(n)})dx\\
&\geq\iint_{Q}G_{x,y}\left(D^{s}u^{(n)}\right)d\mu-\int_{\Omega}\frac{1}{h_{1}^{-}}u^{(n)}f(x,u^{(n)})dx\\
&\geq\iint_{Q}G_{x,y}\left(D^{s}u^{(n)}\right)d\mu-\frac{1}{h_{1}^{-}}\iint_{Q}g_{x,y}\left(D^{s}u^{(n)}\right)\left(D^{s}u^{(n)}\right)d\mu+\frac{1}{h_{1}^{-}}I(u^{(n)})\\
&\geq\iint_{Q}G_{x,y}\left(D^{s}u^{(n)}\right)d\mu-\frac{g^{+}}{h_{1}^{-}}\iint_{Q}G_{x,y}\left(D^{s}u^{(n)}\right)d\mu\\
&\geq\left(1-\frac{g^{+}}{h_{1}^{-}}\right)\iint_{Q}G_{x,y}\left(D^{s}u^{(n)}\right)d\mu\\
&\geq\left(1-\frac{g^{+}}{h_{1}^{-}}\right)\min\left\{[u^{(n)}]^{g^{-}}_{s,G_{x,y}},[u^{(n)}]^{g^{+}}_{s,G_{x,y}}\right\}.
\end{align*}
Thus, due to \eqref{for large n}, we conclude that
$$\int_{0}^{t}||u^{(n)}_t(\cdot,\tau) ||^{2}_{L^{2}(\Omega)} ~d \tau+\left(1-\frac{g^{+}}{h_{1}^{-}}\right)\min\left\{[u^{(n)}]^{g^{-}}_{s,G_{x,y}},[u^{(n)}]^{g^{+}}_{s,G_{x,y}}\right\}\leq E(u^{(n)}(0))<d,$$
which shows that
$$\int_{0}^{t}||u^{(n)}_t(\cdot,\tau) ||^{2}_{L^{2}(\Omega)} ~d\tau <d $$
and
$$\min\left\{[u^{(n)}]^{g^{-}}_{s,G_{x,y}},[u^{(n)}]^{g^{+}}_{s,G_{x,y}}\right\}<\frac{dh_{1}^{-}}{h_{1}^{-}-g^{+}}.$$
Combining, the above pieces of information, we infer that
\begin{equation}\label{bound of u with W0 norm}
    [u^{(n)}]_{s, G_{x,y}} < C_{d}, \quad C_{d}:= \max\left\{\left(\frac{dh_{1}^{-}}{h_{1}^{-}-g^{+}}\right)^{\frac{1}{g^{+}}}, \left(\frac{dh_{1}^{-}}{h_{1}^{-}-g^{+}}\right)^{\frac{1}{g^{-}}} \right\}.
\end{equation}
Furthermore, using Lemma \ref{RELATION:MODULAR & NORM} and Corollary \ref{bounds on f:f}(i), one has
\begin{align*}
    \varrho_{\Phi^{*}}(f(x,u^{(n)}))&=\int_{\Omega}\frac{h_{2}(x)-1}{h_{2}(x)}|f(x,u^{(n)})|^{\frac{h_{2}(x)}{h_{2}(x)-1}}dx\\
    &\leq \max \left\{A^{\frac{h_{2}^{+}}{h_{2}^{-}-1}},A^{\frac{h_{2}^{-}}{h_{2}^{+}-1}}\right\} \frac{h_{2}^{+}}{h_{2}^{-}}\int_{\Omega}|h_{2}(x)|^{\frac{h_{2}(x)}{h_{2}(x)-1}}|u^{(n)}|^{h_{2}(x)}dx\\
    &\leq \max \left\{A^{\frac{h_{2}^{+}}{h_{2}^{-}-1}},A^{\frac{h_{2}^{-}}{h_{2}^{+}-1}}\right\} \frac{h_{2}^{+}}{h_{2}^{-}}|h_{2}^{+}|^{h_{2}^{+}}\int_{\Omega}|u^{(n)}|^{h_{2}(x)}dx\\
    &\leq \max \left\{A^{\frac{h_{2}^{+}}{h_{2}^{-}-1}},A^{\frac{h_{2}^{-}}{h_{2}^{+}-1}}\right\} \frac{h_{2}^{+}}{h_{2}^{-}}|h_{2}^{+}|^{h_{2}^{+}}\max\left\{\|u^{(n)}\|^{h_{2}^{-}}_{L^{\Phi}(\Omega)},\|u^{(n)}\|^{h_{2}^{+}}_{L^{\Phi}(\Omega)}\right\}\\
    &\leq C^{*}_{G}\max \left\{A^{\frac{h_{2}^{+}}{h_{2}^{-}-1}},A^{\frac{h_{2}^{-}}{h_{2}^{+}-1}}\right\} \frac{h_{2}^{+}}{h_{2}^{-}}|h_{2}^{+}|^{h_{2}^{+}}\max \left\{[u^{(n)}]^{{h_{2}}^{-}}_{s,G_{x,y}},[u^{(n)}]^{h_{2}^{+}}_{s,G_{x,y}}\right\}.
\end{align*}
Hence, in light of \eqref{bound of u with W0 norm}, there exists a positive constant $\tilde{C}>0$ such that
$$\varrho_{\Phi^{*}}(f(x,u^{(n)}))\leq \tilde{C}, \ \mbox{and} \
 \min \left\{\|f(x,u^{(n)})\|_{L^{\Phi^{*}}(\Omega)}^{{h_{2}}^{*-}},\|f(x,u^{(n)})\|_{L^{\Phi^{*}}(\Omega)}^{{h_{2}}^{*+}}\right\}< \tilde{C}, $$
where
$$\tilde{C}:=C^{*}_{G}\max \left\{(A^{\frac{h_{2}^{+}}{h_{2}^{-}-1}},A^{\frac{h_{2}^{-}}{h_{2}^{+}-1}}\right\}\frac{h_{2}^{+}}{h_{2}^{-}}|h_{2}^{+}|^{h_{2}^{+}}\max \left\{(C_{d})^{{h_{2}}^{-}},(C_{d})^{h_{2}^{+}}\right\}.$$
This finishes the proof of Claim $3$. \\

The last claim establishes the existence of $u$ and a subsequence $\{u^{(n)}\}_{n\in \mathbb{N}}$ (still denoted by $\{u^{(n)}\}_{n\in \mathbb{N}}$) such that, as $n \to \infty$, we have
 \begin{align}
 \label{weak convergence}
 u^{(n)}&\stackrel{\ast}{\rightharpoonup} u\qquad\text{ in }L^{\infty}(0,T_{\max};W_{0}^{s,G_{x,y}}(\Omega));\\
 u_{t}^{n}&\weak u_{t} \qquad\text{ in }L^{2}(0,T_{\max};L^{2}(\Omega));\\
 f(x,u^{(n)})&\stackrel{\ast}{\rightharpoonup}\xi\qquad\text{ in }L^{\infty}(0,T_{\max};L^{\Phi^{*}}(\Omega))\nonumber.
 \end{align}
Then, invoking Theorem \ref{thm2} and  from Aubin-Lions compactness theorem\cite[Corollary 4, p. 85]{simon}, we get
 \begin{equation}
 \label{stong:conv}
 \begin{aligned}
 u^{(n)}&\to u \text{ in } C(0,T_{\max};L^{\Phi}(\Omega)),\text{  as } n \to \infty,
 \end{aligned}
  \end{equation}
and so, $\xi=f(x,u).$\\
\textbf{ Claim 4}: The function $u$ is a weak solution to problem (\ref{main:problem}). \\
To this end, choose $v\subseteq C^{1}(0,T_{\max};C_{0}^{\infty}(\Omega))$ with the following form
 $$v=\sum_{j=1}^{k}l_{j}(t)e_{j},$$
 where $l_{j}(t)\in C^{1}(0,T_{\max})$ with $j=1,2,3,\dots,k $ ($k\leq n $). Multiplying the first equality of \eqref{Galerkin:approximate} by $l_{j}(t)$ summing for $j$ from 1 to $n$, integrating with respect to $t$ from $0$ to $T$, we get
 $$\int_{0}^{T}\left(u_{t}^{n},v\right)_{L^2(\Omega)}dt+\int_{0}^{T}\left(u^{(n)},v\right)_{W_{0}^{s,G_{x,y}}(\Omega)} dt=\int_{0}^{T}\left(f(x,u^{(n)}),v\right)_{L^{2}(\Omega)}dt.$$
Therefore, taking the limit as \( n \to \infty \) and applying the theory of monotone and hemicontinuous operators as presented in \cite{Aberqi}, we obtain
 \begin{equation}
 \label{sol:v}
 \begin{aligned}
\int_{0}^{T}\left(u_{t},v\right)_{L^2(\Omega)}dt+\int_{0}^{T}\left(u,v\right)_{W_{0}^{s,G_{x,y}}(\Omega)} dt &=\int_{0}^{T}\left(f(x,u),v\right)_{L^{2}(\Omega)}dt.
\end{aligned}
\end{equation}
Since \( C^{1}(0,T_{\max};C_{0}^{\infty}(\Omega)) \) is dense in \( L^{2}(0,T_{\max};W_{0}^{s,G_{x,y}}(\Omega)) \), the identity \eqref{sol:v} holds for \( v \in L^{2}(0,T_{\max};W_{0}^{s,G_{x,y}}(\Omega)) \). Moreover, by the arbitrariness of \( T \in (0, T_{\max}) \), we have
\[
\left(u_{t},\phi \right)_{L^2(\Omega)} + \left(u,\phi\right)_{W_{0}^{s,G_{x,y}}(\Omega)} = \left(f(x,u),\phi\right)_{L^{2}(\Omega)},
\]
for all \( \phi \in W_{0}^{s,G_{x,y}}(\Omega) \) and a.e. \( t \in (0, T_{\max}) \). 
By \eqref{stong:conv} and the fact that \( u^{(n)}(x,0) \to u_{0}(x) \) in \( W_{0}^{s,G_{x,y}}(\Omega) \), it follows that \( u(x,0) = u_{0}(x) \). On the other hand, from \eqref{Energy:approximation} we have 
\begin{align*}
\int_{0}^{t}||u^{(n)}_t(\cdot,\tau)||^{2}_{L^{2}(\Omega)}d\tau +\iint_{Q}G_{x,y}\left(D^{s}u^{(n)}\right)d\mu &= E(u^{(n)}(0)) + \int_{\Omega}F(x,u^{(n)})dx.
\end{align*}
Then, by using \eqref{weak convergence}-\eqref{stong:conv}, the lower semicontinuity of modular functions \cite[Lemma 3.1.4]{M32} and the fact that \( u^{(n)}(\cdot, 0) \to u_{0} \) in \( W_{0}^{s,G_{x,y}}(\Omega) \), we obtain
\begin{align*}
   \int_{0}^{t}||u_t(\cdot, \tau)||^{2}_{L^{2}(\Omega)}d\tau +\iint_{Q}G_{x,y}\left(D^{s}u\right)d\mu &\leq \liminf_{n \to \infty}\int_{0}^{t}||u^{(n)}_t(\cdot,\tau)||^{2}_{L^{2}(\Omega)}d\tau + \liminf_{n \to \infty} \iint_{Q}G_{x,y}\left(D^{s}u^{(n)}\right)d\mu\\
   &\leq  \liminf_{n \to \infty} \left(\int_{0}^{t}||u^{(n)}_t(\cdot,\tau)||^{2}_{L^{2}(\Omega)}d\tau + \iint_{Q}G_{x,y}\left(D^{s}u^{(n)}\right)d\mu \right)\\
   &=\liminf_{n \to \infty}\left( E(u^{(n)}(0)) + \int_{\Omega}F(x,u^{(n)})dx \right)\\
    &=\lim_{n \to \infty}\left( E(u^{(n)}(0)) + \int_{\Omega}F(x,u^{(n)})dx \right)\\
    &=E(u_{0}) + \int_{\Omega}F(x,u)dx.
\end{align*}
Thus, we have
\begin{equation}\label{energy:ineq:relation}
    \int_{0}^{t}||u_t(\cdot, \tau)||^{2}_{L^{2}(\Omega)} ~d\tau + E(u)\leq E(u_{0})\qquad \text{a.e. } t \in (0, T_{\max}).
\end{equation}

\textbf{Claim 5:} $u(\cdot, t) \in W$ for all $0 \leq  t< T_{\max}.$ To achieve this, 
assume the contrary: there exists \(t_{0} \in (0,\infty) \) such that
\( u(\cdot, t_{0}) \notin  W \). If $t_0$ is not unique, without loss of generality, we suppose $t_0$ is the first time such that \( u(\cdot, t_{0}) \notin W \). Then we have \( u(\cdot, t) \not\equiv 0 \) for all $t \in (0, t_0]$ and
\[
I(u(\cdot, t_{0})) = 0\text{ or } I(u(\cdot, t_{0})) < 0  \text{ or } E(u(\cdot, t_{0})) = d \text{ or } E(u(\cdot, t_{0})) > d  .
\]
Clearly, from \eqref{energy:ineq:relation}, we know that \( E(u(\cdot, t_{0})) < E(u_0)<d\). Thus either $I(u(\cdot, t_{0})) = 0\text{ or } I(u(\cdot, t_{0})) < 0$. 
Suppose that \( I(u(\cdot, t_0)) = 0 \). Since \( u(\cdot, t_0) \notin W \), it follows that \( u(\cdot, t_0) \not\equiv 0 \), and thus \( u(\cdot, t_0) \in \mathcal{N} \). By the definition of \( d \), we then obtain \( E(u(\cdot, t_0)) \geq d \), which contradicts the energy relation given in \eqref{energy:ineq:relation}.
\\
Since $u\in C(0,T_{\max};L^{\Phi}(\Omega))$ and from \eqref{assump:g_5} we have $2<h_{2}^-$ and $ u\in C(0,T_{\max};L^{2}(\Omega))$.
Now, by taking $\phi =u(\cdot, t)$ for $0< t< T_{\max}$ in \eqref{definition of weak sol}, we obtain
 \begin{align}
 \label{u_t}
        \frac{1}{2}\frac{d}{dt}\|u(\cdot,t)\|^{2}_{L^{2}(\Omega)}&=-I(u(\cdot,t)).
    \end{align}
Suppose $I(u(\cdot, t_{0}))<0$. Then, from \eqref{u_t}, we get
\[ \frac{d}{dt}\|u(\cdot,t)\|^{2}_{L^{2}(\Omega)}|_{t=t_{0}}>0.\]
Therefore, since $ u\in C(0,T_{\max};L^{2}(\Omega))$, there exists $\epsilon> 0$ such that
\begin{align}
\label{U increasing}
  \|u(\cdot,t_{1})\|^{2}_{L^{2}(\Omega)}&<\|u(\cdot,t_{0})\|^{2}_{L^{2}(\Omega)}<\|u(\cdot,t_{2})\|^{2}_{L^{2}(\Omega)} \qquad\text{for all } t_{1},t_{2}\in (t_{0}-\epsilon, t_{0}+\epsilon).  
\end{align}
Now, multiplying the first equality of \eqref{Galerkin:approximate} by $c^{(n)}_{j}(t)$ and summing for $j$ from 1 to $n$, we get
 \begin{align}
        \frac{1}{2}\frac{d}{dt}\|u^{(n)}(\cdot,t)\|^{2}_{L^{2}(\Omega)}&=-I(u^{(n)}(\cdot,t))\nonumber.
    \end{align}
    Then, since $u^{(n)}(\cdot,t)\in W$ for all $0\leq t < T_{\max}$ and $u^{(n)}\in C(0,T_{\max};L^{2}(\Omega))$, $\|u^{(n)}(\cdot,t)\|^{2}_{L^{2}(\Omega)}$ is decreasing for all $0\leq t <T_{\max}.$ Therefore, we have 
\begin{align}
    \|u^{(n)}(\cdot,t_1)\|^{2}_{L^{2}(\Omega)}&> \|u^{(n)}(\cdot,t_0)\|^{2}_{L^{2}(\Omega)}>\|u^{(n)}(\cdot,t_2)\|^{2}_{L^{2}(\Omega)}.\nonumber
\end{align}
Taking the limit as \( n \to \infty \), we obtain  
\begin{align}
    \label{U decreasing}
     \|u(\cdot,t_1)\|^{2}_{L^{2}(\Omega)} &\geq \|u(\cdot,t_0)\|^{2}_{L^{2}(\Omega)} \geq \|u(\cdot,t_2)\|^{2}_{L^{2}(\Omega)},
\end{align}
which contradicts \eqref{U increasing}. Hence, we conclude that \( u(\cdot,t) \in W \) for all \( 0 \leq t < T_{\max} \).\\

\textbf{Claim 6:} $T_{\max} = + \infty.$ For this purpose, 
suppose by contradiction $T_{\max} < +\infty$ {\it i.e} $\|u(\cdot, t)\|_{W_{0}^{s,G_{x,y}}(\Omega)} \to + \infty$ as $t \to T_{\max}^-.$ By using the \textbf{Claim 5}, \eqref{Assump:f_2} and \eqref{Assump:g_3}, we obtain
\begin{align*}
d>E(u(\cdot, t))
&\geq\iint_{Q}G_{x,y}\left(D^{s}u(\cdot, t)\right)d\mu-\int_{\Omega}\frac{1}{h_{1}(x)} u(\cdot, t) f(x, u(\cdot, t))dx\\
&\geq\iint_{Q}G_{x,y}\left(D^{s} u(\cdot, t) \right)d\mu-\frac{1}{h_{1}^{-}}\iint_{Q}g_{x,y}\left(D^{s} u(\cdot, t) \right)\left(D^{s} u(\cdot, t) \right)d\mu+\frac{1}{h_{1}^{-}}I(u(\cdot, t))\\
&\geq\iint_{Q}G_{x,y}\left(D^{s} u(\cdot, t) \right)d\mu-\frac{g^{+}}{h_{1}^{-}}\iint_{Q}G_{x,y}\left(D^{s}u(\cdot, t)\right)d\mu\\
&\geq\left(1-\frac{g^{+}}{h_{1}^{-}}\right)\min\left\{[u(\cdot, t)]^{g^{-}}_{s,G_{x,y}},[u(\cdot, t)]^{g^{+}}_{s,G_{x,y}}\right\} \to +\infty \quad \text{as} \ t \to T_{\max}^-
\end{align*}
which is a contradiction. Hence, $T_{\max} = +\infty.$
\\

 It remains to prove the uniqueness of bounded solution. Accordingly, assume $u$ and $v$ bounded weak solutions of equation \eqref{main:problem}. Then, by Definition~\ref{Def:weak solution}, for any test function $\phi \in W_{0}^{s,G_{x,y}}(\Omega)$, we have

$$\left(u_{t},\phi\right)_{L^2(\Omega)}+\left(u,\phi\right)_{W_{0}^{s,G_{x,y}}(\Omega)} =\left(f(x,u),\phi \right)_{L^{2}(\Omega)},$$
and
$$\left(v_{t},\phi\right)_{L^2(\Omega)}+\left(v,\phi\right)_{W_{0}^{s,G_{x,y}}(\Omega)} =\left(f(x,v),\phi \right)_{L^{2}(\Omega)}.$$
Subtracting the two equalities above, setting \( \phi = u - v \in W_{0}^{s,G_{x,y}}(\Omega) \), and integrating with respect to \( t \) over the interval \([0, t]\), we get
\begin{align*}
&\int_{0}^{t}\int_{\Omega}(u_{t}-v_{t})(u-v)\,dx\,dt\\
&+\int_{0}^{t}\iint_{Q}g_{x,y}\left(\frac{u(x,t)-u(y,t)}{|x-y|^s}\right)\left(\frac{u(x,t)-v(x,t)-u(y,t)+v(y,t)}{|x-y|^s}\right)d\mu\, dt\\
&-\int_{0}^{t}\iint_{Q}g_{x,y}\left(\frac{v(x,t)-v(y,t)}{|x-y|^s}\right)\left(\frac{u(x,t)-v(x,t)-u(y,t)+v(y,t)}{|x-y|^s}\right)d\mu\, dt\\
&=\qquad \int_{0}^{t}\int_{\Omega}\left(f(x,u)-f(x,v)\right)(u-v)\,dx\,dt
\end{align*}
We can rewrite the above equality as
\begin{align*}
&\int_{0}^{t}\int_{\Omega}(u_{t}-v_{t})(u-v)\,dx\,dt\\
&+\int_{0}^{t}\iint_{Q}\left[g_{x,y}\left(\frac{u(x,t)-u(y,t)}{|x-y|^s}\right)- g_{x,y}\left(\frac{v(x,t)-v(y,t)}{|x-y|^s}\right)\right] \left(\frac{u(x,t)-u(y,t)}{|x-y|^s}-\frac{v(x,t)-v(y,t)}{|x-y|^s}\right)d\mu\,dt\\
&=\int_{0}^{t}\int_{\Omega}\left(f(x,u)-f(x,v)\right)(u-v)\,dx\,dt
\end{align*}

Note that, from \eqref{assump:g2}, the second term on the left-hand side of the equality is always  non-negative. So, we have
\begin{align*}
\int_{0}^{t}\int_{\Omega}(u_{t}-v_{t})(u-v)\,dx\,dt \leq \int_{0}^{t}\int_{\Omega}\left(f(x,u)-f(x,v)\right)(u-v)\,dx\,dt.
\end{align*}
Thus, since $u$ and $v$ are bounded weak solutions, from Corollary \ref{bounds on f:f} and \eqref{assump:f_0}, we have $f(\cdot,u)$ and $f(\cdot,v)$ are bounded and locally Lipschitz uniformly in $x \in \Omega$. Now, by using Cauchy Schwartz's inequality, we get 
\begin{align*}
\int_{\Omega}\int_{0}^{t} (u_{t} - v_{t}) (u-v) dt\, dx \leq C\int_{0}^{t}\int_{\Omega}\left(u-v\right)^{2}\,dx\,dt.
\end{align*}
Thus, by Gronwall's inequality and the fact that $u(x,0)=v(x,0)$, we deduce that $u=v$ a.e. in $\Omega.$
\end{proof}
 
\begin{theorem}\label{thm:global-existence-strong}
     Let conditions \eqref{assump:f_0}--\eqref{Assump:f_3}, \eqref{cond:g0}--\eqref{assump:g_5}, \eqref{Cond:g_7}--\eqref{Cond:g_8}, and \eqref{assump:h_0} hold, and let \( u_0 \in W_{0}^{s,G_{x,y}}(\Omega) \). If \( E(u_0) < d \) and \( I(u_0) > 0 \), then problem \eqref{main:problem} admits a global strong solution \( u \in L^\infty(0, \infty; W_{0}^{s,G_{x,y}}(\Omega)) \) with \( u_t \in L^2(0, \infty; L^2(\Omega)) \) and \( u(\cdot, t) \in W \) for all \( 0 \leq t < \infty \). Moreover, there exist constants \( \delta' \in (0,1) \) and \( C_{*} > 0 \) such that the following estimates hold for all \( 0 \leq t < +\infty \):

\begin{enumerate}[leftmargin=1.5cm, label=(\roman*), font=\normalfont]
    \item If \( [u(\cdot, t)]_{s,G_{x,y}} \geq 1 \), then
    \begin{equation*}
    \|u(\cdot, t)\|^2_{L^2(\Omega)} \leq
    \begin{cases}
     \left( \|u_0\|^{2-g^-}_{L^2(\Omega)} - (2 - g^-)(1 - \delta')g^{-}C_*^{-g^-}t \right)^{\frac{2}{2-g^-}}_{+}, & \text{if } g^- < 2,\\
     \|u_0\|^2_{L^2(\Omega)} e^{-2( 1-\delta')C_*^{-2}t}, & \text{if } g^- = 2,\\
     \left( \frac{1}{\|u_0\|^{2-g^-}_{L^2(\Omega)} + (2 - g^-)(\delta' - 1)g^{-}C_*^{-g^-}t} \right)^{\frac{2}{g^- - 2}}, & \text{if } g^- > 2,
    \end{cases}
    \end{equation*}
    where \( (z)_{+}:=\max\{z,0\} \). In particular, if \( g^- < 2 \), the solution vanishes in finite time  
    \[
    t^{*}=\frac{\|u_{0}\|^{2-g^{-}}_{L^{2}(\Omega)}}{(2-g^{-})(1-\delta')g^{-}C_{*}^{-g^{-}}}.
    \]
    
    \item If \( [u(\cdot, t)]_{s,G_{x,y}} < 1 \), then
    \begin{equation*}
    \|u(\cdot, t)\|^2_{L^2(\Omega)} \leq
    \begin{cases}
        \left( \|u_0\|^{2-g^+}_{L^2(\Omega)} - (2 - g^+)(1 - \delta')g^{-}C_*^{-g^+}t \right)^{\frac{2}{2-g^+}}_{+}, & \text{if } g^+ < 2,\\
        \|u_0\|^2_{L^2(\Omega)} e^{-2( 1-\delta')C_*^{-2}t}, & \text{if } g^+ = 2,\\
        \left( \frac{1}{\|u_0\|^{2-g^+}_{L^2(\Omega)} + (2 - g^+)(\delta' - 1)g^{-}C_*^{-g^+}t} \right)^{\frac{2}{g^+ - 2}}, & \text{if } g^+ > 2,
    \end{cases}
    \end{equation*}
    where \( (z)_{+}:=\max\{z,0\} \). In particular, if \( g^+ < 2 \), the solution \( u \) vanishes in finite time  
    \[
    t^{*}=\frac{\|u_{0}\|^{2-g^{+}}_{L^{2}(\Omega)}}{(2-g^{+})(1-\delta')g^{-}C_{*}^{-g^{+}}}.
    \]
\end{enumerate}
\end{theorem}
\begin{proof}
The existence of global weak solution follows by Theorem \ref{thm:global-existence-weak}. Using \eqref{assump:h_0} and Lemma \ref{del phi is bounded}, we obtain $f(\cdot, u) \in L^2(\Omega).$ By taking $g= f(\cdot, u) - u_t$ in Lemma \ref{continuity of strong solution}, we obtain \[u\in C([0,T];W_{0}^{s,G_{x,y}}(\Omega)) \quad \text{and} \quad \int_{0}^{t}\|u_{t}(\cdot,\tau)\|_{L^{2}(\Omega)}^{2}+E(u(\cdot,t))=E(u(\cdot,0)), \quad \text{for all } \, t\in [0,\infty).\]
Hence, the existence of strong solution $u$ of the problem \eqref{main:problem}.

Now, we derive the upper estimates. By applying Lemmas \ref{I(u)>0 implies E(u)>0} and \ref{u in W delta}, we conclude that \(u(\cdot,t) \in W_{\delta}\) for \(0 < t < \infty\) and \(\delta_{1} < \delta < \delta_{2}\), where \(\delta_1 < 1 < \delta_2\). The parameters \(\delta_1\) and \(\delta_2\) are determined as the solutions of the equation \(d(\delta) = E(u_0)\).
Furthermore, \(u(\cdot,t) \in W_{\delta}\) implies that \(I_{\delta'}(u) > 0\) for all \(0 < t < \infty\) and \(\delta_{1} < \delta' < 1 \). Now, by taking $\phi =u(\cdot, t)$ for $0< t< \infty$ in \eqref{definition of weak sol}, we obtain
\begin{equation}
\label{u_t is nonzero}
    \begin{aligned}
        \frac{1}{2}\frac{d}{dt}\|u\|^{2}_{L^{2}(\Omega)}&=-I(u).
    \end{aligned}
\end{equation}
Therefore, for \(\delta_{1} < \delta' < 1\) and using condition \eqref{Assump:g_3}, we obtain

\begin{align*}
    \frac{1}{2}\frac{d}{dt}\|u\|^{2}_{L^{2}(\Omega)}&=-I(u) =(\delta'-1){\iint_{Q}}g_{x,y}\left(D^{s}u\right)\left(D^{s}u\right)d\mu-I_{\delta'}(u)\\
    &\leq (\delta'-1){\iint_{Q}}g_{x,y}\left(D^{s}u\right)\left(D^{s}u\right)d\mu \leq (\delta'-1)g^{-}{\iint_{Q}}G_{x,y}\left(D^{s}u\right)d\mu,
\end{align*}
which implies, from Lemma \ref{RELATION:MODULAR & NORM}(i),
\begin{equation}
\label{blow up eqn}
\begin{aligned}
    \frac{1}{2}\frac{d}{dt}\|u\|^{2}_{L^{2}(\Omega)}&\leq (\delta'-1)g^{-}\min\{ [u]_{s,G_{x,y}}^{g^-},[u]_{s,G_{x,y}}^{g^+}\}.
\end{aligned}
\end{equation}
\noindent To estimate further, we consider two cases: \\
\textbf{Case 1:} $[u(\cdot, t)]_{s,G_{x,y}}\geq 1$. \vspace{0.1cm}\\
From \eqref{blow up eqn} and
 the embedding of $W_{0}^{s,G_{x,y}}(\Omega)$ into $L^{2}(\Omega)$, we get
     \begin{align*}
   \frac{1}{2}\frac{d}{dt}\|u\|^{2}_{L^{2}(\Omega)} &\leq(\delta'-1)g^{-} C_{*}^{-g^{-}}\|u\|_{L^{2}(\Omega)}^{g^-}.
\end{align*}
Now, if $g^{-}<2$, we obtain
$$\|u\|^{2}_{L^{2}(\Omega)}\leq \left(\|u_{0}\|^{2-g^{-}}_{L^{2}(\Omega)}-(2-g^{-})(1-\delta')g^{-}C_{*}^{-g^{-}}t\right)_+^{\frac{2}{2-g^{-}}},\text{\hspace{1cm}              } \forall \ 0\leq t < \infty.$$
This implies that the solution vanishes at a time $t^{*}=\frac{\|u_{0}\|^{2-g^{-}}_{L^{2}(\Omega)}}{(2-g^{-})(1-\delta')g^{-}C_{*}^{-g^{-}}}.$\\
If $g^{-}=2$, we get 
$$\|u\|^{2}_{L^{2}(\Omega)}\leq \|u_{0}\|^{2}_{L^{2}(\Omega)} e^{-2(1-\delta')C_{*}^{-2}t},\text{\hspace{1cm}} \forall \ 0\leq t< \infty. $$
If $g^{-}>2$, we get 
$$ \|u\|^{2}_{L^{2}(\Omega)}\leq \left(\frac{1}{\|u_{0}\|^{2-g^{-}}_{L^{2}(\Omega)}+(2-g^{-})(\delta'-1)g^{-}C_{*}^{-g^{-}}t}\right)^{\frac{2}{g^{-}-2}},\text{\hspace{1cm}} \forall \ 0\leq t< \infty.$$
\textbf{Case 2:}  $[u(\cdot, t)]_{s,G_{x,y}}>1 $ \vspace{0.1cm}\\
From \eqref{blow up eqn} and the embedding from $W_{0}^{s,G_{x,y}}(\Omega)$ into $L^{2}(\Omega)$, we get
     \begin{align*}
   \frac{1}{2}\frac{d}{dt}\|u\|^{2}_{L^{2}(\Omega)} 
    &\leq(\delta'-1)g^{-} C_{*}^{-g^{+}}\|u\|_{L^{2}(\Omega)}^{g^+}.
\end{align*}
Now, by repeating the above arguments, we obtain the required claim.
\end{proof}
Next, we are concerned with the blow-up in finite time behavior of the strong solution. For this, we assume that
 \begin{enumerate}[leftmargin=1.5cm,label=\textnormal{($\tilde{f}_3$)},ref=\textnormal{($\tilde{f}_3$)}]
     \item \label{Assump: tilde{f}} 
     $ t \left( f'(x,t)t - (\alpha -1) f(x,t) \right) > 0, \quad \forall \ (x,t) \in \Omega \times \mathbb{R}, \quad \text{where} \ \alpha = \begin{cases}
     g^+, & \ \text{if} \ g^+ > 2,\\
     > 2, & \ \text{if} \ g^+ \leq  2.
     \end{cases}$
 \end{enumerate}

\begin{theorem}
\label{Blow-up thm}
    Let the conditions \eqref{assump:f_0}--\eqref{Assump:f_2}, \ref{Assump: tilde{f}}, \eqref{cond:g0}--\eqref{assump:g_5} and \eqref{assump:h_0} hold and $u_{0} \in W_{0}^{s,G_{x,y}}(\Omega)$. Assume further that $E(u_{0}) < d$ and $I(u_{0}) < 0$. Then, the strong solution $u$ of problem \eqref{main:problem} exhibits blow-up in the sense that
\[
\lim_{t \to T^{*}} \int_{0}^{t} \|u(\cdot, \tau)\|^{2}_{L^{2}(\Omega)} \, d\tau = +\infty \quad \text{where} \quad T^\ast:= \frac{4\|u_0\|_{L^2(\Omega)}^2(\alpha-1)}{\alpha(\alpha-2)^{2}(d-E(u_{0}))}.
\]
\end{theorem}
\begin{proof}
Arguing by contradiction, we assume that the solution is global in time, {\it i.e.} $T_{max} = +\infty.$ For $T>0$, we define the auxiliary function $M:[0, T] \rightarrow (0,\infty)$ as
\begin{equation}
\label{def of M}
\begin{aligned}
  M(t)&:= \int_{0}^{t}\|u(\cdot, \tau)\|_{L^2(\Omega)}^{2}\,d\tau+(T-t)\|u_{0}\|_{L^2(\Omega)}^{2}+b(t+a)^{2}, \quad \forall \ t \in [0, T],
  \end{aligned}
\end{equation}
where $a$ and $b$ are positive constants satisfying suitable conditions which will be stated later.
By differentiating $M$ with respect to $t$ and from \eqref{u_t is nonzero}, we obtain
    \begin{equation}
    \label{DM(t):u}
    \begin{aligned}
     & M'(t)=\|u(\cdot, t)\|^{2}_{L^{2}(\Omega)}-\|u_{0}\|^{2}_{L^{2}(\Omega)}+2b(t+a) \ \ \text{and} \ \ M''(t)=-2I(u(\cdot, t))+2b, \quad \forall \ t \in [0, T].
    \end{aligned}
\end{equation}
Now, by using the fact that $u_{0}\in V$ and applying Lemma \ref{u in W delta}(ii) for $\delta=1$, we get 
    \begin{equation}
        \label{u in V}
I(u(\cdot,t))<0, \quad \forall \ t \geq 0 \ \text{and the maps} \ t \longmapsto M(t),  t \longmapsto M'(t) \ \text{are strictly increasing in} \ [0, T].
     \end{equation}
Moreover, by the definition of $d$ and Lemma \ref{Lemma:3.3}(iii), there exists a $\lambda^{*}\in (0,1)$ such that 
     \begin{equation}\label{lambda-star-est}
    I(\lambda^{*}u(\cdot,t))=0\quad\text{and} \quad d\leq E(\lambda^{*}u(\cdot,t)), \quad \forall \ t \geq 0.
     \end{equation}
Now, for a fixed $t \in [0, T]$ and $\beta \in \big[\frac{1}{\alpha}, \frac{1}{2}\big)$, we define a function $g:[\lambda^{*},1]\rightarrow (0,\infty)$ as
    \[g(\lambda):= E(\lambda u(\cdot,t))-\beta I(\lambda u(\cdot,t)).\]
By differentiating the function $g$ with respect to $\lambda$ and using \eqref{existence of lambda d lambda}-\eqref{Derivative of I}, we obtain
\begin{align*}
    \frac{dg(\lambda)}{d\lambda} &= \frac{d}{d\lambda}E(\lambda u(\cdot,t))-\frac{d}{d\lambda}\left(\beta I(\lambda u(\cdot,t)) \right)\\
    &=\frac{I(\lambda u(\cdot,t))}{\lambda}-\beta \frac{d}{d\lambda}\left(\lambda \frac{d}{d\lambda}E(\lambda u(\cdot,t)) \right)\\
    &= \frac{I(\lambda u(\cdot,t))}{\lambda}-\frac{\beta}{\lambda}I(\lambda u(\cdot,t))- \beta\lambda\frac{d^{2}}{d\lambda^{2}}E(\lambda u(\cdot,t))\\
    &=\frac{(1-\beta)}{\lambda}I(\lambda u(\cdot,t))-\frac{\beta}{\lambda}\left[\iint_{Q}g_{x,y}'\left(\lambda D^{s}u\right)\left(\lambda D^{s}u\right)^{2}d\mu-\int_{\Omega}f'(x,\lambda u)(\lambda u)^{2}dx \right].
\end{align*}
Now, by using the definition of $I$, \eqref{Assump:g_3}, \ref{Assump: tilde{f}}, and \eqref{u in V} combined with $\beta \geq \frac{1}{\alpha}$ in the above estimate, we deduce that
\begin{align}
         \qquad \qquad \frac{dg(\lambda)}{d\lambda}&\geq \frac{(1-\beta)}{\lambda}\iint_{Q}g_{x,y}\left(\lambda D^{s}u\right)\left(\lambda D^{s}u\right)d\mu- \frac{(1-\beta)}{\lambda}\int_{\Omega}f(x,\lambda u)(\lambda u)dx \nonumber\\
     &\qquad\qquad -\frac{\beta}{\lambda}(g^{+}-1)\iint_{Q}g_{x,y}\left(\lambda D^{s}u\right)\left(\lambda D^{s}u\right)d\mu+\frac{\beta}{\lambda}\int_{\Omega}f'(x,\lambda u)(\lambda u)^{2}dx\nonumber\\
    &> \frac{(1-\beta g^{+})}{\lambda}\left(\iint_{Q}g_{x,y}\left(\lambda D^{s}u\right)\left(\lambda D^{s}u\right)d\mu \right) + \frac{(\beta \alpha -1)}{\lambda}\int_{\Omega}f(x,\lambda u)(\lambda u)dx\nonumber\\
    & > \begin{cases}
    \frac{\beta}{\lambda}(\alpha-g^{+})\int_{\Omega}f(x,\lambda u)(\lambda u)dx, & \ \text{if} \ \beta g^+ \geq 1,\\
    \frac{(\beta \alpha -1)}{\lambda} \int_{\Omega}f(x,\lambda u)(\lambda u)dx , & \ \text{if} \ \beta g^+ < 1,
    \end{cases}\\
    & \geq 0 \nonumber.
     \end{align}
This implies that $g$ is an increasing function in $[\lambda^{*},1]$. Moreover, in view of \eqref{lambda-star-est}, we have
     \begin{align}
     \label{estimate for I}
         I(u(\cdot,t))&\leq \frac{1}{\beta}\left(E(u(\cdot,t))-d\right), \quad \forall \ t \in [0, T].
     \end{align}
Next, we show that 
\begin{equation*}
    M''(t)M(t) - \frac{1}{2\beta} \left(M'(t)\right)^{2} \geq 0, \quad \forall \ t \in [0, T].
\end{equation*}
We begin with estimating the term $M'(t)$. For this, we introduce the auxiliary functions $\zeta$ and $\rho$ given by
\begin{equation}
\label{Def of zeta}
    \begin{aligned}
\zeta(t)&=\left(\int_{0}^{t}\|u(\cdot, \tau)\|_{L^{2}(\Omega)}^2 \,d\tau\right)^{\frac{1}{2}}\quad \text{and}\quad \rho(t)=\left(\int_{0}^{t}\|u_{\tau}(\cdot, \tau)\|_{L^{2}(\Omega)}^2 \,d\tau\right)^{\frac{1}{2}}.
    \end{aligned}
\end{equation}
By applying the Cauchy-Schwarz inequality, we obtain
\begin{equation}\label{M'-est auxiliary function}
\begin{split}
    0 & \leq  \left[\sqrt{b}\zeta(t)-\sqrt{b}(t+a)\rho(t)\right]^{2} =\left[\left(\sqrt{b}\zeta(t)\right)^{2}+\left(\sqrt{b}(t+a)\rho(t)\right)^{2}-2\zeta(t)\rho(t)b(t+a)\right]\\
    & = \left[(\zeta(t))^{2}+b(t+a)^{2} \right]\left[(\rho(t))^{2}+b\right]-\left[ \zeta(t)\rho(t)+b(t+a)\right]^{2}\\
    & \leq \left(\int_{0}^{t}\|u(\cdot, \tau)\|_{L^2(\Omega)}^{2}\,d\tau+b(t+a)^{2}\right)\left( \int_{0}^{t}\|u_{\tau}(\cdot, \tau)\|_{L^{2}(\Omega)}^2 \,d\tau+b\right)\\
     & -\left(\int_{0}^{t}\int_{\Omega}u(\cdot, \tau)u_{\tau}(\cdot, \tau)\,dx\,d\tau+b(t+a)\right)^{2}.
\end{split}
\end{equation}
From \eqref{def of M}-\eqref{DM(t):u} and \eqref{M'-est auxiliary function}, we have
\begin{equation}\label{m-est}
    \begin{split}
        \left(M'(t)\right)^{2}
&=\left(\left(\|u(\cdot, t)\|^{2}_{L^{2}(\Omega)}-\|u_{0}\|^{2}_{L^{2}(\Omega)}\right)+ 2b(t+a)\right)^{2} = \left(\int_0^t \frac{d}{d\tau} \|u(x, \tau)\|_{L^2(\Omega)}^2 ~d \tau+ 2 b(t+a)\right)^{2} \\
&=4\left(\int_{0}^{t}\int_{\Omega}u(\cdot, \tau)u_{\tau}(\cdot, \tau)\,dx\,d\tau+b(t+a)\right)^{2}\\
&=4\left(M(t)-(T-t)\|u_{0}\|_{L^2(\Omega)}^{2}\right)\left(\int_{0}^{t}\|u_{\tau}(\cdot, \tau)\|_{L^{2}(\Omega)}^2 \,d\tau+b\right)\\
&-4\Bigg[\left(\int_{0}^{t}\|u(\cdot, \tau)\|_{L^2(\Omega)}^{2}\,d\tau+b(t+a)^{2}\right)\left( \int_{0}^{t}\|u_{\tau}(\cdot, \tau)\|_{L^{2}(\Omega)}^2 \,d\tau+b\right)\\
~&~ -\left(\int_{0}^{t}\int_{\Omega}u(\cdot, \tau)u_{\tau}(\cdot, \tau)\,dx\,d\tau+b(t+a)\right)^{2}\Bigg]\\
&\leq 4 M(t)\left(\int_{0}^{t}\|u_{\tau}(\cdot, \tau)\|_{L^{2}(\Omega)}^2 \,d\tau+b\right), \quad \forall\, t\in[0,T].
    \end{split}
\end{equation}

Finally, by using \eqref{DM(t):u}, \eqref{strong solution}, \eqref{estimate for I} and \eqref{m-est}, we obtain
\begin{equation*}
\begin{split}
    M''(t) & = -2 I(u(\cdot, t)) + 2b \geq \frac{2}{\beta}(d -E(u(\cdot,t)))+2b\\
    &\geq \frac{2}{\beta}\left(d- E(u_{0})+\int_{0}^{t}\|u_{\tau}(\cdot,\tau)\|_{L^{2}(\Omega)}^{2}\, d\tau\right)+2b\\
    &\geq \frac{2}{\beta}\left( \frac{\left(M'(t)\right)^{2}}{4M(t)}-b\right)+ \frac{2}{\beta}\left( d-E(u_{0})\right)+2b, \quad \forall \ t \in [0, T].
    \end{split}
\end{equation*}
This further implies
\begin{equation}
    \label{final estimate}
    \begin{aligned}
       M(t)M''(t)-\frac{1}{2\beta}\left(M'(t)\right)^{2}\geq M(t)\left[\frac{2}{\beta}\left( d-E(u_{0})\right)-2b\left(\frac{1}{\beta}-1\right)\right] \geq 0, \quad \forall \ t \in [0, T],   
    \end{aligned}
\end{equation}
where 
\begin{equation}\label{cond:b}
   \frac{1}{\alpha} \leq \beta <1 \quad \text{and} \quad b \leq \frac{d-E(u_{0})}{1-\beta}.
\end{equation}
Now, by following the concavity method introduced by Levine \cite[Theorem I]{Levine-1973}, we observe that \eqref{u in V}, \eqref{final estimate} and $M'(0) = 2ab>0$ gives
\[
\ \left( M^{-\theta}(t) \right)' = -\theta M^{-\theta-1}(t) M'(t) < 0, \quad \forall \ t \in [0, T],
\]
and \[
\left( M^{-\theta}(t) \right)'' = \theta M^{-\theta-2}(t) \left( (1+\theta)\left(M'(t)\right)^{2} - M(t)M''(t) \right) \leq 0, \quad \forall \ t \in [0, T],
\]
where \( \theta := \frac{1 - 2\beta}{2\beta} > 0\) and $\beta \in \big[\frac{1}{\alpha}, \frac{1}{2}\big)$. The above inequalities show that $M^{-\theta}$ is a positive  decreasing and a concave function in $[0, T]$. Since concave functions lie below any tangent line, we obtain

\[
0 < M^{-\theta}(t) \leq M^{-\theta}(0) + t \left(M^{-\theta}\right)'(0) = M^{-\theta}(0) - \theta t M^{-\theta-1}(0) M'(0), \quad \forall \ t \in [0, T],
\]
which implies 
\[ M(t) \geq M^{(1+\frac{1}{\theta})}(0)\left\{M(0)-\theta t M'(0) \right\}^{-\frac{1}{\theta}} \geq 0, \quad t \in [0, T].\]
From the facts that \( M(0)) > 0 \) and \( M'(0) > 0 \), it follows that
\begin{equation}
    \label{Tmax}
    \begin{aligned}
        &M(t) \to +\infty \ \text{as} \ t \to  \frac{M(0)}{\theta M'(0)} = \frac{T \|u_0\|_{L^2(\Omega)}^2 + ba^2}{2ab\l(\frac{1}{2\beta}-1\r)} \leq T
    \end{aligned}
\end{equation}
where the last inequality follows by choosing $a$ and $T$ large enough such that 
\begin{equation}
    \label{T-ast}
    \begin{aligned}
        \frac{\|u_0\|_{L^2(\Omega)}^2}{\l(\frac{1}{2\beta}-1\r)2a} < b \leq \frac{d-E(u_0)}{1-\beta} \qquad \text{and} \qquad T_a(b, \beta):=\frac{a^{2}b}{2ab\l(\frac{1}{2\beta}-1\r)-\|u_0\|_{L^2(\Omega)}^2}\leq T.
    \end{aligned}
\end{equation}
Moreover, by taking $T= T_a(b, \beta)$ in \eqref{Tmax}, we have
\[
M(t) \to +\infty \ \text{as} \ t \to  T_a(b, \beta):= \frac{a^{2}b}{2ab\l(\frac{1}{2\beta}-1\r)-\|u_0\|_{L^2(\Omega)}^2}
\]
which is a contradiction of $u$ being a global strong solution. Hence, $T_{\max} < +\infty.$ 

Next, we minimize the blow-up time $T_a(b, \beta)$ with respect to the parameters $a, b$ and $\beta$ such that $(b, \beta) \in \mathcal{R}_a:= \bigg(\frac{\|u_0\|_{L^2(\Omega)}^2}{\l(\frac{1}{2\beta}-1\r)2a}, \frac{d-E(u_0)}{1-\beta} \bigg] \times \bigg[\frac{1}{\alpha}, \frac{1}{2}\bigg)$. 
It is easy to see that the maps $\beta \mapsto T_a(b, \beta)$ and $b \mapsto T_a(b, \beta)$ are increasing and decreasing, respectively. Therefore,  the minimum value of $T_a(b, \beta)$ in $\mathcal{R}_a$ is given by

\begin{equation}
    \label{h_a increasing}
    \begin{aligned}
         T_a\l(\frac{\alpha(d-E(u_{0}))}{\alpha-1},\frac{1}{\alpha}\r) = \frac{\alpha a^2 (d-E(u_{0}))}{a \alpha(d-E(u_{0})) \l(\alpha-2\r) - (\alpha-1)\|u_0\|^2_{L^2(\Omega)}}.
    \end{aligned}
\end{equation}
Next, we define $g: \mathbb{S} \to \mathbb{R}$ such that
\[
g(a):= T_a\l(\frac{\alpha(d-E(u_{0}))}{\alpha-1},\frac{1}{\alpha}\r) \quad \text{for} \ a \in \mathbb{S}:= \l(\frac{\|u_0\|_{L^2(\Omega)}^2(\alpha-1)}{(\alpha-2)\alpha(d-E(u_{0}))}, +\infty\r).
\]
Note that the function $g$ attained its minimum value at $\frac{2(\alpha-1)\|u_0\|_{L^2(\Omega)}^2}{\alpha(\alpha-2)(d-E(u_{0}))}$ and the minimum value of $g$ in $\mathbb{S}$ is given by
\[
g\l(\frac{2\|u_0\|_{L^2(\Omega)}^2(\alpha-1)}{\alpha(\alpha-2)(d-E(u_{0}))}\r) = \frac{4  (\alpha-1) \|u_0\|_{L^2(\Omega)}^2}{\alpha (\alpha-2)^2 (d-E(u_0)}.
\]
Therefore, the least blow time $T^\ast$ independent of paramteres $a, b$ and $\beta$ is given by
\[T^{\ast}:=\frac{4\|u_0\|_{L^2(\Omega)}^2(\alpha-1)}{\alpha(\alpha-2)^{2}(d-E(u_{0}))} \quad \text{such that} \quad \lim_{t \to T^{*}} \int_{0}^{t} \|u(\cdot, \tau)\|^{2}_{L^{2}(\Omega)} \, d\tau = +\infty.\]
\end{proof}

\subsection{On the Case of Critical Initial Energy: $E(u_0) = d$}
In this subsection, we address the case of critical initial energy \( E(u_{0}) = d \). Specifically, we will prove the following:\\
\begin{enumerate}

 \item[-] If \( I(u_{0}) \geq 0 \), then problem \eqref{main:problem} admits a global weak solution.
\item [-] If \( I(u_{0}) < 0 \), then all strong solutions of problem \eqref{main:problem} blow up in finite time.
\end{enumerate}
\vspace{0,5 cm}
The main result of this section is stated as follows:
\begin{theorem}\label{main-exist-critical-weak}
    Let the conditions \eqref{assump:f_0}–\eqref{Assump:f_3} and \eqref{cond:g0}–\eqref{Cond:g_8} hold, and let \( u_{0} \in W_{0}^{s,G_{x,y}}(\Omega) \). If \( E(u_{0}) = d \) and \( I(u_{0}) \geq 0 \), then problem \eqref{main:problem} admits a global weak solution \( u \in L^{\infty}(0,\infty; W_{0}^{s,G_{x,y}}(\Omega)) \) with \( u_{t} \in L^{2}(0,\infty; L^{2}(\Omega)) \) and \( u(\cdot, t) \in W \cup \partial W \) for \( 0 \leq t < \infty \). Moreover, the weak solution is unique if it is bounded. Furthermore, if there exists \( t^{*} > 0 \) such that \( I(u(\cdot,t)) > 0 \) for \( 0 < t < t^{*} \) and \( I(u(\cdot,t^{*})) = 0 \), then there exists a weak solution \( u(\cdot,t) \) which vanishes in finite time \( t^{*} \).
\end{theorem}

\begin{proof}
    Let $\lambda_k = 1 - \frac{1}{k}$, for $k \in \mathbb{N}$. Consider the following initial value problem:
\begin{equation} \label{modified:problem}
    \left\{
    \begin{aligned}
      u_t + (-\Delta)^{s}_{{g}_{x,y}} u &= f(x,u), && \text{in } \Omega \times (0, \infty), \\
      u &= 0, && \text{in } \mathbb{R}^N \setminus \Omega \times (0, \infty), \\
      u(x,0) &= \lambda_k u_0(x) := u_{0k}, && \text{in } \Omega.
    \end{aligned}
    \right.
\end{equation}
Let $I(u_0) \geq 0$. By  Lemma \ref{positive:depth} and Lemma \ref{Lemma:3.3}(iii), we deduce that $u_0 \not \equiv 0$ and there exists a unique $\lambda^* = \lambda^*(u_0) \geq 1$ such that $I(\lambda^* u_0) = 0$. Since $\lambda_k < 1 \leq \lambda^*$, and using Lemma \ref{Lemma:3.3}(ii)-(iii), we obtain the following:
$$ I(u_{0k}) = I(\lambda_k u_0) > 0 \quad \text{and} \quad E(u_{0k}) = E(\lambda_k u_0) < E(u_0) = d. $$
Therefore, by Theorem \ref{thm:global-existence-weak}, for each $k$, the problem \eqref{modified:problem} has a global weak solution $u^{(k)} \in L^\infty(0, \infty; W_0^{s, G_{x,y}}(\Omega))$, with $u_t^{(k)} \in L^2(0, \infty; L^2(\Omega))$, and $u^{(k)} \in W$, satisfying
$$ \int_0^t \| u^{(k)}_t(\cdot, \tau) \|_{L^2(\Omega)}^2 \,d\tau + E(u^{(k)}(\cdot, t)) = E(u_{0k}) < d \ \text{for} \ t>0.$$
Repeating the argument in \textbf{Claim 3} and \textbf{Claim 4} of Theorem \ref{thm:global-existence-weak}, there exists a subsequence (denoted with same notation) $\{u^{(k)}\}_{k \in \mathbb{N}}$ that converges to a function $u$, and $u$ is a weak solution of the problem \eqref{main:problem} with $I(u(\cdot, t)) \geq 0$ and $E(u(\cdot, t)) \leq d$ for $0 \leq t < \infty$. It implies that $u(\cdot,t)\in W\cup\partial W$ for $0 \leq t < \infty$. The proof of the uniqueness of the bounded solutions follows from the same reasoning as in Theorem \ref{thm:global-existence-weak}. If we assume that $I(u(\cdot,t))>0$ for $0<t<t^{*}$ and $I(u(\cdot,t^{*}))=0$. Then \eqref{u_t is nonzero} implies $u_{t}(\cdot, t)\not \equiv 0$ for $0<t<t^{*}$.
Therefore, by \eqref{sol: u} we have
$$E(u(\cdot,t^{*}) \leq \  d-\int_{0}^{t^{*}}\|u_t(\cdot,\tau)\|^{2}_{L^{2}(\Omega)}\,d\tau < d.$$
By the definition of $d$, we get $[u(\cdot,t^{*})]_{s,G_{x,y}}=0$. Now, by extending the function $u$ as $u(\cdot,t)\equiv 0$ for $t\geq t^{*}$, we obtain a weak solution vanishing in finite time $t^{*}$.
\end{proof}
\begin{theorem}\label{main-exist-critical-strong}
    Let the conditions \eqref{assump:f_0}–\eqref{Assump:f_3} and \eqref{cond:g0}-\eqref{assump:g_5}, \eqref{Cond:g_7}-\eqref{Cond:g_8}, \eqref{assump:h_0} hold, and let \( u_{0} \in W_{0}^{s,G_{x,y}}(\Omega) \). If \( E(u_{0}) = d \) and \( I(u_{0}) \geq 0 \), then problem \eqref{main:problem} admits a global strong solution \( u \in L^{\infty}(0,\infty; W_{0}^{s,G_{x,y}}(\Omega)) \) with \( u_{t} \in L^{2}(0,\infty; L^{2}(\Omega)) \) and \( u(\cdot, t) \in W \cup \partial W \) for \( 0 \leq t < \infty \).  Furthermore, if \( I(u(\cdot, t)) > 0 \) for \( 0 < t < \infty \), then for any \( t_{0} > 0 \) there exists a \( \delta^{'} \in (0,1) \) such that the following estimates hold for $t_0 < t< + \infty$:
 \begin{enumerate}[leftmargin=1.5cm,label=(\roman*),font=\normalfont]
    \item If $[u(\cdot, t)]_{s,G_{x,y}} \geq 1$, then
    \begin{equation*}
    \|u(\cdot, t)\|^2_{L^2(\Omega)} \leq
    \begin{cases}
     \left( \|u(\cdot, t_0)\|^{2-g^-}_{L^2(\Omega)} - (2 - g^-)(1 - \delta^{'})g^{-}C_*^{-g^-}t \right)^{\frac{2}{2-g^-}}_{+} \ & \ \text{if} \ g^- < 2,\\
     \|u(\cdot, t_0)\|^2_{L^2(\Omega)} e^{-2( 1-\delta^{'})C_*^{-2}t} \ & \ \text{if} \ g^- = 2,\\
     \left( \frac{1}{\|u(\cdot, t_0)\|^{2-g^-}_{L^2(\Omega)} + (2 - g^-)(\delta^{'} - 1)g^{-}C_*^{-g^-}t} \right)^{\frac{2}{g^- - 2}} \ & \ \text{if} \ g^- >2,
    \end{cases}
    \end{equation*}
    where $(z)_{+}:=\max\{z,0\}$. In particular, if $g^-<2$, the solution vanishes in finite time $t^{*}=\frac{\|u(\cdot, t_0)\|^{2-g^{-}}_{L^{2}(\Omega)}}{(2-g^{-})(1-\delta')g^{-}C_{*}^{-g^{-}}}.$\\
    \item If $[u(\cdot, t)]_{s,G_{x,y}} < 1$, then
    \begin{equation*}
    \|u(\cdot, t)\|^2_{L^2(\Omega)} \leq
    \begin{cases}
        \left( \|u(\cdot, t_0)\|^{2-g^+}_{L^2(\Omega)} - (2 - g^+)(1 - \delta^{'})g^{-}C_*^{-g^+}t \right)^{\frac{2}{2-g^+}}_{+} \ & \ \text{if} \ g^+<2,\\
        \|u(\cdot, t_0)\|^2_{L^2(\Omega)} e^{-2( 1-\delta^{'})C_*^{-2}t} \ & \ \text{if} \ g^+ =2,\\
        \left( \frac{1}{\|u(\cdot, t_0)\|^{2-g^+}_{L^2(\Omega)} + (2 - g^+)g^{-}(\delta^{'} - 1)C_*^{-g^+}t} \right)^{\frac{2}{g^+ - 2}} \ & \ \text{if} \ g^+>2,
    \end{cases}
    \end{equation*}
      where $(z)_{+}:=\max\{z,0\}$. In particular, if $g^+<2$ the solution $u$ vanishes in finite time $t^{*}=\frac{\|u(\cdot, t_0)\|^{2-g^{+}}_{L^{2}(\Omega)}}{(2-g^{+})(1-\delta')g^{-}C_{*}^{-g^{+}}}.$
\end{enumerate}
    
\end{theorem}
\begin{proof}
The existence of strong solution can be proved by using the same arguments as in Theorem \ref{main-exist-critical-weak} by using the Theorem \ref{thm:global-existence-strong} in place of Theorem \ref{thm:global-existence-weak}. Now, if $I(u(\cdot, t))> 0$ for $0<t<\infty$, then \eqref{u_t is nonzero} implies that $u_{t}(\cdot, t) \not \equiv 0$. Therefore, by Lemma \ref{I(u)>0 implies E(u)>0} and \eqref{sol: u},  for any $t_0>0$ we have
 $$ 0< E(u(\cdot, t_{0})) \  \leq \ d-\int_{0}^{t_0} \|u_t (\cdot, \tau)\|^{2}_{L^{2}(\Omega)}\, d\tau <d.$$
Now, by taking $t=t_{0}$ as the initial time, from Lemma \ref{u in W delta}, we know that $u(\cdot,t)\in W_{\delta}$ for $\delta_{1}<\delta<\delta_{2}$ and $t_{0}<t<\infty$ under the condition $E(u(\cdot, t_{0}))< d$ and $I(u(\cdot,t_{0}))> 0$, where $\delta_{1}<\delta<\delta_{2}$ are two roots of $d(\delta)=E(u(\cdot,t_{0}))$. Thus,  $I_{\delta^{'}}(u(\cdot,t)) > 0$ for $\delta^{'}\in (\delta_{1},1)$ and $t_{0}<t<\infty$. Finally, by repeating the arguments of Theorem \ref{thm:global-existence-strong}, we obtain the required estimate.
\end{proof}
Next, we are concerned with blow-up in finite time.

\begin{theorem}\label{Main thm3}
Let conditions \eqref{assump:f_0}-\eqref{Assump:f_2}, \ref{Assump: tilde{f}}, \eqref{cond:g0}-\eqref{assump:g_5} and \eqref{assump:h_0} hold and $u_{0}\in W_{0}^{s,G_{x,y}}(\Omega)$. If $E(u_{0})=d$ and $I(u_{0})<0$,  then there exists a finite time $t_0>0$ such that $E(u(\cdot, t)) <d$ for all $t \geq t_0$ and the strong solution of the problem \eqref{main:problem} blows up in the sense of
$$\lim_{t\to T^{**}}\int_{0}^{t}\|u(\cdot,\tau)\|_{L^{2}(\Omega)}^{2}\,d\tau=+\infty \quad \text{where} \quad  T^{**}:=\frac{4\|u_0\|_{L^2(\Omega)}^2(\alpha-1)}{\alpha(\alpha-2)^{2}(d-E(u(\cdot, t_0)))}.$$
\end{theorem}
\begin{proof}
   Applying Lemma \ref{positive:depth} and Lemma \ref{continuity of E(u) with t}, we obtain $E(u_{0})=d>0$, and $E(u(\cdot,t))$ and $I(u(\cdot,t))$ are continuous with respect to $t$. Then, there exists a $t_{0}$ such that $E(u(\cdot,t))>0$ and $I(u(\cdot,t))< 0$ for $0<t \leq t_{0}.$ Using \eqref{u_t is nonzero}, we have $u_{t}(\cdot,t)\not\equiv 0$  for $0<t \leq t_{0}$. From \eqref{strong solution}
   $$ 0< E(u(\cdot,t_{0}))= d-\int_{0}^{t_{0}}\|u_t(\cdot,\tau)\|^{2}_{L^{2}(\Omega)}d\tau < d. $$ 
 Taking $t=t_{0}$ as the initial time we have $E(u(\cdot,t_{0}))<d$ and $I(u(\cdot,t_{0}))<0$ {\it i.e.} $u(\cdot,t_{0})\in V$. By Lemma \ref{u in W delta}(ii) we have $E(u(\cdot,t))<d$ and  $I(u(\cdot,t))<0$ for all $t\geq t_{0}$. The rest of the proof is the same as in Theorem \ref{Blow-up thm} for all $t\geq t_{0}$.
\end{proof}

\subsection{On the case of high initial energy: $E(u_{0})> d$}
   This subsection gives sufficient conditions for the global existence of strong solutions and blow-up in finite time regarding the high initial energy.  Before proving the main results, we introduce the following notations:\\
   Define
$$\mathcal{N}_{+}:=\{u\in W_{0}^{s,G_{x,y}}(\Omega)\,|\,I(u)> 0\},\qquad \mathcal{N}_{-}:=\{u\in W_{0}^{s,G_{x,y}}(\Omega)\,|\,I(u)< 0\}$$
and
$$O_{\zeta}:=\{u\in W_{0}^{s,G_{x,y}}(\Omega)\,|\,E(u)<\zeta\}.$$
By the definition of $E$, $\mathcal{N}$, $O_{\zeta}$ and $d,$ we get
$$\mathcal{N}_{\zeta}:=\mathcal{N}\cap O_{\zeta}= \{u\in\mathcal{N}\,|\,E(u)<\zeta\}\neq\emptyset \qquad \text{  for all    }\zeta >d.$$
For $\zeta>d$, define
$$\lambda_{\zeta}:=\inf \{\|u\|_{L^{2}(\Omega)}\,|\,u\in\mathcal{N}_{\zeta}\}, \qquad \Lambda_{\zeta}:=\sup \{\|u\|_{L^{2}(\Omega)}\,|\,u\in\mathcal{N}_{\zeta}\}.$$
Note that the map $\zeta\mapsto \lambda_{\zeta}$ is non-increasing and $\zeta\mapsto\Lambda_{\zeta}$ is non-decreasing.

\begin{lemma}
\label{N+ is bounded}
    If the conditions \eqref{assump:f_0}-\eqref{Assump:f_2} and \eqref{cond:g0}-\eqref{Cond:g_6} hold, then
 \begin{enumerate}
 \item[$(i)$] $0$ is away from both $\mathcal{N}$ and $\mathcal{N_{-}}$, {\it i.e.} $\dist(0,\mathcal{N})> 0$ and $\dist(0,\mathcal{N_{-}})> 0.$
 \item[$(ii)$] For any $\zeta> 0$, the set $O_{\zeta}\cap \mathcal{N_{+}}$ is bounded in $W_{0}^{s,G_{x,y}}(\Omega).$
\end{enumerate}
\end{lemma}
\begin{proof}
$(i)$    Let $u\in \mathcal{N}$. From Lemma \ref{F:bounds}(i) and Lemma \ref{RELATION:MODULAR & NORM}(i), we get
    \begin{align*}
        d\leq E(u)&={\iint_{Q}}G_{x,y}\left(D^{s}u\right)d\mu-\int_{\Omega}F(x,u)\,dx\\
        &\leq \max\{ [u]_{s,G_{x,y}}^{g^-},[u]_{s,G_{x,y}}^{g^+}\}+ A\max\{ \|u\|_{L^{\Phi}(\Omega)}^{h_{2}^{-}},\|u\|_{L^{\Phi}(\Omega)}^{h_{2}^{+}}\}.
    \end{align*}
    Since $W_{0}^{s,G_{x,y}}(\Omega)$ is embedded into $L^{\Phi}(\Omega)$, we further obtain
    \begin{equation}
        \label{dist(0,N)}
    \begin{aligned}
        \qquad \quad\qquad d &\leq \max\{ [u]_{s,G_{x,y}}^{g^-},[u]_{s,G_{x,y}}^{g^+}\}+ AC^{*}_{G}\max\{ [u]_{s,G_{x,y}}^{h_{2}^{-}},[u]_{s,G_{x,y}}^{h_{2}^{+}}\},
    \end{aligned}
     \end{equation}
 where $C^{\ast}_G$ is defined in \eqref{lower bound of norm u}. Now, if $ [u]_{s,G_{x,y}}\geq 1$, then clearly dist$(0,\mathcal{N})>0$. Otherwise, if $ [u]_{s,G_{x,y}} < 1$, then from \eqref{dist(0,N)} and condition \eqref{assump:g_5}, we find that
\begin{align*}
      [u]_{s,G_{x,y}}&\geq \left(\frac{d}{1+AC^{*}_{G}}\right)^{\frac{1}{g^{-}}}.
\end{align*}
This implies that there exists a constant $\rho >0$ such that
$$\text{dist}(0, \mathcal{N})=\inf_{u\in \mathcal{N}}[u]_{s,G_{x,y}}\geq\rho>0.$$
For $u\in \mathcal{N_{-}}$, we get $[u]_{s,G_{x,y}}\neq 0$. From \eqref{lower bound of norm u} we have
\begin{align}
\label{dist(0,N-)}
  g^{-} \min\{ [u]_{s,G_{x,y}}^{g^-},[u]_{s,G_{x,y}}^{g^+}\} 
 & <A h_{2}^{+} C^{*}_{G} \max\{ [u]_{s,G_{x,y}}^{h_{2}^{-}}, [u]_{s,G_{x,y}}^{h_{2}^{+}}\}.
 \end{align}
  

    
\noindent If $[u]_{s,G_{x,y}}\geq 1$, then clearly dist$(0,\mathcal{N}_{-})>0.$ Otherwise, if $[u]_{s,G_{x,y}} < 1$, then from \eqref{dist(0,N-)} and \eqref{assump:g_5} we get
   \begin{align*}
[u]_{s,G_{x,y}}& > \left(\frac{g^{-}}{A h_{2}^{+} C^{*}_{G}}\right)^{\frac{1}{h_{2}^{-}-g^{+}}}.
   \end{align*}
   This implies that there exists a constant $\delta >0$ such that
$$\text{dist}(0, \mathcal{N}_{-})=\inf_{u\in \mathcal{N}_{-}}[u]_{s,G_{x,y}}\geq\delta.$$
$(ii)$ If $u\in O_{\zeta}\cap \mathcal{N_{+}}$, then $E(u)<\zeta$ and $I(u)>0.$ Therefore, from condition \eqref{Assump:f_2}, \eqref{Assump:g_3}, \eqref{assump:g_5} and Lemma \ref{RELATION:MODULAR & NORM}(i) we get
\begin{align}
\label{E tau is bounded}
      \zeta> E(u)  
    &\geq {\iint_{Q}}G_{x,y}\left(D^{s}u\right)d\mu-\int_{\Omega}\frac{1}{h_{1}(x)}f(x,u)u\ dx \nonumber\\
     &\geq \left(\frac{1}{g^{+}}-\frac{1}{h_{1}^{-}}\right){\iint_{Q}}g_{x,y}\left(D^{s}u\right)\left(D^{s}u\right)d\mu+\frac{1}{h_{1}^{-}}I(u) \nonumber\\
     &\geq \left(\frac{1}{g^{+}}-\frac{1}{h_{1}^{-}}\right){\iint_{Q}}g_{x,y}\left(D^{s}u\right)\left(D^{s}u\right)d\mu \nonumber\\
     & \geq g^{-}\left(\frac{1}{g^{+}}-\frac{1}{h_{1}^{-}}\right)\min\{ [u]_{s,G_{x,y}}^{g^-},[u]_{s,G_{x,y}}^{g^+}\}.
\end{align}
If $[u]_{s,G_{x,y}} \leq 1$, then clearly  $O_{\zeta}\cap \mathcal{N_{+}}$ is bounded. Otherwise, if $[u]_{s,G_{x,y}} > 1$, then from \eqref{E tau is bounded} we have
\begin{align*}
    [u]_{s,G_{x,y}}&<\left(\frac{\zeta h_{1}^{-}g^{+}}{g^{-}(h_{1}^{-}-g^{+})} \right)^{\frac{1}{g^{-}}}.
\end{align*}
This implies that the set $O_{\zeta}\cap \mathcal{N_{+}}$ is bounded in $W_{0}^{s,G_{x,y}}(\Omega)$.
\end{proof}
\noindent Next we assume the following condition:
\begin{enumerate}[leftmargin=1.5cm,label=\textnormal{($\hat{g}_6$)},ref=\textnormal{$\hat{g}_{6}$}]
    \item \label{assump:h_1}
    $h_{2}^{+}\leq g^{-}\left(1+\frac{2s}{N}\right).$
\end{enumerate}
\begin{lemma}
    Let the conditions \eqref{cond:g0}-\eqref{Cond:g_6}, \eqref{assump:f_0}-\eqref{Assump:f_2}, and \eqref{assump:h_1} hold. Then, for any $\zeta > d$, the constants $\lambda_{\zeta}$ and $\Lambda_{\zeta}$ satisfy  
\[
0 < \lambda_{\zeta} \leq \Lambda_{\zeta} < +\infty.
\]
\end{lemma}

\begin{proof}
   If $u\in \mathcal{N}_{\zeta}$, then from \eqref{E tau is bounded} we have
    \begin{align}
        \zeta> E(u) 
 & \geq g^{-}\left(\frac{1}{g^{+}}-\frac{1}{h_{1}^{-}}\right)\min\{ [u]_{s,G_{x,y}}^{g^-},[u]_{s,G_{x,y}}^{g^+}\}\nonumber.
\end{align}
This further gives
\begin{align*}
[u]_{s,G_{x,y}} &< \max \left\{ \left(\frac{\zeta h_{1}^{-}g^{+}}{g^{-}(h_{1}^{-}-g^{+})} \right)^{\frac{1}{g^{-}}}, \left(\frac{\zeta h_{1}^{-}g^{+}}{g^{-}(h_{1}^{-}-g^{+})} \right)^{\frac{1}{g^{+}}} \right\}=:\delta_{\max}(\zeta).
\end{align*}
Since $W_{0}^{s,G_{x,y}}(\Omega)$ is embedded in $L^{2}(\Omega)$, we have
$$\|u\|_{L^{2}(\Omega)}\leq C_{*}[u]_{s,G_{x,y}}\leq C_{*}\delta_{\max}(\zeta). $$
This implies that
$$\Lambda_{\zeta}\leq C_{*}\delta_{\max}(\zeta) < +\infty. $$
Now we will prove that $\lambda_{\zeta}> 0.$ Let $u\in \mathcal{N}_{\zeta}$, then
 from \eqref{lower bound of norm u} and Theorem \ref{thm2} we get
\begin{align}
\label{LH NORM}
 g^{-} \min\{ [u]_{s,G_{x,y}}^{g^-},[u]_{s,G_{x,y}}^{g^+}\} \leq g^{-}{\iint_{Q}}G_{x,y}\left(D^{s}u\right)d\mu \nonumber
 &\leq A h_{2}^{+} \max\left\{ \|u\|_{L^{\Phi}{(\Omega)}}^{{h_{2}}^{+}}, \|u\|_{L^{\Phi}{(\Omega)}}^{{h_{2}}^{-}}\right\}\\\nonumber
&\leq  A h_{2}^{+} \max\left\{ (C_h)^{h_{2}^{+}}\|u\|_{L^{h_{2}^{+}}{(\Omega)}}^{{h_{2}}^{+}}, (C_h)^{h_{2}^{-}}\|u\|_{L^{h_{2}^{+}}{(\Omega)}}^{{h_{2}}^{-}}\right\}\\
&\leq  A h_{2}^{+}C_{h,\max} \max\left\{ \|u\|_{L^{h_{2}^{+}}(\Omega)}^{{h_{2}}^{+}},\|u\|_{L^{h_{2}^{+}}(\Omega)}^{{h_{2}}^{-}}\right\},
\end{align}
where \[C_{h,\max}=\max\{ (C_h)^{h_{2}^{+}}, (C_h)^{h_{2}^{-}} \}.\]
 We can further divide \eqref{LH NORM} into two cases:\\
\textbf{Case: 1} $\|u\|_{L^{h_{2}^{+}}(\Omega)}\geq 1$.\\
From \eqref{LH NORM}, we have
\begin{align}
g^{-} \min\{ [u]_{s,G_{x,y}}^{g^-},[u]_{s,G_{x,y}}^{g^+}\} &\leq A h_{2}^{+}C_{h,\max} \|u\|_{L^{h_{2}^{+}}(\Omega)}^{{h_{2}}^{+}}.\nonumber
\end{align}
Therefore, by using \eqref{assump:g_5} -\eqref{Cond:g_6}, interpolation inequality \cite[Remark 2, p. 93]{Brezis} and the embedding of $W_{0}^{s,G_{x,y}}(\Omega)$ into $L^{g_{*,s}^{-}}(\Omega)$ we obtain
\begin{align}
    g^{-} \min\{ [u]_{s,G_{x,y}}^{g^-},[u]_{s,G_{x,y}}^{g^+}\}&\leq A h_{2}^{+}C_{h,\max} \|u\|^{\alpha h_{2}^{+}}_{L^{g^{-}_{\ast,s}}(\Omega)}\|u\|^{(1-\alpha)h_{2}^{+}}_{L^{2}(\Omega)}\nonumber\\ \nonumber
    &\leq A C_{1}^{*} h_{2}^{+}C_{h,\max} [u]^{\alpha h_{2}^{+}}_{s,G_{x,y}}\|u\|^{(1-\alpha)h_{2}^{+}}_{L^{2}(\Omega)},
\end{align}
where $\alpha = \frac{g_{*,s}^{-}(h_{2}^{+}-2)}{h_{2}^{+}(g_{*,s}^{-}-2)} \in (0,1)$ and $C_{1}^{*}$ is the embedding constant. The above equation can be rewritten as
\begin{equation}
\label{Lamdba>0}
\begin{aligned}
\|u\|^{(1-\alpha)h_{2}^{+}}_{L^{2}(\Omega)} &\geq\frac{g^{-}}{ A C_{1}^{*} h_{2}^{+}C_{h,\max}}\min\{ [u]_{s,G_{x,y}}^{g^{-}-\alpha h_{2}^{+}},[u]_{s,G_{x,y}}^{g^{+}-\alpha h_{2}^{+}}\}.
\end{aligned}
\end{equation}

Note that in view of \eqref{assump:h_1}, we have
\begin{align*}
\label{g-alpha is positive}
     g^{+}-\alpha h_{2}^{+} > g^{-}-\alpha h_{2}^{+} = g^{-}-\frac{g_{*,s}^{-}(h_{2}^{+}-2)}{g_{*,s}^{-}-2} =\frac{{g_{*,s}^{-}}}{g_{*,s}^{-}-2}\left(g^{-}\left(1+\frac{2s}{N}\right)-h_{2}^{+}\right) >0.
 \end{align*}

By Lemma \ref{N+ is bounded} and the embedding of $W_{0}^{s,G_{x,y}}(\Omega)$ into $L^{2}(\Omega)$, it follows that the right-hand side of \eqref{Lamdba>0} is bounded and strictly positive. Therefore, by the definition of $\lambda_{\zeta}$, we conclude that $\lambda_{\zeta} > 0$.
\\
\textbf{Case: 2} $\|u\|_{L^{h_{2}^{+}}(\Omega)} < 1$.\\
From \eqref{LH NORM}, we have
\begin{align}
g^{-} \min\{ [u]_{s,G_{x,y}}^{g^-},[u]_{s,G_{x,y}}^{g^+}\} &\leq A h_{2}^{+}C_{h,\max} \|u\|_{L^{h_{2}^{+}}(\Omega)}^{{h_{2}}^{-}}.\nonumber
\end{align}
Now, by following the same argument as in \textbf{case 1}, we conclude that $\lambda_{\zeta}>0$.
\end{proof}
To state the main results concerning the case of high initial energy (\(E(u_{0}) > d\)), we introduce the following sets:  
\[
\mathcal{B} = \left\{ u_{0} \in W_{0}^{s,G_{x,y}}(\Omega) \mid \text{ the strong solution } u \text{ to problem } \eqref{main:problem} \text{ blows up (in the } L^2\text{-norm) in finite time} \right\},
\]
\[
\mathcal{G}_{0} = \left\{ u_{0} \in W_{0}^{s,G_{x,y}}(\Omega) \mid \text{ the strong solution } u \text{ to problem } \eqref{main:problem} \text{ satisfies } u(\cdot,t) \to 0 \text{ in } W_{0}^{s,G_{x,y}}(\Omega) \text{ as } t \to \infty \right\}.
\]
We define the $\omega$-limit set $\omega(u_{0})$ of the initial data $u_{0}\in W_{0}^{s,G_{x,y}}(\Omega)$ by
\[\omega(u_{0})=\bigcap_{\ell \geq 0} \overline{\left\{u(\cdot,t)\,|\, t \geq \ell \right\}}^{W_{0}^{s,G_{x,y}}(\Omega)}.\]
\begin{theorem}\label{mainthm3}
\label{high initial energy main thm}
    Let the conditions \eqref{assump:f_0}-\eqref{Assump:f_3}, \eqref{cond:g0}-\eqref{Cond:g_6} and \eqref{assump:h_1} hold and $u_{0}\in W_{0}^{s,G_{x,y}}(\Omega).$ If $E(u_{0})>d$, then the following statements hold:
    \begin{enumerate}
    \item[$(i)$] If $u_{0}\in \mathcal{N}_{+}$ and $\|u_{0}\|_{L^{2}(\Omega)}\leq \lambda_{E(u_{0})}$, then $u_{0} \in \mathcal{G}_{0}.$
    \item[$(ii)$] If $u_{0}\in \mathcal{N}_{-}$ and $\|u_{0}\|_{L^{2}(\Omega)}\geq \Lambda_{E(u_{0})}$, then $u_{0} \in \mathcal{B}.$
    \end{enumerate}
\end{theorem}
\begin{proof}
$(i)$ If \( u_{0} \in \mathcal{N}_{+} \) and \( \|u_{0}\|_{L^{2}(\Omega)} \leq \lambda_{E(u_{0})} \), then we claim that \( u(\cdot,t) \in \mathcal{N}_{+} \) for all \( t \in [0,T_{\max}) \). If \( u(\cdot,t) \notin \mathcal{N}_{+} \) for some \( t \in [0,T_{\max}) \), then, by Lemma \ref{continuity of E(u) with t}, there exists \( t_{0} \in [0,T_{\max}) \) such that \( u(\cdot,t) \in \mathcal{N}_{+} \) for \( 0 \leq t < t_{0} \) and \( u(\cdot,t_{0}) \in \mathcal{N} \). Therefore, from \eqref{u_t is nonzero}, we have \( u_{t}(\cdot,t) \not\equiv 0 \) for \( t \in (0,t_{0}) \). It follows, due to \eqref{strong solution}, that \( E(u(\cdot,t_{0})) < E(u_{0}) \) and \( u(\cdot,t_{0}) \in \mathcal{N}_{E(u_{0})} \). Thus, by the definition of \( \lambda_{E(u_{0})} \), we obtain
\begin{equation}
\label{lambda L2 norm}
\|u(\cdot,t_{0})\|_{L^{2}(\Omega)} \geq \lambda_{E(u_{0})}.
\end{equation}
Noticing that \( I(u(\cdot,t)) > 0 \) for \( t \in [0,t_{0}) \) and using \eqref{u_t is nonzero}, we deduce that \( \|u(\cdot,t)\|_{L^{2}(\Omega)} \) is decreasing for \( 0 \leq t < t_{0} \). Therefore, we have
\[
\|u(\cdot,t_{0})\|_{L^{2}(\Omega)} = \lim_{t \to t_{0}^-} \|u(\cdot,t)\|_{L^{2}(\Omega)} < \|u(\cdot,0)\|_{L^{2}(\Omega)} \leq \lambda_{E(u_{0})},
\]
which contradicts \eqref{lambda L2 norm}. Therefore, \( u(\cdot,t) \in \mathcal{N}_{+} \) for all \( t \in [0,T_{\max}) \).

Next, from \eqref{u_t is nonzero} and \eqref{strong solution}, we get \( u(\cdot,t) \in O_{E(u_{0})} \) for all \( t \in [0,T_{\max}) \). By Lemma \ref{N+ is bounded}, we conclude that \( u(\cdot,t) \) is bounded in \( W_{0}^{s,G_{x,y}}(\Omega) \) for all \( t \in [0,T_{\max}) \).

Now, since \( E(u(\cdot,t)) < E(u_{0}) \) and \( I(u(\cdot,t)) > 0 \) for all \( t \in [0,T_{\max}) \), and following the same argument as in \textbf{Claim 6} of Theorem \ref{thm:global-existence-weak}, we can prove that \( T_{\max} = +\infty \). This implies that \( u \in \mathcal{N}_{+} \cap O_{E(u_{0})} \) for all \( 0 \leq t < \infty \). Since \( I(u(\cdot,t)) > 0 \) for all \( 0 \leq t < \infty \), it follows from \eqref{u_t is nonzero} that \( \|u(\cdot,t)\|_{L^{2}(\Omega)} \) is decreasing for \( 0 \leq t < \infty \) and \( E(u(\cdot,t)) < E(u_{0}) \) for all \( 0 \leq t < \infty \).

Therefore, for any \( w \in \omega(u_{0}) \), using the compact embedding of \( W_{0}^{s,G_{x,y}}(\Omega) \) into \( L^{2}(\Omega) \) and the lower semicontinuity of the modular function \cite[Lemma 3.1.4]{M32}, we obtain
\begin{align}
\label{w limit}
\|w\|_{L^{2}(\Omega)} = \lim_{t \to \infty} \|u(\cdot,t)\|_{L^{2}(\Omega)} < \lambda_{E(u_{0})}, \  \ \mbox{and} \ 
E(w) \leq \liminf_{t \to \infty} E(u(\cdot,t)) < E(u_{0}).
\end{align}
From \eqref{w limit} and the definition of \( \lambda_{E(u_{0})} \), we obtain \( \omega(u_{0}) \cap \mathcal{N} = \emptyset \). Since \( w \) is the solution of the stationary counterpart of the problem \eqref{main:problem}, we conclude that \( \omega(u_{0}) = \{0\} \), i.e., \( u_{0} \in \mathcal{G}_{0} \).\\
$(ii)$ If \( u_{0} \in \mathcal{N}_{-} \) and \( \|u_{0}\|_{L^{2}(\Omega)} \geq \Lambda_{E(u_{0})} \), then we claim that \( u(\cdot,t) \in \mathcal{N}_{-} \) for all \( t \in [0,T_{\max}) \). If \( u(\cdot,t) \notin \mathcal{N}_{-} \) for some \( t \in [0,T_{\max}) \), then, by Lemma \ref{continuity of E(u) with t}, there exists \( t_{1} \in (0,T_{\max}) \) such that \( u(\cdot,t) \in \mathcal{N}_{-} \) for \( 0 \leq t < t_{1} \) and \( u(\cdot,t_{1}) \in \mathcal{N} \). 

As in the previous case, we obtain \( E(u(\cdot,t_{1})) < E(u_{0}) \), which implies \( u(\cdot,t_{1}) \in O_{E(u_{0})} \). Furthermore, \( u(\cdot,t_{1}) \in \mathcal{N}_{E(u_{0})} \). By the definition of \( \Lambda_{E(u_{0})} \), we have
\begin{equation}
\label{wedge L2 norm}
\|u(\cdot,t_{1})\|_{L^{2}(\Omega)} \leq \Lambda_{E(u_{0})}.
\end{equation}

Since \( I(u(\cdot,t)) < 0 \) for all \( t \in [0,t_{1}) \), it follows from \eqref{u_t is nonzero} that \( \|u(\cdot,t)\|_{L^{2}(\Omega)} \) is increasing for \( 0 \leq t < t_{1} \). Consequently, we obtain
\[
\|u(\cdot,t_{1})\|_{L^{2}(\Omega)} > \|u(\cdot,0)\|_{L^{2}(\Omega)} \geq \Lambda_{E(u_{0})},
\]
which contradicts \eqref{wedge L2 norm}.

Suppose \( T_{\max} = \infty \). This implies that \( u \in \mathcal{N}_{-} \cap O_{E(u_{0})} \) for all \( 0 \leq t < \infty \). Since \( I(u(\cdot,t)) > 0 \) for all \( 0 \leq t < \infty \), it follows from \eqref{u_t is nonzero} that \( \|u(\cdot,t)\|_{L^{2}(\Omega)} \) is strictly increasing for \( 0 \leq t < \infty \), and \( E(u(\cdot,t)) < E(u_{0}) \) for all \( 0 \leq t < \infty \). Therefore, for any \( w \in \omega(u_{0}) \), using the compact embedding of \( W_{0}^{s,G_{x,y}}(\Omega) \) into \( L^{2}(\Omega) \) and the lower semicontinuity of the modular function \cite[Lemma 3.1.4]{M32}, we obtain
\begin{align}
\label{w'' limit}
\|w\|_{L^{2}(\Omega)} = \lim_{t \to \infty} \|u(\cdot,t)\|_{L^{2}(\Omega)} > \Lambda_{E(u_{0})}, \ \ \mbox{and} \ 
E(w) \leq \liminf_{t \to \infty} E(u(\cdot,t)) < E(u_{0}).
\end{align}

From \eqref{w'' limit} and the definition of \( \Lambda_{E(u_{0})} \), we deduce that \( \omega(u_{0}) \cap \mathcal{N} = \emptyset \). However, since \( \text{dist}(0, \mathcal{N}_{-}) > 0 \) by Lemma \ref{N+ is bounded}, we also have \( 0 \notin \omega(u_{0}) \), and \( w \) is the solution of the stationary counterpart of the problem \eqref{main:problem}. Therefore, \( \omega(u_{0}) = \emptyset \), which contradicts \( T_{\max} = \infty \). 

Thus, \( u_{0} \in \mathcal{B} \).
\end{proof}
Next, we show that the set $\{u_0 \in \mathcal{N}_- : \|u_0\|_{L^2(\Omega)} \geq \Lambda_{E(u_0)}\}$ is nonempty for all high initial energy data $u_0$ satisfying $E(u_0)>d.$
\begin{corollary}
\label{high initial blow up}
    Let conditions \eqref{assump:f_0}-\eqref{Assump:f_3}, \eqref{cond:g0}-\eqref{Cond:g_6} and \eqref{assump:h_1} hold and $u_{0}\in W_{0}^{s,G_{x,y}}(\Omega)$. If
    \begin{align}
        \label{E(u0)<d}
         &d<E(u_{0})< g^{-}\left(\frac{h_{1}^{-}-g^{+}}{h_{1}^{-}g^{+}}\right)C_{*,max}\min\left\{ \|u_{0}\|_{L^{2}(\Omega)}^{g^{-}},\|u_{0}\|_{L^{2}(\Omega)}^{g^{+}} \right\},
    \end{align}
    then $u_{0}\in \mathcal{N}_{-}$ and $\|u_{0}\|_{L^{2}(\Omega)}\geq \Lambda_{E(u_{0})}$, where  $C_{*,max}= \max\{ C_{*}^{-g^{-}},C_{*}^{-g^{+}}\}$. 
\end{corollary}
\begin{proof}
From the conditions \eqref{Assump:f_2}, \eqref{Assump:g_3}, \eqref{assump:g_5}, Lemma \eqref{RELATION:MODULAR & NORM}(i), and the embedding of \( W_{0}^{s,G_{x,y}}(\Omega) \) into \( L^{2}(\Omega) \), we obtain the following inequality:
\begin{align}
\label{NE(U0)}
    E(u_0) &\geq \frac{1}{g^+} \iint_{Q} g_{x,y} \left( D^s u_0 \right) \left( D^s u_0 \right) d\mu - \frac{1}{h_1^-} \int_{\Omega} f(x, u_0) u_0 \, dx \nonumber \\
    &\geq \frac{h_1^- - g^+}{h_1^- g^+} \iint_{Q} g_{x,y} \left( D^s u_0 \right) \left( D^s u_0 \right) d\mu + \frac{1}{h_1^-} I(u_0) \nonumber \\
    &\geq g^{-} \left( \frac{h_1^- - g^+}{h_1^- g^+} \right) \iint_{Q} G_{x,y} \left( D^s u_0 \right) d\mu + \frac{1}{h_1^-} I(u_0) \nonumber \\
    &\geq g^{-} \left( \frac{h_1^- - g^+}{h_1^- g^+} \right) C_{*,\max} \min \left\{ [u_0]_{s,G_{x,y}}^{g^-}, [u_0]_{s,G_{x,y}}^{g^+} \right\} + \frac{1}{h_1^-} I(u_0).
\end{align}

Then, from \eqref{E(u0)<d}, we have:
$$
E(u_0) > E(u_0) + \frac{1}{h_1^-} I(u_0),
$$
which implies that \( I(u_0) < 0 \). Therefore, we conclude that \( u_0 \in \mathcal{N}_- \).


Next, let \( u \in \mathcal{N}_{E(u_0)} \). Following the same reasoning as above, and using \eqref{NE(U0)} by replacing \( u_0 \) with \( u \), we obtain the following energy estimate:
\begin{equation}
\label{E(u0)}
\begin{aligned}
    g^{-} \left( \frac{h_1^- - g^+}{h_1^- g^+} \right) C_{*,\max} \min \left\{ \|u\|_{L^2(\Omega)}^{g^-}, \|u\|_{L^2(\Omega)}^{g^+} \right\} &\leq E(u) < E(u_0).
\end{aligned}
\end{equation}

We now divide this into two cases:

\textbf{Case 1: \( \|u\|_{L^2(\Omega)} \leq 1 \).}

In this case, using the estimates from \eqref{E(u0)<d} and \eqref{E(u0)}, we obtain:
\begin{align*}
    \|u\|_{L^2(\Omega)}^{g^+} &< \min \left\{ \|u_0\|_{L^2(\Omega)}^{g^-}, \|u_0\|_{L^2(\Omega)}^{g^+} \right\}.
\end{align*}
Taking the supremum over the set \( \mathcal{N}_{E(u_0)} \), we get:
$$
\Lambda_{E(u_0)}^{g^+} < \min \left\{ \|u_0\|_{L^2(\Omega)}^{g^-}, \|u_0\|_{L^2(\Omega)}^{g^+} \right\}.
$$
This implies:
$$
\|u_0\|_{L^2(\Omega)} > \Lambda_{E(u_0)}.
$$
\textbf{Case 2: \( \|u\|_{L^2(\Omega)} > 1 \).}

The same reasoning and calculations apply in this case as well, yielding the same conclusion that:
$$
\|u_0\|_{L^2(\Omega)} > \Lambda_{E(u_0)}.
$$
\end{proof}
\begin{theorem}  
\label{finthm}
    Let conditions \eqref{assump:f_0}-\eqref{Assump:f_3}, \eqref{cond:g0}-\eqref{Cond:g_6} hold. For any $M> d$, then there exists $u_{M}\in \mathcal{N}_{-}$ such that $E(u_{M})=M$ and $\|u_{M}\|_{L^{2}(\Omega)}\geq \Lambda_{E(u_{M})}$.
    \end{theorem}
    \begin{proof}
        Assume that \( M > d \) and \( \Omega_1, \Omega_2 \) are two arbitrary disjoint open subdomains of \( \Omega \). Denote
\[
Q_i = \left( \mathbb{R}^N \times \mathbb{R}^N \right) \setminus \left( \mathcal{C} \Omega_i \times \mathcal{C} \Omega_i \right), \quad \mathcal{C} \Omega_i = \mathbb{R}^N \setminus \Omega_i, \quad i = 1, 2.
\]
Define
\[
W_0^{s, G_{x,y}}(\Omega_i) = \left\{ u : u \in L^{\widehat{G}_x}(\Omega_i), \, u = 0 \text{ in } \mathcal{C} \Omega_i, \, \frac{u(x) - u(y)}{|x - y|^s} \in L^{\widehat{G}_x}(Q_i) \right\}, \quad i = 1, 2.
\]
Furthermore, assume that \( v \in W_0^{s, G_{x,y}}(\Omega_1) \) is an arbitrary nonzero function. Then, choose \( \zeta \) large enough such that
\[
E(\zeta v) = \iint_{Q} G_{x,y} \left( \zeta D^s u \right) d\mu - \int_\Omega F(x, \zeta u) \, dx \leq 0,
\]
and
\[
\min \left\{ \|\zeta v\|_{L^2(\Omega)}^{g^-}, \|\zeta v\|_{L^2(\Omega)}^{g^+} \right\} > \frac{h_1^- g^+}{g^- C_{*,\max} (h_1^- - g^+)} M.
\]
We fix such a value of \( \zeta \) and choose a function \( \mu \in W_0^{s, G_{x,y}}(\Omega_2) \) such that \( M = E(\mu + \zeta v) = E(u_M) \) (where \( u_M = \mu + \zeta v \)). It follows that
\[
\min \left\{ \|u_M\|_{L^2(\Omega)}^{g^-}, \|u_M\|_{L^2(\Omega)}^{g^+} \right\} \geq \min \left\{ \|\zeta v\|_{L^2(\Omega)}^{g^-}, \|\zeta v\|_{L^2(\Omega)}^{g^+} \right\} > \frac{h_1^- g^+}{g^- C_{*, \max} (h_1^- - g^+)} E(u_M).
\]
By Corollary \ref{high initial blow up}, we conclude that \( u_M \in \mathcal{N}_- \) and \( \|u_M\|_{L^2(\Omega)} \geq \Lambda_{E(u_M)} \).
\end{proof}

\section{Classes of Problems, Final Remarks, and Open Questions}\label{final-cooments}

In this section, we present various examples of the function \( g \) that generalize existing results and models. Additionally, we introduce new examples that, to the best of our knowledge, have not been previously studied. We also provide concluding remarks on our work and highlight open and challenging questions.

\subsection{Classes of Problems}

In this subsection, we present examples of functions \( g \) that extend known results and introduce cases that have not been explored before.

\medskip

$\bullet$ \textbf{The fractional $p$-Laplacian.} When \( g_{x,y}(t) = |t|^{p-2}t \) with \( p > \max\left\{\frac{2N}{N+2s},1\right\} \), equation \eqref{main:problem} reduces to
\begin{equation*} 
        \left\{
    \begin{aligned}
      u_{t} + (-\Delta)_{p}^{s} u &=  f(x,u), &&
 \text{in } \Omega \times (0, \infty),\\
 u &= 0, && \text{in}~  \mathbb{R}^N \setminus \Omega \times (0, \infty),\\
 u(x,0) &= u_0(x), && \text{in}~ \Omega,
\end{aligned}
    \right.
 \end{equation*}
where \( (-\Delta)_{p}^{s} \) is the well-known fractional \( p \)-Laplacian. Our results extend those of \cite{Liao-Liu-Ye-2020} by establishing the existence of local strong solutions without assuming condition \eqref{Assump:f_3}; see Theorem \ref{loc} and finite time blow up of strong solution; see Theorems \ref{Blow-up thm} and \ref{Main thm3}.

\medskip

$\bullet$ \textbf{The fractional $p(x,\cdot)$-Laplacian with variable exponent.} Let 
\[
g_{x,y}(t) = \frac{1}{p(x,y)}|t|^{p(x,y)-2}t, 
\]
where \( p \) is a symmetric, continuous function satisfying
\[
1 < p^- = \min_{(x,y) \in \overline{\Omega} \times \overline{\Omega}} p(x,y) \leq p(x,y) < p^+ = \max_{(x,y) \in \overline{\Omega} \times \overline{\Omega}} p(x,y) < +\infty.
\]
Then, equation \eqref{main:problem} becomes
\begin{equation*} 
    \left\{
    \begin{aligned}
      u_{t} + (-\Delta)_{p(x,\cdot)}^{s} u &=  f(x,u), && \text{in } \Omega \times (0, \infty),\\
      u &= 0, && \text{in } \left(\mathbb{R}^N \setminus \Omega\right) \times (0, \infty),\\
      u(x,0) &= u_0(x), && \text{in } \Omega.
    \end{aligned}
    \right.
\end{equation*}
This operator is relatively new in the literature, with only one known study with only one known study on evolution equations \cite{boudd}, which considers a specific form of \( f(x,t)=|t|^{q(x)-2}t \) and focuses on local solutions with low initial energy. In contrast, our work examines a more general nonlinearity and provides a comprehensive analysis, including the critical and high initial energy cases.

\medskip

$\bullet$ \textbf{The fractional Orlicz $g$-Laplacian.} If \( g_{x,y}(t) = g(t) \), then the fractional Musielak \( g_{x,y} \)-Laplacian reduces to the fractional Orlicz \( g \)-Laplacian. The study of this operator began in \cite{Bonder-Salort-2019}. To our knowledge, no work has addressed the parabolic equation involving the fractional Orlicz Laplacian, making our results novel in this direction.

\medskip

$\bullet$ \textbf{The fractional double phase operator.} For \( 1<p<q<N \) and a non-negative symmetric function \( a\in L^\infty(\Omega\times \Omega) \), we define
\[
g_{x,y}(t)=|t|^{p-2}t+ a(x,y) |t|^{q-2}t.
\]
This leads to the operator  
\[
(-\Delta)^s_{\Phi_{x,y}}  u := (-\Delta)_p^s u +  (-\Delta)_{q,a}^s u,
\]
which is associated with the energy functional  
\[
J_{s,G}(u)=\int_{\Omega\times\Omega} \left( \frac{|D_su|^p}{p} + a(x,y)\frac{|D_su|^q}{q} \right) \mathrm{d}\mu, \quad u\in W_0^{s,G_{x,y}}(\Omega).
\]

The double phase operator has gained significant attention due to its applications in mathematical physics, particularly in composite materials, fractional quantum mechanics, fractional superdiffusion, and modified electromagnetic models; see \cite{Ambrosio-Radulescu1, Zhang-Tang-Radulescu}.

To the best of our knowledge, only one work \cite{Aberqi} has studied the evolution equation involving the fractional double phase operator, focusing on global solutions for the low initial energy case. Our work is the first to provide a comprehensive study, including cases of higher initial energy. 

The study of the double phase operator and the associated function space was recently advanced by Crespo–Blanco, Gasinski, Harjulehto, and Winkert \cite{Crespo-Blanco-Gasinski-Harjulehto-Winkert-2022}, and Arora and Shmarev \cite{Arora-Shmarev-2023, Arora-Shmarev-2023-2}. Specifically, the authors investigated a quasilinear elliptic and parabolic equations involving a double phase operator with variable exponents, given by:

\begin{align}\label{I4}
A(u) = -\operatorname{div}\left(|\nabla u|^{p(x)-2} \nabla u + \mu(x)|\nabla u|^{q(x)-2} \nabla u\right),
\end{align}
where the function governing the growth conditions is given by 
\[
g_{x,y}(t) = |t|^{p(x)-2}t +\mu(x) |t|^{q(x)-2}t.\]

In this context, the operator \((-\Delta)_{g_{x,y}}^{s}\) can be regarded as the fractional counterpart of the newly introduced double phase operator. Consequently, our work constitutes the first study on the corresponding evolution equations, providing a significant extension of the existing framework.

\subsection{Concluding Remarks and Open Problems}

In this subsection, we summarize our main findings and propose some interesting open problems for future research.\\

$\bullet$ In our paper, we studied a more general case, where the fractional Musielak space encompasses all existing and well-treated nonlocal function spaces. Moreover, we provided a comprehensive study by addressing local, global, and strong solutions in the cases of low initial energy, critical initial energy, and high initial energy. 

The main challenges of our problem stem from its nonlocality and nonhomogeneity, which necessitated the development of new technical analysis methods. These techniques may prove useful for other problems in the field.\\

$\bullet$ The proof of the main results is strongly based on assumption \eqref{Assump:g_3}, namely the $\Delta_2$ condition. This condition plays a crucial role in ensuring the reflexivity of the associated space and in establishing all the necessary technical lemmas. In \cite{Chlebicka-Gwiazda-Goldstein-2018}, the authors studied the parabolic equations involving local Musielak-Orlicz Laplacian without assuming the $\Delta_2$ condition. Their proof relies on truncation techniques, the Young measures method, and monotonicity arguments. Therefore, it is essential to investigate whether the results obtained in \cite{Chlebicka-Gwiazda-Goldstein-2018} can be extended to our setting, particularly to equation \eqref{main:problem}.\\

$\bullet$ Another important question is whether our results can recover the local Musielak-Orlicz Laplacian by taking the limit as \( s 	\to 1 \) in the nonlocal problem \eqref{main:problem}. This type of problem was studied in \cite{rossi} for the fractional \( p \)-Laplacian. Therefore, it is essential to investigate whether this approach can be extended to our setting.


\end{document}